\documentclass[10pt,a4paper,nouppercase,nowrite]{article}
\usepackage[T1]{fontenc}
\usepackage[utf8]{inputenc}
\usepackage[english]{babel}
\usepackage{amsmath}
\numberwithin{equation}{section}
\usepackage{amssymb}
\usepackage{amsthm}
\usepackage{calc}
\usepackage{accents}
\usepackage{bbm}
\usepackage{authblk}
\usepackage{marginnote}

\usepackage{pgfplots}
\pgfplotsset{compat=1.18}

\newcommand{\di}{{\rm d}}
\newcommand{\Sp}{{\textsc (S)} }

\newcommand{\eps}{\varepsilon}
\newcommand{\F}{\widetilde{\mathcal{F}}}
\newcommand{\R}{\mathbb{R}}
\newcommand{\T}{\mathbb{T}}
\newcommand{\triplenorm}[1]{%
	\left\lvert\!\left\lvert\!\left\lvert #1
	\right\rvert\!\right\rvert\!\right\rvert}

\newtheorem{defn}{Definition}[section]
\newtheorem{lem}[defn]{Lemma}
\newtheorem{prop}[defn]{Proposition}
\newtheorem{thm}[defn]{Theorem}

\theoremstyle{definition}

\newtheorem{exmpl}[defn]{Example}
\newtheorem{rmk}[defn]{Remark}
\usepackage{palatino}
\usepackage{mathrsfs}
\usepackage{fancyhdr}
\usepackage[autostyle, italian=guillemets]{csquotes}
\usepackage{guit} 
\usepackage{frontespizio}
\usepackage{graphicx}
\usepackage[left=2.75cm,right=2.75cm,top=3.2cm,bottom=3.2cm]{geometry}
\usepackage{filecontents}
\usepackage{hyperref}
\expandafter\def\expandafter\normalsize\expandafter{%
	\normalsize%
	\setlength\abovedisplayskip{5.5pt}%
	\setlength\belowdisplayskip{5.5pt}%
	\setlength\abovedisplayshortskip{2pt}%
	\setlength\belowdisplayshortskip{2pt}%
}

	\title{The local turnpike property in Mean Field Control and Games with quadratic Hamiltonian}
	\author{Marco Cirant and Nicol\`o De Bernardi}

	\date{}
	
\begin{document}

	\maketitle
	
	\begin{abstract} 
		We study the local stability properties of solutions to ergodic and discounted mean field games systems, as the time horizon $T \to +\infty$, around stationary equilibria, when the Hamiltonian is quadratic. We replace the usual monotonicity of the coupling term with a weaker, local assumption on the stationary equilibrium (that need not be unique), stemming from a second-order strict positivity condition. This new stability assumption, together with a symmetry property of the system, allows us to derive an exponential turnpike property for those solutions that are close to the stationary one, whenever the spatial domain $\Omega$ is either the flat torus $\T^n$ or $\R^n$. Finally, through a fixed-point argument, we establish the actual existence of stable solutions, both on the finite horizon $[0,T]$ and on the infinite horizon, in the periodic setting $\Omega=\T^n$, provided that the initial (and terminal) data are close enough to the stationary equilibrium.
         \end{abstract}
         
	\tableofcontents
		
		\section{Introduction}
	This paper is devoted to the study of the long time behavior of solutions to discounted Mean Field Games systems of PDEs with quadratic Hamiltonian, of the following form:
	\begin{equation}
		\begin{cases}
			-\partial_t u - \Delta u + \frac{|Du|^2}{2} - f(x,m) + \delta u = 0,& \text{ on } (0,T) \times \Omega \\
			\partial_t m - \Delta m - \text{div}(m Du)=0,& \text{ on } (0,T) \times \Omega \\
			m(0,x)=m_0(x), \qquad u(T,x)=u_T(x),& \text{ on } \Omega,
		\end{cases}
		\label{dynsys}
	\end{equation}
	where $T > 0$ is the time-horizon, $f : \Omega \times \R \to \R$ is of class $C^2$ and has bounded first and second derivatives with respect to the second component, $\delta \geq 0$ is the discount factor. Here, $\Omega$ can be either the Euclidean space $\R^n$ or the flat torus $\T^n$ (in which case, data are periodic in the $x$ variable); $m_0, u_T$ are the initial distribution and the terminal cost, respectively, and they are assumed to be of class $C^{2,\alpha}$.
	
	This system of PDEs has been introduced in \cite{LLcras} to describe \textit{Mean Field Nash equilibria} among a population of players with density $m(t,x)$, where a typical player controls his own state
	\[
\di X_s = \alpha_s \di s + \sqrt{2} \, \di B_s \qquad\qquad \text{in }\Omega,
\]
with the aim of minimizing the cost
\[
J(\alpha) =  \mathbb{E} \int_{0}^T e^{-\delta s}\left(\frac12|\alpha_s|^2 + f\big(X_s, m(s,X_s) \big) \right) \di s + e^{-\delta T}u_T(X_T).
\]
	Within this framework, $u$ is the so-called value function of the typical player. Moreover, the same PDE system constitutes the optimality conditions for the \textit{Mean Field optimal control problem} of the functional (see again \cite{LLcras})
	\[
	\mathcal F(w) = \int_0^T \int_{\Omega} e^{-\delta s}\left(\frac{1}2m(s,y) |w(s,y)|^2 + F\big(y, m(s,y) \big) \right) \di y \di s + e^{-\delta T}\int_{\Omega} u_T(y) m(T,y)\di y
	\]
	subject to
	\[
	\partial_t m - \Delta m + \text{div}(m w)=0,\quad  \text{ on } (0,T) \times \Omega \qquad m(0,x)=m_0(x)
	\]
	where $F(x,m) = \int_0^m f(x,\nu)\di \nu$.
	
	\medskip
	
	The main purpose of this paper is to address the following problem. Consider a solution $(\bar u, \bar m)$ to the \textit{stationary}  system
	\begin{equation}
		\begin{cases}
			\Delta \bar{u} - \frac{|D\bar{u}|^2}{2} + f(x,\bar{m}) - \delta \bar{u} = 0,& \text{ on }  \Omega \\
			\Delta \bar{m} + \text{div}(\bar{m} D\bar{u})=0,& \text{ on }  \Omega, \\
			\int_{\Omega} \bar m = 1,
		\end{cases}
		\label{statsys}
	\end{equation}
	then \textit{under what conditions can we guarantee that $(\bar u, \bar m)$ attracts dynamic equilibria $(u, m)$ as $T\to \infty$, at least those that are sufficiently close to $(\bar u, \bar m)$ itself ?} The main point of our work is not to assume any kind of global uniqueness of stationary solutions $(\bar u, \bar m)$ and focus on their \textit{local} stability properties. Note that if $\delta = 0$, then we shall look at the so-called \textit{ergodic} system, where the additional constant $\lambda$ appears in the HJ equation:
	\begin{equation}
		\begin{cases}
			\Delta \bar{u} - \frac{|D\bar{u}|^2}{2} + f(x,\bar{m}) - \lambda = 0,& \text{ on }  \Omega \\
			\Delta \bar{m} + \text{div}(\bar{m} D\bar{u})=0,& \text{ on }  \Omega, \\
			\int_{\Omega} \bar m = 1.
		\end{cases}
		\label{ergodic}
	\end{equation}
	In such case, we will assume throughout the paper that $\lambda = 0$. More precisely, given a triple $(\bar u, \bar m, \lambda)$ satisfying the previous system, we tacitly replace $f$ by $f - \lambda$.
	The previous question is formulated in terms of stability properties of PDE systems, but it can be rephrased within the context of optimal control, in particular when $\delta = 0$. In such case, looking at the \textit{static} Mean Field optimal control problem, that is, of the functional
	\[
	\overline{\mathcal F}(w) =\int_{\Omega}\frac{1}2m(y) |w(y)|^2 + F\big(y, m(y) \big) \di y  
	\]
	subject to
	\[
	- \Delta m + \text{div}(m w)=0,\quad  \text{ on } \Omega \qquad \int_{\Omega} m(y) \di y=1,
	\]
	we want to investigate conditions on static minimizers (or critical points) $(\bar w, \bar m)$ ensuring their ability to attract minimizers (or critical points) of $\mathcal F$ as $T \to \infty$.
	
	\medskip
	
	Being the PDE system under consideration at the crossroads of Mean Field Games and (infinite dimensional) optimal control, we can borrow intuitions from both sides. In the optimal control literature, the phenomenon that solutions spend most of the time close to a steady state is often referred to as the \textit{turnpike property}. Its analysis has a long history that started in the '60s, though the first observations  in applied sciences go back to works of von Neumann. A fundamental observation coming from the optimal control literature is that (local) stability of stationary states originates from the (local) saddle point structure of the optimality conditions (or extremal equations, or Pontryagin system), that takes the form of a Hamiltonian system. This fact was already suggested in \cite{Wilde72} in the Linear-Quadratic setting, and it has been more recently developed in general frameworks in \cite{ZP, TZh, TZZ, TZ14}. We refer to the very nice survey \cite{TZsur} for historical insights, recent developments and further references.
	
	The renewed interest in the turnpike phenomenon in optimal control has grown in parallel with the study of long-time behavior in Mean Field Game theory. Within this context, we mention the pioneering works \cite{CLLP1, CLLP2} and the more recent ones \cite{BZ24, CP19, CCDE, CM24, CP21}. The common thread in these contributions is the presence of some assumption that guarantees in fact a \textit{global} turnpike phenomenon: both dynamic and stationary solutions are \textit{unique}, and the former converge in long time to the latter independently of the initial / terminal conditions. These assumptions typically take the form of \textit{monotonicity} (or ``smallness'' of the data in some sense). It should be stressed that \textit{turnpike} is usually intended in this global sense, while here we focus on local stability. Whether it is possible to observe a global turnpike behavior in presence of several stationary states is a difficult, interesting question, and the answer is very likely to depend on the specific model under consideration. Nevertheless, if there was a specific, global mechanism that drives trajectories close to a certain stationary state, what we could deduce via the methods developed here is that such stationary state would be in fact exponentially attractive, provided that a mild nondegeneracy condition (not related in general to monotonicity) holds.
	
	The main purpose of this work is indeed to investigate problems with absence of monotonicity, where multiple stationary states are likely to coexist. The stability analysis of stationary equilibria in non-monotone frameworks has been addressed in few works \cite{BK, CCS, CC24, MSM22}, but it seems at this stage still limited to quite specific models. See also \cite{CarMas} for the application of Weak KAM methods to MFG. In fact, without monotonicity, dynamic phenomena that are far from the turnpike one have been observed, such as periodic behavior \cite{C19} or heteroclinic connections \cite{CC21}.
	
	It is important to observe that, once specialized to the MFG setting, the methods developed in optimal control theory seem to require assumptions that resemble the ones of monotonicity (at least in a local sense). That is why, though we are taking inspiration from those methods here, some ideas will be introduced to address truly non monotone situations. This crucial point will be discussed in detail in Section \ref{sdisc}.
		
	The presence of the quadratic Hamiltonian simplifies the analysis: the method presented here exploits a sort of \textit{symmetry} property of the MFG, which looks clear in the quadratic case. This symmetry property will be clarified and formalized in a general sense in Section \ref{sdisc} below.
		
	\medskip
	%, in this kind of MFGs has the following meaning: the bigger is such a factor, the earlier the agent tries to minimize the cost associated to the game. Our standing assumptions will be the following.
	%\begin{itemize}
%		\item [(T1)] $f: \mathbb{R}^n \times \mathbb{R} \to \mathbb{R}$ has bounded first and second derivative w.r.t. the second component; 
		% \item [(T2)] $\forall \delta \geq 0$, the stationary system~\eqref{statsys} has a unique solution $(\bar{m}_{\delta},\bar{u}_{\delta})$ such that $\bar{m}$ is a strictly positive and bounded function a.e. From now on, we will denote $(\bar{m}_{\delta},\bar{u}_{\delta})$ simply by $(\bar{m},\bar{u})$;
	%\end{itemize}

	We now state the main results. The crucial assumption is that we are given a stationary state  $(\bar u, \bar m)$, that is a (classical) solution to~\eqref{statsys}, satisfying the following properties (note that $\bar m$ is everywhere positive).
	\begin{itemize}
		\item [\Sp]  
		$\bullet$ \, The identity $D\bar m = -\bar m D\bar u $ holds on $\Omega$, \\
		$\bullet$ \,  $\bar m$ satisfies a Poincar\'e inequality with constant $C_P > 0$, i.e.
		\[
		\int_{\Omega} g^2(x) \bar m(x) \di x - \left(\int_{\Omega} g(x) \bar m(x) \di x \right)^2 \le C_P \int_{\Omega} |Dg|^2(x) \bar m(x) \di x \qquad \forall g \in W^{1,2}(\bar m \di x),
		\]
		$\bullet$ \, there exists $\eta \in (0,1)$ such that 
%		\begin{equation}
%			\int_{\mathbb{R}^n} f_m(x,\bar{m}) \mu^2 + \bar m |\beta|^2  + \frac{\delta \mu^2}{\bar{m}} \di x \geq \eta \int_{\mathbb{R}^n}\bar m |\beta|^2 \di x,
%			\label{T3}
%		\end{equation}
%		$\forall \ t \in [0,T]$, $\forall \ (\mu,\beta)$ such that $\Delta \mu + \text{div}(\mu D\bar{u}) + \text{div}(\bar{m} \beta) = 0$ and $\int_{\mathbb{R}^n} \mu(t,x) \di x = 0$.
		\begin{equation}
			\int_{\Omega} f_m(x,\bar{m}) \mu^2 + \bar m\left| D \left( \frac{\mu}{\bar{m}} \right) \right|^2 + \frac{\delta \mu^2}{\bar{m}} \, \di x \geq \eta \int_{\Omega} \bar m\left| D \left( \frac{\mu}{\bar{m}} \right) \right|^2 \di x,
			\label{T3}
		\end{equation}
		for all $\mu \in L^1(\Omega) \cap L^\infty(\Omega)$ such that $\int_{\Omega} \mu(x) \di x = 0$, $\mu/\bar m \in  W^{1,2}(\bar m \di x)$, $\mu/\bar m \in L^\infty(\Omega)$. 
	\end{itemize}
	While the first two points are generally expected to be true, the third one is way more delicate. This is a sort of second-order positivity condition. A thorough description of its derivation, meaning and role in our analysis will be given in Section \ref{sdisc}. We anticipate here that \Sp can be also seen as the positivity of a certain ``principal eigenvalue'' for the linearized stationary system, as described in Remark \ref{firsteigrmk}. This is stronger than the notion of stable solution explored in \cite{BLS25, BriCar17}, where, in some sense, it is required that zero is not an eigenvalue. Note in particular that \cite{BriCar17} focuses on the nondegeneracy properties of the linearized time-dependent problem. Heuristically, what we do here is to show how to transfer the nondegeneracy properties of the linearized \textit{stationary} problem to stability properties of the time-dependent one, uniformly with respect to the time horizon. Several examples of stationary solutions exhibiting such nondegeneracy properties, beyond the usual monotonicity settings, are given below, see Remarks \ref{generic}, \ref{localLLrmk} and Example \ref{exg}.

\medskip
	
	The first result is technical, but crucial to develop what follows. It is an \textit{a priori} estimate on solutions to the MFG system describing a contractivity property; in particular, it states that if we are given a solution $(u,m)$ that is $\eps$-close to the stationary solution $(\bar u, \bar m)$ (in terms of a suitable weighted in time sup-norm), then it has to be in fact $\frac\eps2$-close to $(\bar u, \bar m)$, at least if the initial/final condition $(m(0), u_T)$ is close enough to $(\bar m, \bar u)$.

	\begin{thm} \label{apriori}
		Assume that \Sp holds. Let $(u,m)$ and $(\bar u,\bar m)$ be (classical) solutions to \eqref{dynsys} and \eqref{statsys} respectively, such that
		\[
		\mu=m - \bar{m}, \qquad v=u-\bar{u}
		\]
		satisfy the integrability conditions \eqref{integrab}. Then, there exist $\bar{\eps} > 0, \bar{\gamma} > 0, \sigma > 0$ (which depend only on $\bar{m}, \bar u, \eta, f$, but not on the time horizon $T$), such that the following is true for all $\gamma \le \bar{\gamma}$, $\delta < \sigma$: if
		\begin{equation*}
			\begin{aligned}
			& \| m_0(\cdot)-\bar m(\cdot)\| _{L^{\infty}(\Omega)} + \| u(T,\cdot) - \bar u(\cdot)\|_{W^{1,\infty}(\Omega)} < \gamma,
			%\label{ininf} 
			\\
			& \left \| \frac{m_0(\cdot)-\bar m(\cdot)}{\sqrt{\bar{m}(\cdot)}} \right \|_{L^{2}(\Omega)} + \left\| \sqrt{\bar{m}(\cdot)} \left| Du(T,\cdot) - D\bar u(\cdot) \right| \right\|_{L^2(\Omega)} < \gamma,
			%\label{inl2} 
			\\
			& \|m(t,\cdot)-\bar m(\cdot)\|_{L^{\infty}(\Omega)} \le \bar{\eps}\left( e^{-\sigma_1 t} + e^{-\sigma_2(T-t)} \right) \qquad \forall t \in [0,T],
		%	\label{epsturn} 
		\end{aligned}
		\end{equation*}
		then
		\begin{equation*}
			\|m(t,\cdot)-\bar m(\cdot)\|_{L^{\infty}(\Omega)} \le \frac{\bar{\eps}}{2} \left( e^{-\sigma_1 t} + e^{-\sigma_2(T-t)} \right) \qquad \forall t \in [0,T],
		%	\label{epsturn2}
		\end{equation*}
		where $\sigma_1=\frac{\sigma-\delta}{n+1}$ and $\sigma_2=\frac{\sigma+\delta}{n+1}$.
	\end{thm}
	
		%In (T3) one can notice the role played by the function $\frac{\mu}{\bar{m}}$ and its gradient. In particular, it's really important also for the proof of Theorem~\ref{apriori}. Since we expect the turnpike property to be verified in the case where the discount factor is small, the argument for general $\delta$ will be a perturbation of the one with $\delta=0$. The proof is quite technical; the idea behind it is 

	This estimate is fundamental to obtain the main result, that shows (via a fixed point argument) the existence of \textit{at least} one dynamic solution satisfying the (local) exponential turnpike property around $(\bar u,\bar m)$, provided that the initial / terminal conditions are close enough to the stationary state. We state and prove the theorem in the simpler case where $\Omega = \T^n$, to avoid some non-compactness issues.
		\begin{thm} \label{existenceintro}
	Let $\Omega=\T^n$. Suppose that $(\bar u, \bar m)$ is a (classical) solution to~\eqref{statsys} and that \Sp holds. Then, there exist constants $\bar{\eps} > 0, \gamma > 0, \sigma > 0$ (which depend only on $\bar m, \bar u, \eta, f$ but not on the time horizon $T$), such that, if
	\begin{equation*}
		\delta < \sigma, \quad \| m_0(\cdot)-\bar m(\cdot) \|_{L^{\infty}(\T^n)} + \| u_T(\cdot) - \bar u(\cdot)\|_{W^{1,\infty}(\T^n)} < \gamma,
		%\label{ininf}
		%& \left\| \frac{m_0(\cdot)-\bar m(\cdot)}{\sqrt{\bar{m}(\cdot)}}\right\|_{L^{2}(\T^n)} + \left\| \sqrt{\bar{m}(\cdot)} |Du_T(\cdot) - D\bar u(\cdot)| \right\|_{L^2(\T^n)} < \lambda,
		%\label{inl2}
	\end{equation*}
	then there exists a (classical) solution $(u^T,m^T)$ to~\eqref{dynsys} satisfying the following property:
		\begin{equation*}
		\|m^T(t,\cdot)-\bar m(\cdot)\|_{L^{\infty}(\T^n)} \le \frac{\bar{\eps}}{2} \left( e^{-\sigma_1 t} + e^{-\sigma_2(T-t)} \right) \qquad \forall t \in [0,T],
		%\label{epsturn2}
	\end{equation*}
	where $\sigma_1=\frac{\sigma-\delta}{n+1}$ and $\sigma_2=\frac{\sigma+\delta}{n+1}$.
\end{thm}
	
	The (local) turnpike property of $Du$ to $\bar Du$ is also obtained, in the weighted-$L^2$ sense (see Theorem \ref{existence} below). Note that a bootstrap argument would allow to improve the $L^2$ sense into the $L^\infty$ one, and then one could propagate the turnpike at the level of the value function $u$; this further step would be rather technical, and it will not be detailed here. Note that the existence of at least one trajectory that spends most of the time close to $(\bar u, \bar m)$ is the best that one can expect in general, and it is sometimes false that all trajectories satisfying the initial/final conditions enjoy that property, see for example the discussion at the end of Remark \ref{localLLrmk}; the presence of heteroclinic connections can also give rise to the existence of trajectories that are arbitrarily close to a stationary state at times $0,T$, but spend $T/2$ time close to another stationary state, see for example \cite{CC21}.
	
	Finally, we can pass to the infinite-horizon $T\to\infty$ limit to produce a dynamic equilibrium on the time interval $[0,\infty)$, and having $(\bar u, \bar m)$ as a long-time limit. This result is given by Theorem \ref{exinfhor}. 
	
	\smallskip
	
	The paper is structured as follows. The rest of the introduction contains a heuristic description of the method and several remarks on \Sp. Section \ref{nlturnl2} is devoted to the proof of the crucial a priori turnpike estimates in $L^2$, and these are improved to $L^\infty$ in Section \ref{nlturnlinf}. Finally, in Section \ref{fix2} the existence of solutions satisfying the turnpike property is proven. In the Appendix, some precise estimates on solutions to linear equations in divergence form are given.
		
\subsection{Discussion on the method}\label{sdisc}

Consider systems~\eqref{dynsys} and~\eqref{statsys}, with solutions $(m,u)$ and $(\bar{m},\bar{u})$, respectively. The approach developed here is of perturbative nature, namely, it revolves around the analysis of the couple
	\[
	\mu:=m - \bar{m} \qquad \text{and} \qquad v:=u-\bar{u},
	\]
	that solves the PDE system
	\begin{equation*}
		\begin{cases}
			-\partial_t v - \Delta v + \langle Dv,D\bar{u} \rangle + \delta v = f(x,\bar{m}+\mu) - f(x,\bar{m}) - \frac{|Dv|^2}{2},& \text{ on } (0,T) \times \Omega \\
			\partial_t \mu - \Delta \mu - \text{div}(\bar{m}Dv) - \text{div}(\mu D\bar{u}) = \text{div}(\mu Dv),& \text{ on } (0,T) \times \Omega \\
			\mu(0,\cdot)=\mu_0(\cdot)=m_0(\cdot) - \bar{m}(\cdot) \qquad v(T,\cdot)=v_T(\cdot) = u_T(\cdot)- \bar{u}(\cdot),& \text{ on } \Omega.
		\end{cases}
		%\label{nlPDE}
	\end{equation*}
 Consider now its linearization :
	\begin{equation*}
		\begin{cases}
			\partial_t \mu = \Delta \mu + \text{div}(\bar{m}Dv) + \text{div}(\mu D\bar{u}), \\
			\partial_t v = - \Delta v + \langle Dv,D\bar{u} \rangle - f_m(x,\bar{m})\mu + \delta v. \\
		\end{cases}
		%\label{linPDE}
	\end{equation*}
	Our main purpose is to understand the stability properties of the equilibrium state $(0,0)$. Since we want to write the previous problem as a Hamiltonian system, consider first the exponential transformation
	\begin{equation*}
		\begin{aligned}
			&\tilde{\mu}(t,x) := \mu(t,x) e^{-\frac{\delta t}{2}} \\ 
			&\tilde{v}(t,x) := v(t,x) e^{-\frac{\delta t}{2}},
		\end{aligned}
		%\label{change}
	\end{equation*}
	that satisfy the following system of linear PDEs:
	\begin{equation}
		\begin{cases}
			\partial_t \tilde{\mu} = \Delta \tilde{\mu} + \text{div}(\tilde{\mu} D\bar{u}) + \text{div}(\bar{m} D\tilde{v}) - \frac{\delta}{2} \tilde{\mu}, \\
			\partial_t \tilde{v} = - \Delta \tilde{v} + \langle D\tilde{v}, D \bar{u} \rangle - f_m(x,\bar{m}) \tilde{\mu} + \frac{\delta}{2} \tilde{v}, \\ 
		\end{cases}
		\label{deltahalf}
	\end{equation}
	and that can be rewritten in operator form as follows
	\begin{equation*}
		\begin{bmatrix}
			\dot{\tilde{\mu}} \\
			\dot{\tilde{v}}	
		\end{bmatrix}
		=
		\begin{bmatrix}
			A-\frac{\delta}{2} I & -BB^{*} \\
			-Q & -A^{*} + \frac{\delta}{2} I \\
		\end{bmatrix}
		\begin{bmatrix}
			\tilde{\mu} \\
			\tilde{v}
		\end{bmatrix}
		,
	\end{equation*}
	where
	\begin{align*}
		& A h = \Delta h + \text{div}(h D\bar{u}),\\
		& A^{*} h = \Delta h - \langle Dh,D\bar{u} \rangle,\\
		& BB^{*} h = -\text{div}(\bar{m}Dh),\\
		& Qh = f_m(x,\bar{m}) h.
	\end{align*}
	Notice that the operator $h \mapsto -{\rm div}(\bar{m}Dh)$ is self-adjoint and positive, hence it can be written in the form $BB^{*}$, where $Bf = -{\rm div}(\sqrt{\bar m}f)$. Being a strictly elliptic operator, we may think of its inverse\footnote{note that the invertibility $BB^*$ is sufficient for the (linear) controllability of the system.} (once a normalization condition is introduced). Note also that $Q$ is symmetric. It is well known that, for the Hamiltonian matrix 
	\begin{equation*}
		M=
		\begin{bmatrix}
			A-\frac{\delta}{2} I & -BB^{*} \\
			-Q & -A^{*} + \frac{\delta}{2} I \\
		\end{bmatrix}
	\end{equation*}
	if $\lambda \in \mathbb{C}$ is one of its eigenvalues, then $\lambda^{*}, - \lambda, -\lambda^{*}$ are eigenvalues, too.
	
	Now, linear stability of $(0,0)$, and in particular the turnpike behavior close to it, comes from its hyperbolic nature: we need to prevent $M$ from having purely imaginary eigenvalues, that is, if $\omega$ is any non zero real number, then the operators matrix
	\begin{equation*}
		\begin{bmatrix}
			A-\frac{\delta}{2} I -i\omega I & -BB^{*} \\
			-Q & -A^{*} +\frac{\delta}{2} I - i\omega I \\
		\end{bmatrix}
	\end{equation*}
	must have trivial kernel. So, let $(\bar{\mu},\bar{v})$ be such that 
	\begin{equation*}
		\begin{bmatrix}
			A-\frac{\delta}{2} I -i\omega I & -BB^{*} \\
			-Q & -A^{*} +\frac{\delta}{2} I - i\omega I \\
		\end{bmatrix}
		\begin{bmatrix}
			\bar{\mu} \\
			\bar{v}
		\end{bmatrix}
		=
		\begin{bmatrix}
			0 \\
			0 \\
		\end{bmatrix}
		,
	\end{equation*}
	we have to verify that the only solution be $(\bar{\mu},\bar{v})=(0,0)$. From the first row, we derive that $\bar{v}=(BB^{*})^{-1} (A-\frac{\delta}{2} I - i\omega I ) \bar{\mu}$, from which 
	\begin{multline*}
		\left \{ Q+A^{*}(BB^{*})^{-1}A - \frac{\delta}{2} \left( A^{*} (BB^{*})^{-1} + (BB^{*})^{-1} A \right) + i\omega \left[ (BB^{*})^{-1}A - A^{*}(BB^{*})^{-1} \right] + \right. \\ \left. + \left( \frac{\delta^2}{4} + \omega^2 \right) (BB^{*})^{-1} \right \} \bar{\mu} = 0.
	\end{multline*}
	Since $BB^{*} > 0$, to guarantee that the only solution be $\bar{\mu}=0$ we may assume that
	\begin{itemize}
		\item[\textit{(i)}] $Q+A^{*}(BB^{*})^{-1} A - \frac{\delta}{2} \left( A^{*} (BB^{*})^{-1} + (BB^{*})^{-1} A \right)  \geq 0$ %   \Longrightarrow $ % $(\bar{m},\bar{u})$ minimum of the stationary functional;
		\item[\textit{(ii)}] $A^{*} (BB^{*})^{-1} = (BB^{*})^{-1} A$ % \Longrightarrow$ structural symmetry of the problem.
	\end{itemize}
	
	Let us first show that \textit{(ii)} is a symmetry property that is clear within the quadratic Hamiltonian framework by identity $D \bar u = - D \bar m / \bar m$. Indeed, % With our choice of $A$ and $BB^{*}$ and the fact that $\Delta \bar{m} + \text{div}(\bar{m}D\bar{u}) = 0$, it's easy to prove that the symmetry property holds.
	note first that,
	\begin{equation*}
		(BB^{*}) f = - \text{div}(\bar{m}Df) = -\langle D\bar{m},Df \rangle -\bar{m} \Delta f =  \bar{m} \langle D\bar{u},Df \rangle - \bar{m} \Delta f = -\bar{m} A^{*} f = -T_{\bar{m}} A^{*} f,
	\end{equation*}
	where $T_{\bar{m}}$ is the multiplication operator by $\bar{m}$. Similarly,
	\begin{equation}\label{bbking}
		(BB^{*}) f = - \text{div}(\bar{m}Df) =- \text{div}(D \bar m f + \bar{m}Df + \bar m f D \bar u) = - \Delta(\bar m f)- \text{div}( \bar m f D \bar u) =  -A T_{\bar{m}}  f,
	\end{equation}
	hence,
	\[
	(BB^{*}) A^* = - A T_{\bar{m}} A^* = A (BB^{*}).
	\]
	
	To better understand \textit{(i)}, observe first that\footnote{This observation appears also in \cite{CLLP2}.}, by \eqref{bbking},
	\[
	(BB^{*})^{-1} A \mu = -\frac\mu{\bar m}
	\]
	and again by  $D \bar u = - D \bar m / \bar m$, 
	\begin{equation*}
		A^{*} (BB^{*})^{-1} A \mu = - \Delta \frac\mu{\bar m} + \langle D\frac\mu{\bar m},D\bar{u} \rangle
		= -\frac{1}{\bar{m}} \text{div} \left( \bar{m} D\left( \frac{\mu}{\bar{m}} \right) \right).
	\end{equation*}
%	
%	
%	
%	\begin{equation}
%		\langle Q+A^{*} (BB^{*})^{-1} A - \frac{\delta}{2} \left( A^{*} (BB^{*})^{-1} + (BB^{*})^{-1} A \right)  \mu, \mu \rangle_{L^2(\mathbb{R}^n)} \geq 0.
%		\label{pos}
%	\end{equation}
%	It's easy to prove that $\langle A^{*} (BB^{*})^{-1} A \mu,\mu \rangle_{L^2(\mathbb{R}^n)} =\int_{\mathbb{R}^n} \bar{m} \left| D\left( \frac{\mu}{\bar{m}} \right) \right|^2 \di x $. This yields
%	Finally,
%	\begin{equation}
%		A^{*} (BB^{*})^{-1} \mu = (BB^{*})^{-1} A \mu = (BB^{*})^{-1} \left( \Delta \mu + \text{div}(\mu D\bar{u}) \right) = (BB^{*})^{-1} \text{div}\left(\bar{m}D\left( \frac{\mu}{\bar{m}} \right)\right) = \frac{\mu}{\bar{m}},
%		\label{BBA}
%	\end{equation}
%	where in the last equality we used the equation for $\bar{m}$. \\
	Recalling the definition of $Q$, we put all together and we rewrite \textit{(i)} (using the $L^2$ scalar product and integrating by parts) as
	\begin{equation*}
		\int_{\Omega} f_m(x,\bar{m})\mu^2 + \bar{m} \left| D \left( \frac{\mu}{\bar{m}} \right) \right|^2 +\delta \frac{\mu^2}{\bar{m}} \di x \geq 0.
		%\label{pos2}
	\end{equation*}
%	If we transfer this property and we recall that the equation should be considered backward, we get that we are asking for the following
%	If we now let $\beta = - \sqrt {\bar m} D \left( \frac{\mu}{\bar{m}} \right)$, then the previous inequality is a consequence of
%	\begin{equation}
%		\int_{\mathbb{R}^n} f_m(x,\bar{m})\mu^2 + |\beta|^2 + \frac{\delta \mu^2}{\bar{m}} \di x \geq 0,
%	\end{equation}
%	for all $(\mu,\beta)$ such that $\Delta \mu + \text{div}(\mu D\bar{u}) + \text{div}(\sqrt{\bar{m}} \beta)=0$, with $\int_{\mathbb{R}^n} \mu(t,x) \di x = 0$. 
This inequality is closely related to \Sp: in fact \Sp is a strict version of it. We will show below why the strict version of inequality is needed.
	
	\medskip
%	Therefore, recalling that $f_m = f$, we establish~\eqref{T3}.
%	Notice that if $\beta=\frac{\mu}{\bar{m}}$, the equation $\Delta \mu + \text{div}(\sqrt{m}\beta) + \text{div}(\mu D\bar{u})=0$ is satisfied: this is the precise use of~\eqref{T3} we will make in the computation of Section~\ref{nlturn2}. 

	We are going to elaborate more on \Sp in the next subsection. Before we do that, we continue our discussion on the (linear) stability argument, assuming for simplicity that $\delta = 0$. For comparison purposes, we start with a brief description of the viewpoint adopted in \cite{ZP} (which is also common in the MFG literature, see for example \cite{CLLP1, CLLP2}), that looks at the evolution of the state/adjoint-state quantity
	\[
	\Psi(t) = \langle\mu(t), v(t)\rangle.
	\]
	Using the equations \eqref{deltahalf}, ($\tilde \mu = \mu$ and $\tilde v = v$ since $\delta = 0$) 
	\[
	\dot{\Psi}(t) = - \langle (BB^*)v(t), v(t)\rangle - \langle \mu(t), Q\mu(t)\rangle
	\]
	Then, by the Cauchy-Schwarz and Young inequalities, assuming that both $BB^*$ and $Q$ are \textit{positive definite}, there exists $c > 0$ such that
	\begin{equation}\label{phiineq}
	\dot{\Psi}(t) \le - c |\Psi(t)|.
	\end{equation}
	This differential inequality is the core of the exponential stability: once integrated over time, it gives the exponential decay in time that is typical of the turnpike property. In other words, $\Psi(t)$ can be regarded as a \textit{Lyapunov function}.
	
	A fundamental observation here is that assuming that \textit{$Q$ is positive definite} in our setting amounts to ask that $f_m \ge \bar c > 0$, that is exactly requiring the so-called Lasry-Lions monotonicity condition. Note that this assumption is required in several works on the turnpike property in optimal control \cite{ZP, TZh, TZZ, TZ14}, and forces the functionals $\mathcal F$, $\overline{\mathcal F}$ described in the introduction to be optimized to be uniformly convex (at least close to the stationary point).
	
	\smallskip
	
	The main purpose of this paper is to address situations where such monotonicity / convexity fails. Since $Q \not> 0$, we need to take a different route, starting from the identity
	\[
	-(BB^*)^{-1}\dot{\mu} = v - (BB^*)^{-1}A \mu.
	\]
	Letting
	\[
	\Phi(t) = \langle\mu(t), v(t) - (BB^*)^{-1}A \mu(t) \rangle,
	\]
	We have
	\[
	\dot{ \Phi} = -\langle\dot{\mu},(BB^*)^{-1}\dot{\mu} \rangle + \langle\mu, -Q\mu - A^* v - (BB^*)^{-1}A (A\mu - BB^*v)\rangle.
	\]
	Using the symmetry property \textit{(ii)},
	\begin{equation}\label{phidot}
	\dot{ \Phi} = -\langle\dot{\mu},(BB^*)^{-1}\dot{\mu} \rangle - \langle\mu, Q\mu\rangle  -\langle \mu, A^* (BB^*)^{-1}A \mu \rangle,
	\end{equation}
	which gives, assuming \textit{(i)},
	\[
	\frac{\di}{\di t} \langle\mu, (BB^*)^{-1}\dot{\mu} \rangle = -\frac{\di}{\di t}   \Phi \ge \langle\dot{\mu},(BB^*)^{-1}\dot{\mu} \rangle.
	\]
	It should be clear that it is not possible at this stage to derive an inequality of the form \eqref{phiineq}, that is to control the right-hand side from below by $c|\Phi(t)|$; to this purpose a stronger form of \textit{(i)} is required, for instance ensuring the presence of further term proportional to $\langle \mu,(BB^*)^{-1} \mu \rangle$ in the right-hand side. By the structure of our problem, it will be in fact more natural to assume that there exists $\eta > 0$ such that
	\[
	Q+A^{*}(BB^{*})^{-1} A  \geq \eta A^{*}(BB^{*})^{-1} A,
	\]
	which is exactly the main requirement in \Sp. This allows to reach the inequality
	\[
	- \dot{ \Phi} (t) \ge c|\Phi(t)|,
	\]
	that, arguing as above for $\Psi$, gives the exponential stability.
%	and also that $(BB^{*})^{-1} \geq \eta I$,
%	\[
%	2\eta^2 |\tilde \Phi| \le 2\|\mu\| \,  \eta \|\dot{\mu}\|\le \eta \|\mu\|^2 + \eta  \|\dot{\mu}\|^2 \le \langle\dot{\mu},(BB^*)^{-1}\dot{\mu} \rangle + \langle\mu, Q\mu\rangle +\langle \mu, A^* (BB^*)^{-1}A \mu \rangle = - \dot{\tilde \Phi},
%	\]
%	hence we are in the same position as before to claim the exponential decay of $\tilde \Phi$ and get the turnpike property, since, after integration,
%	\begin{equation}
%		\tilde{\Phi}(T) e^{-2\eta^2(T-t)} \le \tilde{\Phi}(t) \le \tilde{\Phi}(0) e^{-2\eta^2 t} \quad \forall t \in [0,T],
%	\end{equation}
%	

	As the reader will see (Proposition \ref{phinl1} below), there will be no explicit use of the operators $A, BB^*$... below, though all the arguments are clearly inspired by the previous discussion. % There is a crucial use of the Poincar\'e inequality, that may be seen as the spectral gap of the operator $BB^*$. 
	Turning indeed to the explicit form of the operator $(BB^*)^{-1}A$ obtained above, one should look at the evolution of (now for general $\delta \ge 0$)
	\begin{equation*}
		\tilde{\Phi}(t) := \int_{\Omega} e^{-\delta t} \left( \mu(t,x) v(t,x) + \frac{\mu^2(t,x)}{\bar{m}(x)} \right) \di x.
		%\label{tildephiintro}
	\end{equation*}
	Following the previous computations we obtain
	\begin{equation}
			\dot{\tilde{\Phi}}(t) = - \int_{\Omega} e^{-\delta t} \left[ \underbrace{f_{m}(x,\bar{m}) \mu^2 + \bar{m}\left| D\left( \frac{\mu}{\bar{m}} \right) \right|^2 + \frac{\delta \mu^2}{\bar{m}}}_{\text{stationary}} + \underbrace{\bar{m}\left|Dv + D\left( \frac{\mu}{\bar{m}} \right) \right|^2}_{\text{dynamic}} \right] \di x.
		\label{dotphi}
	\end{equation}
	It is very convenient to recognize here the splitting between the ``stationary'' component and the ``dynamic'' one, as it clarifies the role of the \textit{static} optimization problem in the dynamic one. Once again, the second term is related to the dynamic behavior of the system for the following reason: observe that
		\begin{equation*}
			\Delta \mu + \text{div}(\mu D\bar{u}) = \text{div}(D\mu + \mu D\bar{u}) = \text{div}\left( \bar{m} D \left( \frac{\mu}{\bar{m}} \right) \right).
		\end{equation*}
		Therefore, the (linearized) equation for $\mu$ can be rewritten as
		\begin{equation*}
			\partial_t \mu - \text{div}(\bar{m}Dv) - \text{div}\left( \bar{m} D \left( \frac{\mu}{\bar{m}} \right) \right) = 0 \Longrightarrow \partial_t \mu = \text{div}\left(\bar{m}\left( Dv+D \left( \frac{\mu}{\bar{m}} \right) \right) \right),
		\end{equation*}
		hence the vector field $Dv + D \left( \frac{\mu}{\bar{m}} \right)$, appearing in the ``dynamic'' part of \eqref{dotphi}, drives in fact the evolution of $\mu$. 
	
	Note that we have focused here on linear stability only. In the general, nonlinear case, the proof of the exponential decay of the quantity $\Phi$ goes along similar lines, but one needs to handle the presence of lower order terms. There is also some technical work that is needed to derive decay bounds on weighted $L^2$ norms, then on $L^\infty$ norms. Moreover, the presence of $\delta > 0$ also needs to be handled with care; in a sense, this is a further perturbation of the problem. Indeed, the previous argument gives the exponential decay of $\tilde \Phi(t) = e^{-\delta t} \Phi(t)$. To recover the exponential decay of $\Phi(t)$, one then needs $\delta$ to be small enough, since the exponential decay for $\Phi$ from the one of $\tilde \Phi$ takes the form
\[
e^{-(\sigma-\delta) t} +  e^{-(\sigma+\delta) (T-t)}.
\]

There is another, elegant approach that allows to turn the linear exponential stability into the nonlinear one, that relies on the analysis of the Riccati feedback operator $\mathcal E(T) \mu_0 = \tilde v (0)$, where $(\tilde v, \tilde \mu)$ is the solution to the linear system \eqref{deltahalf} with final condition $\tilde v(T) = 0$ and initial condition $\tilde \mu(0) = \mu_0$. As shown in \cite{PorrettaMinMax}, the exponential convergence as $T\to\infty$ of $\mathcal E(T)$ is all that one needs to get local, nonlinear long time stability of the stationary state; such exponential convergence is proven in \cite{PorrettaMinMax} using the classical Lasry-Lions monotonicity argument; the approach proposed here allows in fact to show that exponential convergence of $\mathcal E(T)$ holds also under the weaker property \Sp.

%	Under smallness assumptions on $\bar{\eps}$ and $\bar{\lambda}$, hence on $\Theta$, and by using some weighted Poincaré inequality, one deduces from~\eqref{dotphi} that $\tilde{\Phi}$ satisfies the following property
%	
%	which is precisely the one from which is possible to derive the weighted $L^2$ turnpike property for $\mu$ and $Dv$. A further (technical) analysis is then developed to transform weighted $L^2$ estimates, into $L^\infty$ ones.% Notice that, since we need to interpolate with weighted $L^2$ norms, which are not equivalent (in general) to the corresponding standard $L^2$ norms, since $\bar{m}$ may not have a strictly positive infimum over $\mathbb{R}^n$. \\

\subsection{The assumption \Sp}

We discuss in this section the assumption \Sp. 

First, the identity $D\bar m = -\bar m D\bar u$ is true if $\bar m(x)$ is proportional to $e^{-\bar u(x)}$: this ansatz gives in fact a solution to the stationary Fokker-Planck equation, so any given solution $\bar m$ is actually proportional to $e^{-\bar u}$ provided that uniqueness for the Fokker-Planck equation holds. This uniqueness issue has been largely explored, and it holds under rather mild integrability assumptions, for instance if $D \bar u \in L^2(\bar m)$, see \cite[Example 5.5.4]{BKRS} (we refer to the same monograph for a thorough discussion on uniqueness).

Similarly, there is a large body of literature on the Poincar\'e inequality for $\bar m$ if $\Omega$ is bounded (that goes under the name of weighted Poincar\'e inequality); this is valid for example if $\bar m$ is regular enough and bounded away from zero (which is our case, by elliptic regularity and the strong maximum principle). When $\Omega = \R^n$, an additional care is needed: the inequality is somehow related also to the behavior of $\bar m$ as $x \to \infty$, and it is true in many situations, see for example \cite{bakry}. 

\smallskip

The main restriction in \Sp is the inequality \eqref{T3}. We first observe that, under mild assumptions on $f$, such inequality has equivalent formulations. 

\begin{lem}[Equivalent formulations of \Sp]\label{abc}
	Assume that $\bar m, f_m (\cdot,\bar{m}(\cdot)) \in L^\infty(\Omega)$, and that $\bar m$ satisfies a Poincar\'e inequality. Then, the following properties are equivalent :
	\begin{itemize}
		\item [(a)] $\int_{\Omega} f_m (x,\bar{m}) \mu^2 + \bar{m} \left| D\left( \frac{\mu}{\bar{m}} \right) \right|^2 + \delta \frac{\mu^2}{\bar{m}} \di x \geq \eta \int_{\Omega} \bar{m} \left| D\left( \frac{\mu}{\bar{m}} \right) \right|^2 \di x;$
		\item [(b)] $\int_{\Omega} f_m(x,\bar{m}) \mu^2 + \bar{m} \left| D\left( \frac{\mu}{\bar{m}} \right) \right|^2 + \delta \frac{\mu^2}{\bar{m}} \di x \geq \eta' \int_{\Omega} \frac{\mu^2}{\bar{m}} \di x;$
		\item [(c)] $\int_{\Omega} f_m(x,\bar{m}) \mu^2 + \bar{m} \left| D\left( \frac{\mu}{\bar{m}} \right) \right|^2 + \delta \frac{\mu^2}{\bar{m}} \di x \geq \eta'' \int_{\Omega} \mu^2 \di x,$
	\end{itemize}
	for some $\eta, \eta', \eta'' > 0$ and for all $\mu \in L^1(\Omega) \cap L^\infty(\Omega)$ such that $\int_{\Omega} \mu(x) \di x = 0$, $\mu/\bar m \in  W^{1,2}(\bar m \di x)$.
\end{lem}
\begin{proof}
	Suppose that (a) holds. The first implication is a simple application of the weighted-Poincar\'e inequality for $\frac{\mu}{\bar{m}}$. In fact, we know that
	\begin{equation*}
		\int_{\Omega} \left( \frac{\mu}{\bar{m}} \right) \bar{m} \, \di x =  \int_{\Omega} \mu \, \di x = 0,
	\end{equation*}
	hence
	\begin{multline*}
		\int_{\Omega} f_m(x,\bar{m}) \mu^2 + \bar{m} \left| D\left( \frac{\mu}{\bar{m}} \right) \right|^2 + \delta \frac{\mu^2}{\bar{m}} \di x \overset{(a)}{\geq} \eta \int_{\Omega} \bar{m} \left| D\left( \frac{\mu}{\bar{m}} \right) \right|^2 \di x \\ \geq \frac{\eta}{C_P} \int_{\Omega} \bar{m} \frac{\mu^2}{\bar{m}^2} \di x = \frac{\eta}{C_P} \int_{\Omega} \frac{\mu^2}{\bar{m}} \di x,
	\end{multline*}
	where $C_P$ is the Poincar\'e constant of $\bar{m}$. Therefore, (b) holds with $\eta'=\frac{\eta}{C_P}$. 
	
	Next, let (b) hold. Since $\bar{m}$ is bounded from above and non-negative, we have that
	\begin{equation*}
		\eta' \int_{\Omega} \frac{\mu^2}{\bar{m}} \di x % \geq \eta' \int_{\Omega} \frac{\mu^2}{\|\bar{m}\|_{L^{\infty}(\Omega)}} \di x 
		\geq \frac{\eta'}{\|\bar{m}\|_{L^{\infty}(\Omega)}} \int_{\Omega} \mu^2 \di x,
	\end{equation*}
	hence (c) holds with $\eta''=\frac{\eta'}{\| \bar{m} \|_{L^{\infty}}}$.
	
	Finally, let (c) hold. Take $\eps \in (0,1)$. We can split the left hand side of (c) into two pieces, as follows
	\begin{multline*}
			\int_{\Omega} f_m(x,\bar{m}) \mu^2 +  \bar{m} \left| D\left( \frac{\mu}{\bar{m}} \right) \right|^2 + \delta \frac{\mu^2}{\bar{m}} \di x \\=(1-\eps) \int_{\Omega} f_m(x,\bar{m}) \mu^2 +  \bar{m} \left| D\left( \frac{\mu}{\bar{m}} \right) \right|^2 + \delta \frac{\mu^2}{\bar{m}} \di x + \\+  \eps \int_{\Omega} f_m (x,\bar{m}) \mu^2 +  \bar{m} \left| D\left( \frac{\mu}{\bar{m}} \right) \right|^2 + \delta \frac{\mu^2}{\bar{m}} \di x \geq
			\\\geq (1-\eps)\eta'' \int_{\Omega} \mu^2 \di x + \eps \int_{\Omega} f_m(x,\bar{m}) \mu^2 +  \bar{m} \left| D\left( \frac{\mu}{\bar{m}} \right) \right|^2 + \delta \frac{\mu^2}{\bar{m}} \di x=\\=\int_{\Omega} \left[ (1-\eps)\eta'' + \eps f_m(x,\bar{m}) \right] \mu^2 \di x + \eps \int_{\Omega} \bar{m} \left| D\left( \frac{\mu}{\bar{m}} \right) \right|^2 + \delta \frac{\mu^2}{\bar{m}} \di x \geq \\ \geq \int_{\Omega} \left[ (1-\eps)\eta'' - \eps \|f_m(\cdot,\bar{m})\|_{L^{\infty}(\Omega)} \right] \mu^2 \di x + \eps \int_{\Omega} \bar{m} \left| D\left( \frac{\mu}{\bar{m}} \right) \right|^2 + \delta \frac{\mu^2}{\bar{m}} \di x = \blacktriangle
	\end{multline*}
	Now, choosing $\eps=\frac{\eta''}{\eta'' + \| f_m(\cdot,\bar{m})\|_{L^{\infty}(\Omega)}}$, we have
	$$(1-\eps)\eta'' - \eps \| f_m(\cdot,\bar{m})\|_{L^{\infty}(\Omega)} = 0.
	$$
	Hence, we proceed from $\blacktriangle$ to get
	\begin{equation*}
		\blacktriangle \geq \frac{\eta''}{\eta'' + \|f_m(\cdot,\bar{m})\|_{L^{\infty}(\Omega)}} \int_{\Omega} \bar{m} \left| D \left( \frac{\mu}{\bar{m}} \right) \right|^2 \di x,
	\end{equation*}
	so that (a) holds with $\eta:=\frac{\eta''}{\eta'' + \|f_m(\cdot,\bar{m})\|_{L^{\infty}(\Omega)}}$.
\end{proof}

\begin{rmk}[The principal eigenvalue]\label{firsteigrmk} The inequalities (a), (b), (c) of the previous lemma can be interpreted as bounds from below of certain quotients of Rayleigh type. For instance (b) reads as
\[
\inf_{\int_\Omega \mu = 0}\frac{\int_{\Omega} f_m(x,\bar{m}) \mu^2 + \bar{m} \left| D\left( \frac{\mu}{\bar{m}} \right) \right|^2 + \delta \frac{\mu^2}{\bar{m}} \di x}{\int_{\Omega} \frac{\mu^2}{\bar{m}} \di x} \geq \eta' > 0.
\]
If such infimum is attained at some $\mu$, then (formally) the following equality holds
\begin{equation}\label{raylinear}
f_m(x,\bar{m}) \mu - \frac1{\bar m}{\rm div}\left( \bar{m} D\left( \frac{\mu}{\bar{m}} \right) \right) + \delta \frac{\mu}{\bar{m}} = \eta_1 \frac{\mu}{\bar{m}} + \ell,
\end{equation}
where $\eta_1\ge \eta'$ and $\ell \in \R$ ($\ell$ can be characterized in terms of $\mu, \bar m$). Via the change of variables $v = -\frac \mu {\bar m}$ and the identity $D \bar u = - \frac{D \bar m}{\bar m}$ one finds the PDE system solved by the triple $(v, \mu, \ell)$ :
\begin{equation}\label{firsteig}
\begin{cases}
	- \Delta v + \langle Dv,D\bar{u} \rangle + \delta v - f_m(x,\bar{m})\mu = \eta_1 v + \ell \\
	\Delta \mu + \text{div}(\bar{m}Dv) + \text{div}(\mu D\bar{u}) = 0, \\ \int_\Omega \mu = 0.
\end{cases}
\end{equation}
So, another way of seeing \Sp is that 
\begin{align*}
&\textit{the first (or principal) eigenvalue $\eta_1$, that is the smallest number $\eta_1$} \\ & \quad \textit{such that \eqref{firsteig} has nontrivial solutions, must be positive.}
\end{align*}
Arguing similarly with the inequality (c) yields an eigenvalue problem of the form
\[
\begin{cases}
	- \Delta v + \langle Dv,D\bar{u} \rangle + \delta v - f_m(x,\bar{m})\mu = -\hat \eta_1 \mu + \ell \\
	\Delta \mu + \text{div}(\bar{m}Dv) + \text{div}(\mu D\bar{u}) = 0, \\ \int_\Omega \mu = 0.
\end{cases}
\]
Though all considerations are made here in a merely formal way, it would be interesting to develop a systematic spectral theory for the present problems.
\end{rmk}

\begin{rmk}[Non-local couplings]\label{nonlocrem} The methods proposed in this work can be adapted easily to problems involving non-local couplings, that is, when $f : \Omega \times \R \to \R$ is replaced by $f : \Omega \times \mathcal P (\Omega) \to \R$, where $\mathcal P (\Omega)$ is the space of probability measures over $\Omega$, of course under suitable regularity assumptions on $f$. Indeed, the approach relies on linearization arguments, that, in the local case, involve expansions of the form
\[
f(x, \bar{m}(x)+\mu(x)) - f(x,\bar{m}(x)) \simeq f_m(x,\bar m(x)) \mu(x) .
\]
In the non-local case, one would work with the expansion
\[
f(x, \bar{m}+\mu) - f(x,\bar{m}) \simeq \int_\Omega \delta_m f(x,\bar m,y) \mu(y) \di y,
\]
where $\delta_m f$ is the so-called flat derivative. In particular the main inequality in assumption \Sp would take the form
\[
	\int_{\Omega}\int_{\Omega} \delta_m f(x,\bar{m},y) \mu(y) \mu(x) \di y \di x +  \int_{\Omega} \bar m\left| D \left( \frac{\mu}{\bar{m}} \right) \right|^2 + \frac{\delta \mu^2}{\bar{m}} \, \di x \geq \eta \int_{\Omega} \bar m\left| D \left( \frac{\mu}{\bar{m}} \right) \right|^2 \di x.
\]
\end{rmk}

We now point out that \eqref{T3} can be also related to local minimality properties of $\bar m$ for a static optimization problem, at least when $\delta = 0$. This is motivated by the fact, already mentioned in the introduction, that the stationary PDE system \eqref{statsys} constitutes the optimality conditions of the minimization of the functional $\overline{\mathcal{F}}(m,w)$ under the constraints $-\Delta m + \text{div}(m w)=0$ and $\int_{\Omega} m \, \di x=1$. In the quadratic case, it is possible to further simplify the static optimization problem by introducing a functional $\widetilde{\mathcal{F}}$, which depends only on the variable $m$, defined as follows
\begin{equation*}
	\widetilde{\mathcal{F}}(m):=\int_{\Omega} F(x,m) + \frac{|Dm|^2}{2m} \di x = \int_{\Omega} F(x,m) + 2 |D\sqrt{m}|^2 \di x, \qquad \forall \sqrt{m} \in W^{1,2}(\Omega), \int_{\Omega} m =1.
	%\label{functonevar}
\end{equation*}
The next Lemma explains that \eqref{T3} holds whenever $\bar m$ is a strict (local) minimum of $\widetilde{\mathcal{F}}$.

\begin{lem}
	If $\bar{m}$ is a minimum for $\widetilde{\mathcal{F}}$ such that $\bar m$ is continuous and everywhere positive, then \eqref{T3} holds with $\eta = 0$. If, moreover, there exists $\eta'' > 0$ such that
	\begin{equation} 
		\widetilde{\mathcal{F}}(m) - \widetilde{\mathcal{F}}(\bar{m}) \geq \frac{\eta''}{2} \int_{\Omega} (m-\bar{m})^2 \di x \qquad \forall \sqrt m \in W^{1,2}(\Omega),
		\label{coerci}
	\end{equation}
	then assertion (c) of Lemma \ref{abc} holds (hence \eqref{T3} holds for some $\eta > 0$).
\end{lem}
\begin{proof}
	Take $\mu$ such that $\mu \in L^1(\Omega) \cap L^\infty(\Omega)$, $\int_{\Omega} \mu(x) \di x = 0$, $\mu/\bar m \in  W^{1,2}(\bar m \di x)$. Assume also that it has compact support. We take a small variation around the minimum $\bar{m}$, namely $\bar m + \eps \mu$, for $\eps > 0$ small enough so that $\bar m + \eps \mu > 0$ (recall that $\bar m$ is continuous and everywhere positive, and $\mu \in L^\infty$). Then, we compute variations of $\widetilde{\mathcal{F}}$ as follows
	\begin{multline*}
		\F(\bar m + \eps \mu) = \F(\bar m) + \delta\F(\bar m)[\eps \mu] + \frac12 \delta^2\F(\bar m)[\eps \mu, \eps \mu] \\
		-\eps^3 \int_{\Omega}\left[ \frac{\mu \bar m}{2(\bar m+ \eps \mu)} \left|D\left( \frac \mu{\bar m}\right) \right|^2 - \int_0^1 \frac{(1-s)^2}{2} f_{mm}(x, \bar m + s \eps \mu)\mu^3\di s \right] \di x,
	\end{multline*}
	where
	\begin{align*}
	&\delta\F(\bar m)[\eps \mu] = \eps \int_{\Omega} \frac{\langle D \bar m, D \mu \rangle}{\bar m} - \mu \frac{|D \bar m|^2}{2\bar m^2} + f(x,\bar m) \mu \, \di x, \\
	& \delta^2\F(\bar m)[\eps \mu, \eps \mu] = \eps^2 \int_{\Omega} \bar m \left|D\left( \frac \mu{\bar m}\right) \right|^2 + f_m(x,\bar m) \mu^2 \, \di x.
	\end{align*}

	Since $\bar{m}$ is a minimum of $\F$, then letting $\eps \to 0$ implies that $\delta\F(\bar m)[\mu] = 0$ and $\delta^2\F(\bar m)[\mu, \mu] \ge 0$ for all $\mu$ (with compact support). Note that, if $\mu$ has not compact support, one can construct a sequence $\mu_R= (\mu -c_R)\phi(x/R)$, where $\phi$ is a cutoff function, $c_R = \int \mu \phi/\int \phi$ , so that $\delta^2\F(\bar m)[\mu_R, \mu_R] \to \delta^2\F(\bar m)[\mu, \mu]$ as $R \to \infty$; therefore, \eqref{T3} holds with $\eta = 0$.

	For the second assertion, if we assume the validity of~\eqref{coerci}, then $\bar m$ is still a minimizer of $\F$, and
	\begin{equation*}
		\widetilde{\mathcal{F}}(\bar{m}+\eps \mu) - \widetilde{\mathcal{F}}(\bar{m}) \geq \frac{\eta \eps^2}{2} \int_{\Omega} \mu^2 \di x \Longleftrightarrow \frac12 \delta^2\F(\bar m)[\eps \mu, \eps \mu] +o(\eps^2) \geq \frac{\eta\eps^2}{2} \int_{\Omega} \mu^2 \di x,
	\end{equation*}
	from which we derive Lemma \ref{abc}(c).
\end{proof}

%\begin{rmk} Note that if $\bar m$ is a minimizer of $\F$, then \textcolor{blue}{$\delta\F(\bar m)[\mu] = 0$} for all test functions $\mu$ such that $\int_{\Omega} \mu = 0$; this means that
%\[
%-{\rm div}\left(\frac{D \bar m}{\bar m}\right) - \frac12\left|\frac{D \bar m}{\bar m}\right|^2 + f(x, \bar m) = \lambda
%\]
%for some $\lambda\in \R$. If we let $\bar u = -\log \bar m$, then the triple $(\bar u, \bar m, \lambda)$ solves \eqref{ergodic}. Note also that the left-hand side of \eqref{raylinear} is nothing but the linearization of the previous equation.
%\end{rmk}

%\begin{rmk} The minimization of $\F$ can be recast, after the change of variables $m=\varphi^2$ into the minimization of
%\[
%\varphi \mapsto \int_{\Omega} F(x,\varphi^2) + 2 |D\varphi|^2 \di x, \qquad \forall \varphi \in W^{1,2}(\Omega), \int_{\Omega} \varphi^2 \, \di x=1,
%\]
%that is related to the normalized solutions of the (stationary) nonlinear Schr\"odinger equation.
%\end{rmk}

In light of the previous lemmata, we have the following observation on the generic nature of \Sp.

\begin{rmk}[\Sp is a generic property of minimizers]\label{generic}
Note that any (constrained) minimizer $\bar m$ of
\[
	\widetilde{\mathcal{F}}(m)=\int_{\Omega} F(x,m) + \frac{|Dm|^2}{2m} \di x
\]
gives a stationary solution of the PDE system \eqref{statsys}. Indeed, taking variations we have $\delta\F(\bar m)[\mu] = 0$ for all test functions $\mu$ such that $\int_{\Omega} \mu = 0$; this means that
\[
-{\rm div}\left(\frac{D \bar m}{\bar m}\right) - \frac12\left|\frac{D \bar m}{\bar m}\right|^2 + f(x, \bar m) = \lambda
\]
for some $\lambda\in \R$. If we let $\bar u = -\log \bar m$, then the triple $(\bar u, \bar m, \lambda)$ solves \eqref{ergodic}. Note also that the left-hand side of \eqref{raylinear} is nothing but the linearization of the previous equation.

Now, for any $\eta'' > 0$, one can define
\[
\widetilde{\mathcal{F}}_{\eta''}(m):=\widetilde{\mathcal{F}}(m) + \frac{\eta''}{2} \int_{\Omega} (m-\bar{m})^2 \, \di x,
\]
and clearly $\bar m$ is also a minimizer of $\widetilde{\mathcal{F}}_{\eta''}$, yielding in particular a solution of
\[
\begin{cases}
	\Delta \bar{u} - \frac{|D\bar{u}|^2}{2} + f(x,\bar{m}) + \eta''(m - \bar m) = 0,& \text{ on }  \Omega \\
	\Delta \bar{m} + \text{div}(\bar{m} D\bar{u})=0,& \text{ on }  \Omega, \\
	\int_{\Omega} \bar m = 1.
\end{cases}
\]
Moreover, $\widetilde{\mathcal{F}}_{\eta''}$ satisfies the inequality \eqref{coerci}, hence \Sp can be verified. In other words, \textit{any} minimizer $\bar m$ of a given problem with coupling $f$ is actually a (locally) stable solution of a (slightly) perturbed problem, having the coupling $f + \eta''(m-\bar m)$. This is true without any assumption on the sign of $f_m$, and becomes of course relevant when $f$ is monotone decreasing, in which case several stationary states may coexist (as $f + \eta''(m-\bar m)$ remains decreasing if $\eta''$ is small).
\end{rmk}

%In light of the previous lemmata, the following approach to the construction of stationary solutions satisfying \Sp becomes quite natural, at least in the case where $\delta = 0$. First of all, if one considers minimizers of $\F$, these satisfy \eqref{T3} only in a ``weak'' sense. Such minimizers are not necessarily stable for the evolutionary problem. However, if in specific cases one can show that these minimizers are strong (or nondegenerate), so that \eqref{T3} is fully satisfied, then the results of the paper apply. In any case, given \textit{any} minimizer $\bar m$, it is possible to perturb the problem by replacing $f$ with $f + \eta(m-\bar m)$, in such a way that \eqref{T3} is satisfied. This perturbation is clearly artificial, but it shows in some way that the minimizers of \Sp are stable, at least in a generic sense.

The construction and the analysis of minimizers of problems involving non-monotone couplings has been carried out in several works, see for example \cite{CCCon, CVe, CKWZ, KTZ, KTZZ, MS18}. The study of their \textit{nondegeneracy} properties is, to our knowledge, open even in special cases. We believe that this constitutes an interesting research direction: in light of our results, such nondegeneracy would be in fact sufficient to show that those stationary minimizers are actually locally stable for the time-dependent problem.

To this purpose, note that the minimization of $\F$ can be recast, after the change of variables $m=\varphi^2$ into the minimization of
\[
\varphi \mapsto \int_{\Omega} F(x,\varphi^2) + 2 |D\varphi|^2 \di x, \qquad \forall \varphi \in W^{1,2}(\Omega), \int_{\Omega} \varphi^2 \, \di x=1,
\]
that is related to the normalized solutions of the (stationary) nonlinear Schr\"odinger equation, and several ideas from the literature on nonlinear Schr\"odinger equations could be borrowed to prove nondegeneracy of solutions, see for instance \cite{PPVV, W99}.

\begin{rmk}[The \textit{local} Lasry-Lions mildly non-monotone case]\label{localLLrmk} In the standard \textit{global} Lasry-Lions monotonicity regime, that is when $f_m \ge \ell > 0$ everywhere, it is clear that Lemma \ref{abc} (c) holds with $\eta'' = \ell$, and so \Sp holds. Actually, we can allow for a \textit{local, mild non-monotonicity assumption}
\begin{equation}\label{localmono}
\inf_{x \in \Omega} f_m(x,\bar{m}(x))\bar m(x) = -c > - \frac 1 {C_P}
\end{equation}
Indeed, by means of the Poincar\'e inequality,
\[
\int_{\Omega} f_m(x,\bar{m}) \mu^2 + \bar{m} \left| D\left( \frac{\mu}{\bar{m}} \right) \right|^2 + \delta \frac{\mu^2}{\bar{m}} \di x \geq  \left( -c + \frac{1}{C_P} \right)  \int_{\Omega}  \frac{\mu^2}{\bar{m}} \di x,
\]
and Lemma \ref{abc} applies.

We stress that \eqref{localmono} is a local assumption. As an illustrative example, consider, in dimension $n=1$ and with $\delta = 0$, the problem with coupling
\[
f(x,m) = \frac{3A}2 m^{1/2} - \frac{A}2 m^{3/2},
\]
that has the trivial stationary solution $(\bar u, \bar m, \lambda) = (0, 1, A)$ for every $A > 0$. We have that $\bar m$ satisfies \Sp : $C_P$ is the constant of the standard Poincar\'e-Wirtinger inequality, and $f_m(x, \bar m) \equiv 0$, hence \eqref{localmono} holds. So, $\bar m$ is a stable solution (for every $A > 0$).

Note that $f$ is not globally Lasry-Lions monotone (not even in the mild sense), since $f_m(m) \to -\infty$ as $m \to \infty$. Actually, the problem admits another stationary solution $(\tilde u, \tilde m, \tilde \lambda)$, that is given by the global minimum of the minimization problem
\begin{equation*}
	\widetilde{\mathcal{F}}(m)=\int_{-1/2}^{1/2} A m^{3/2} - \frac{A}5 m^{5/2} + \frac{|Dm|^2}{2m} \, \di x, \qquad \forall \sqrt{m} \in W^{1,2}(\T), \int_{-1/2}^{1/2} m =1.
	%\label{functonevar}
\end{equation*}
Indeed, such a minimum exists (as the exponent $3/2$ is mass-subcritical in dimension one, see e.g. \cite{CCCon}) and it is not achieved by $\bar m = 1$ when $A$ is large. To verify this, take any smooth $q \in C^\infty_0(-1/2,1/2)$ such that $\sqrt{q} \in W^{1,2}(-1/2,1/2)$. Then, for $r > 1$,
\[
\widetilde{\mathcal{F}}(r q(r \cdot))=\int_{-1/2}^{1/2} A r^{1/2 }q^{3/2} - \frac{A}5 r^{3/2} q^{5/2} + r^2 \frac{|Dq|^2}{2q} \di x = A(a_1 r^{1/2} - a_2 r^{3/2}) + a_3 r^2
\]
where $a_1, a_2, a_3 > 0$. Choose $r$ large in a way that $a_1 r^{1/2} - a_2 r^{3/2} < 0$. Then, for $A$ large enough,
\[
\widetilde{\mathcal{F}}(\tilde m) = \min_m \widetilde{\mathcal{F}}(m) \le \widetilde{\mathcal{F}}(r q(r \cdot)) < 0 < \frac{4A}5 = \widetilde{\mathcal{F}}(\bar m).
\]
Note that a minimizer $\tilde m$ cannot satisfy \Sp, since the problem is translation invariant and $\{\tilde m(\cdot-y)\}_{y \in \T}$ provides a continuum of minimizers, that are degenerate: \eqref{T3} holds only with $\eta = 0$. Nevertheless, the problem could be perturbed in a way that $\tilde m$ becomes stable (see Remark \ref{generic}), and it is reasonable that the trivial solution $\bar m = 1$ would be perturbed into another stable solution. Note also that the geometry of $\widetilde{\mathcal{F}}$ suggests the presence of another family of solutions of mountain-pass type, arising from the presence of the isolated, local minimum $\bar m$ and the global minima $\{\tilde m(\cdot-y)\}_{y \in \T}$.

We also observe that stability of the trivial stationary state guarantees the existence of dynamic equilibria that remain close to it, at least for initial / final conditions that are close. Anyhow, even if we take the initial / final conditions to be exactly $m(0) = \bar m = 1$ and $u(T) = \bar u = 0$, we \textit{cannot} expect that all solutions $(u,m)$ coincide with $(\bar u, \bar m)$ for all times, not even that the two couples are always close. Indeed, it would be convenient for any \textit{dynamic} minimizer $(u,m)$ (that is a minimizer of the corresponding time-dependent functional involving $\widetilde{\mathcal{F}}$ and a Lagrangian term related to the velocity) to be close to (some translation of) $(\tilde u, \tilde m)$ for most of the time.

Finally, the local condition $f_m(x,\bar{m}(x)) \ge -c/\bar m(x)$ is comparable with the one assumed in \cite{CP21}, that reads $f_m(x,m) \ge -\gamma$ for all $x,m$. Despite the difference between the local/global nature of the two conditions, it could be noticed that the first one appears to be somehow weaker when tested on problems on the whole Euclidean space, where $\bar{m}(x)$ typically vanishes as $x \to \infty$. However, our local condition just guarantees local stability, and we do not know whether the monotonicity of $f(m) + c\log m$ (for small $c > 0$) would be enough to imply global uniqueness (as when $f(m) + \gamma m$ is monotone increasing \cite{CP21}).
\end{rmk}

\begin{exmpl}[A Linear-Quadratic-Gaussian problem] \label{exg} Let $g:(0,\infty) \to (0,\infty)$ be differentiable. Consider a stationary problem with $\delta=0$, $n=1$ and $\Omega=\R$ with quadratic cost and coupling via second-order moments of $m$:
	\begin{equation*} \label{statg1}
		\begin{cases}
			-\bar u_{xx} + \frac{1}{2} \bar u_x^2 +\bar \lambda = g\left( \int_{\R} x^2 \bar m(x) \, \di x \right) x^2 ,& \text{ on } \R \\
			\bar m_{xx} + (\bar m \bar u_x)_x=0,& \text{ on } \R.
		\end{cases}
	\end{equation*}
Using the standard Quadratic-Gaussian ansatz, we get solutions of the form
\[
(\bar u, \bar m, \bar \lambda) = \left(\frac{x^2}{2 \bar M}, \frac{e^{-x^2 / (2\bar M)}}{\int_\R e^{-x^2 / (2\bar M)}}, \frac1{\bar M} \right)
\]
if (and only if)
\[
g(\bar M) = \frac{1}{2\bar M^2}
\]
Notice that $\bar M =  \int_{\R} x^2 \bar m$. In other words, given $g$, it may intersect the curve $M \mapsto 1/(2M^2)$ several times, and each intersection gives a stationary solution. Are these solutions stable? We claim that if
\begin{equation}\label{gprimecond}
g'(\bar M) > -\frac{1}{\bar M^3},
\end{equation}
then \textit{any $(v, \mu, \ell, \eta)$ such that $\|v''\|_{L^\infty(\R)}, \|x^2 \mu\|_{L^1(\R)} < \infty$ satisfying the linearized system}
\begin{equation} \label{firsteigsecg}
		\begin{cases}
			- v_{xx} + v_x \cdot \bar u_x - g'(\bar M) \left( \int_{\R} y^2 \mu(y) \, \di y \right) x^2= \eta v + \ell \\
			\mu_{xx} + (\bar m\cdot v_x)_x + (\mu \cdot \bar u_x)_x = 0, \qquad (v = -\mu/\bar m)  \\
			 \int_\R \mu = 0.
		\end{cases}
	\end{equation}
is such that $\eta \ge \eta_1$ for some $\eta_1 > 0$. This has a strong connection with the validity of \Sp, as discussed in Remark \ref{firsteigrmk}. Interestingly, $-M^{-3}$ is the derivative of $M^{-2}/2$, hence any time $g$ crosses $M \mapsto 1/(2M^2)$ from below, a stable solution (in the sense of spectral properties of \eqref{firsteigsecg}) is generated, see Figure \ref{ggraph}. Local monotonicity in the Lasry-Lions sense holds only if $g'(\bar M) > 0$, and it can be checked that displacement monotonicity (see \cite{CM24}) is locally verified only if $g'(\bar M) > -\bar M^{-3}/4$, hence our analysis can detect further (locally) stable solutions.

\begin{figure}
\centering
\begin{tikzpicture}
\begin{axis}[
    domain=0.2:3,
    samples=200,
    axis lines=left,
    xlabel={$M$},
    ylabel={$g(M)$},
    ymin=0, ymax=2,
    xmin=0.2, xmax=3,
    ticks = none,
        width=0.5\textwidth,
        height=0.3\textwidth,
        ylabel style={rotate=-90}
]

% f(M)
\addplot[blue, dashed] {1/(2*x^2)};

% g(M) (scelta per avere più intersezioni)
\addplot[red, thick] {(2+cos(100*x))/(1+ 2*x^2)};

% punti (prima e terza intersezione - valori approssimati)
\addplot[only marks, mark=*, mark size=2.5pt]
coordinates {
    (0.575, {1/(2*(0.575)^2)})   % prima
    (2.09, {1/(2*(2.09)^2)})   % terza
};

\end{axis}
\end{tikzpicture}
\caption{Solid line is the graph of $g(M)$, while dashed one is the graph of $(2M)^{-1}$. Intersections indicated with black dots give rise to stable stationary solutions.}\label{ggraph}
\end{figure}

Before we proceed with the proof of the claim, it should be noted that the existence Theorem \ref{existenceintro} does not apply directly to this example, which is set on the unbounded space $\R$, and involves solutions with quadratic growth like $\bar u$ (that is the reason why it is reasonable to restrict to perturbations $v$ having bounded second order derivatives). In addition, the problem involves a nonlocal coupling rather than a local one (see Remark \ref{nonlocrem} on this aspect). We still believe that this example is instructive, since it shows how to apply our arguments in a specific situation that exhibits several stationary states. We believe that a more general version of Theorem \ref{existenceintro}, accommodating a framework where solutions are unbounded like $\bar u$ (and not necessarily quadratic polynomials in $x$), could be obtained with some additional technical work.

To prove the claim, we observe first that $v$ (with bounded second derivatives) can be only a quadratic polynomial. Indeed, let $w(x):=v_{xx}(x)$ be the second derivative of $v$. If we plug $\bar u_x = x/\bar M$  in~\eqref{firsteigsecg} and differentiate twice the equation for $v$, we get that the equation satisfied by $w$ is
\begin{equation} \label{Hermite}
	-w_{xx} + \frac{1}{\bar M} w_x \cdot x - \left (\eta - \frac{2}{\bar M}\right) w = 2g'(\bar M) \left( \int_{\R} y^2 \mu(y) \, \di y \right).
\end{equation}
Equation~\eqref{Hermite} takes the form of a (non-homogeneous) Hermite equation. Now, if the right-hand side is zero, it is well known that the values of $\eta$ which give solutions with polynomial growth are
$$\eta_n = (2n+2) \frac{1}{\bar M} \qquad n \in \mathbb{N}.$$
In particular, for each $\eta_n$ the corresponding (admissible) solution $w$ is the Hermite polynomial of degree $n$. Since we are looking for a bounded $w$, we just focus on the node $n=0$, which corresponds to $w=C$, for some constant $C$. Therefore, the function $v$ is a polynomial of degree (at most) $2$. Otherwise, if the right-hand side $g'(\bar M) \left( \int_{\R} y^2 \mu(y) \, \di y \right) \neq 0,$ then we have
$$w(x):=z(x)+\frac{2 g'(\bar M) \left( \int_{\R} y^2 \mu(y) \, \di y \right)}{\frac{2}{\bar M} - \eta},$$
where $z$ solves
$$-z_{xx} + \frac{1}{\bar M} z_x \cdot x - \left( \eta-\frac{2}{\bar M} \right) z = 0.$$
As above, $z$ is constant, hence $v$ is quadratic (note that if $\eta = 2/\bar M$ we are done, since such value of $\eta$ is positive).

Then, $v$ solving \eqref{firsteigsecg} has to be of the following form:
\[
v(x) =\alpha x^2 + \beta x + \gamma.
\]
Note that, since $\int_\R \mu = 0$, testing the equation for $\mu$ by $x^2$ gives the identity
\[
\int_{\R} x^2 \mu \, \di x = -2\alpha \bar M^2,
\]
hence plugging the two identities above into the equation for $v$ gives the following set of conditions:
\begin{equation*}
		\begin{cases}
			\frac{2\alpha}{\bar M} + 2\alpha \bar M^2 g'(\bar M) = \eta \alpha\\
			\frac{\beta}{\bar M} = \eta \beta \\
			-2\alpha = \eta \gamma +\ell
		\end{cases}
	\end{equation*}
Notice first that the case $\alpha = \beta = 0$ has to be ruled out, since, by the identity $v = -\mu/\bar m$, if $v$ is constant then $\mu$ is constant as well, but identically zero since $\int \mu = 0$, so $(v, \mu)$ is the trivial solution. Hence, assume that $\alpha \neq 0$. Then from the first equation
\[
\eta = \frac{2}{\bar M} + 2 \bar M^2 g'(\bar M) > 0
\]
in view of \eqref{gprimecond}. If $\alpha = 0$, the second equation forces $\eta = \frac{1}{\bar M} > 0$ (since $\beta$ cannot be zero). Therefore we obtain that $\eta \ge \min\{\frac{2}{\bar M} + 2 \bar M^2 g'(\bar M), \frac{1}{\bar M} \} > 0$.
%Finally, Lemma~\ref{exglemma} allows to conclude that if $v$ has bounded second order derivatives, but it is not a quadratic polynomial, then $\eta = \frac{2}{\bar M} > 0$. 
\end{exmpl}

\subsection{Final remarks and open questions}
	
	\begin{rmk}[More general Hamiltonians] The presence of a quadratic Hamiltonian allows to present the approach proposed in this paper in a more transparent way, but we believe that more general situations can be addressed similarly. If we revisit the discussion of Section \ref{sdisc} replacing the Hamiltonian $|Du|^2$ by $H(x,Du)$, we first observe that the linearized parabolic system reads
	\[
	\begin{cases}
			\partial_t \mu = \Delta \mu + \text{div}(\mu H_p(D \bar u)) + \text{div}(\bar m H_{pp}(D \bar u)Dv ) , \\
			\partial_t v = - \Delta v + \langle Dv,H_p(D \bar u) \rangle - f_m(x,\bar{m})\mu + \delta v \\
		\end{cases}
	\]
	so we have the same operator form with
	\begin{align*}
		& A \mu = \Delta \mu + \text{div}(\mu H_p(D \bar u)),\\
		& BB^{*} v = -\text{div}(\bar m H_{pp}(D \bar u)Dv).
	\end{align*}
	Heuristically, our approach can be developed following almost identical lines provided that the symmetry condition 
	\[
		A^{*} (BB^{*})^{-1} = (BB^{*})^{-1} A
	\]
	holds. There are situations besides the purely quadratic one where the previous identity is satisfied, but those will be explored in future work.
	
%	for example when $H(x,Du) = |Du|^2 + \langle b(x), Du \rangle$, where $b$ is a fixed drift. Or, other problems involving radial symmetry may fall into our framework, such as the ones described in \cite{CHC}.
%	
%	A further step would be to address perturbative situations where $A^{*} (BB^{*})^{-1} - (BB^{*})^{-1} A$ is small, but not identically zero. One may hope that a ``small asymmetry'' could be compensated by the nondegeneracy of the stationary problem.
	\end{rmk}
	
	\begin{rmk}[The potential structure] In the MFG framework, the variational point of view on the problem under consideration is related with the structure of the Lagrangian cost $|\alpha|^2 + f(x,m)$. It is known that, given a more general Lagrangian $L(x, \alpha, m)$, a potential description may not be available. Nevertheless, the approach presented here may still apply, since it works at the level of PDE systems; indeed, one would need to linearize the parabolic problem, and address the symmetry and nondegeneracy conditions discussed in Section \ref{sdisc}. 
	
	Moreover, one could replace $f$ by a nonlocal coupling (see the previous Remark \ref{nonlocrem}); while in the local case the problem  with cost $|\alpha|^2 + f(x,m)$ has always potential structure, the nonlocal case may fail to have a variational characterization.
	\end{rmk}
	
	\begin{rmk}[Large discount]
		The approach proposed here requires $\delta$ to be small; if one looks at the main statements, the presence of the discount ``unbalances'' the forward / backward in time exponential decay terms, and the arguments break down if $\delta$ is too large. The validity of the (local) turnpike phenomenon is not clear when $\delta$ is large, and it would be worth investigating, also in light of the recent results \cite{BCflock, CCont} on models involving large $\delta$.
		
%		n fact, even if we look at the simple finite dimensional minimization problem
%		\begin{equation*}
%			\int_0^T e^{-\delta t} \left( \frac{|\dot{x}(t)|^2}{2} + V(x(t)) \right) \di t,
%		\end{equation*}
%		with Euler-Lagrange equation 
%		\begin{equation*}
%			\ddot{x}_{\delta}(t) - \delta \dot{x}_{\delta}(t) - V'(x_{\delta}(t)) = 0,
%		\end{equation*}
%		whose solution approaches a constant in time function as $\delta$ becomes large, since
%		\begin{equation*}
%			\frac{\ddot{x}_{\delta}(t)}{\delta} - \dot{x}_{\delta}(t) - \frac{V'(x_{\delta}(t))}{\delta} = 0 \quad \underset{\delta \to +\infty}{\Longrightarrow} \quad -\dot{x}_{\infty}(t) = 0.
%		\end{equation*}
%		The stationary problem does not depend on $\delta$, because it is just $V'(\bar{x})=0$; hence, if $x_{\delta}(0) \neq \bar{x}$, where $\bar{x}$ is such that $V'(\bar{x})=0$, then the solution itself remains near $x_{\delta}(0)$, which implies that the turnpike property cannot be satisfied.
	\end{rmk}
	
	\begin{rmk}[Final condition depending on $m$] If one replaces the final condition $u_T(x)$ with an $m$-dependent final cost $g(x, m)$, we expect that a local Lasry-Lions monotonicity condition of the form $g_m(x, \bar m(x)) \ge 0$ should be sufficient to preserve the stability of $(\bar u, \bar m)$; indeed, that would add to $\Phi(T)$ a positive contribution, which would preserve the validity of the crucial estimates on $\Phi$ (see for example \ref{nuovaint}).
	\end{rmk}

	\begin{rmk}[Dissipativity] The approach presented here can be also revisited under the lens of (local) \textit{dissipativity}, a notion that is known to be connected to the turnpike property in optimal control theory (see for instance \cite{TZh}). In particular, integrating \eqref{phidot} on $[t_1, t_2]$ gives
	\[
	\Phi(t_1) \ge \int_{t_1}^{t_2}\langle\mu, Q\mu\rangle  + \langle \mu, A^* (BB^*)^{-1}A \mu \rangle \di t + \Phi(t_2) ,
	\]
	and whenever the right-hand side controls from above $\int_{t_1}^{t_2} |\mu|^2$, it is known that the turnpike property can be proven, at least in an integral sense (this kind of control from above is in fact our assumption \Sp). Comparing the previous inequality with \cite[Definition 1]{TZh}, one may recognize that $\Phi$ plays the role of a \textit{storage function}, though it is time-dependent (and not state-dependent) and the inequality is reversed.
	
	We also mention the work \cite{KFG}, that uses the notion of local dissipativity to address the local turnpike property in some discrete-time optimal control problems.
	\end{rmk}
	
	\bigskip
	
	\textbf{Acknowledgements.} M. C. is member of the Gruppo Nazionale per l'Analisi Matematica, la Probabilit\`a e le loro Applicazioni (GNAMPA) of the Istituto Nazionale di Alta Matematica (INdAM). He has been partially funded by the EuropeanUnion–NextGenerationEU under the National Recovery and Resilience Plan (NRRP), Mission 4 Component 2 Investment 1.1 - Call PRIN 2022 No. 104 of February 2, 2022 of Italian Ministry of University and Research; Project 2022W58BJ5 (subject area: PE - Physical Sciences and Engineering) “PDEs and optimal control methods in mean field games, population dynamics and multi-agent models”, and by the King Abdullah University of Science and Technology Research Funding (KRF) under award no. CRG2024-6430.6.

\section{$L^2$ turnpike a priori estimates} \label{nlturnl2}
	%We now come back to the non-linear system~\eqref{PDE2}. We'll try to show the same turnpike property also in this case. To do this, we need some additional conditions on the data, in particular a sort of "smallness" hypothesis on the $L^{\infty}$ norm of the function $\mu$ and on the gradient of $v$.
	The results in this section are the crucial steps towards the proof of Theorem \ref{apriori}, and concern the $L^2$ local turnpike behavior of solutions. Let us first recall the non-linear system which comes by taking the difference between the equations satisfied by $u$,$m$ and their stationary counterparts $\bar{u}$ and $\bar{m}$ :
	\begin{equation}
		\begin{cases}
			-\partial_t v - \Delta v + \langle Dv,D\bar{u} \rangle - f(\bar{m}+\mu) + f(\bar{m}) + \frac{|Dv|^2}{2} + \delta v=0,&  \text{ on } (0,T) \times \Omega \\
			\partial_t \mu - \Delta \mu - \text{div}(\bar{m}Dv) - \text{div}(\mu D\bar{u}) - \text{div}(\mu Dv)= 0,&  \text{ on } (0,T) \times \Omega  \\
			\mu(0,\cdot)=\mu_0(\cdot) \quad v(T,\cdot)=v_T(\cdot),& \text{ on }  \Omega. \\
		\end{cases}
		\tag{NL}
		\label{nlsys}
	\end{equation}
	Now, consider a solution to the previous system, i.e. a couple $(v,\mu)$ which satisfies the two coupled equations. We Taylor expand $f$ around $\bar{m}$ with integral remainder, so that the first equation becomes
	\begin{multline*}
	-\partial_t v - \Delta v + \langle Dv,D\bar{u} \rangle - f_m(\bar{m})\mu - \xi \mu^2 + \frac{|Dv|^2}{2} + \delta v= 0, \\
	 \text{where \quad $\xi(x,t)=\int_0^1(1-z)f_{mm}(x, \bar m(x) + z \mu(x,t)) \, \di z$}
	 \end{multline*}
	We need to assume some integrability conditions for $\mu$ and $v$ in order to be able to perform the computations in Lemmata of this section:
	\begin{equation}\label{integrab}
	\begin{aligned}
	& \partial_t \mu(t), \mu(t) \in L^1(\Omega), \quad  \mu(t), D \mu(t), D^2 \mu(t), \mu(t) / \bar m \in L^\infty(\Omega), \\
	& \partial_t v(t), v(t), Dv(t), D^2 v(t), D^3 v(t) \in L^\infty(\Omega), \\
	& \int_\Omega \mu(t) = 0
	\end{aligned}
	\end{equation}
for a.e. $t \in (0,T)$.
\begin{rmk} \label{nurmk2}
	It is worth noting that, as all the estimates in this section are either monotone with respect to the initial and terminal data or independent of them, they extend to the following PDE system:
	\begin{equation} \label{nupde}
			\begin{cases}
			-\partial_t v - \Delta v + \langle Dv,D\bar{u} \rangle - f(\bar{m}+ \nu \mu) + f(\bar{m}) + \frac{|Dv|^2}{2} + \delta v=0,&  \text{ on } (0,T) \times \Omega \\
			\partial_t \mu - \Delta \mu - \text{div}(\bar{m}Dv) - \text{div}(\mu D\bar{u}) - \text{div}(\mu Dv)= 0,&  \text{ on } (0,T) \times \Omega  \\
			\mu(0,\cdot)=\nu \mu_0(\cdot) \quad v(T,\cdot)=\nu v_T(\cdot),& \text{ on }  \Omega, \\
		\end{cases}
	\end{equation}
for all $\nu \in [0,1]$; this will be important to conclude the fixed point argument in Section \ref{fix2}. Indeed, it will be crucial below to use the inequality in \Sp in the more general form
\[
\int_{\Omega} \nu f_m(x,\bar{m}) \mu^2 + \bar m\left| D \left( \frac{\mu}{\bar{m}} \right) \right|^2 + \frac{\delta \mu^2}{\bar{m}} \, \di x \geq \eta \int_{\Omega} \bar m\left| D \left( \frac{\mu}{\bar{m}} \right) \right|^2 \di x,
\]
where $\eta$ is independent of $\nu \in [0,1]$. This is true because the inequality is assumed in \Sp to hold when $\nu = 1$, while for $\nu = 0$ it is obviously satisfied with $\eta = 1$. The general case follows by convex combination.

\end{rmk}

The following Lemma is the fundamental tool to get a weighted $L^2$ exponential-type turnpike property, which will be used in the next Section. Let us first denote by $C_f$ the global bound on first and second derivatives of $f$ with respect to $m$.
$$ \| f_m \|_{L^{\infty}(\Omega \times \R)} + \| f_{mm} \|_{L^{\infty}(\Omega \times \R)} \le C_f.$$

	\begin{prop} \label{phinl1}
		Suppose that \eqref{integrab} and \Sp holds. Let  
		$$\tilde{\Phi}(t):=\int_{\Omega}e^{-\delta t} \left( \mu(t,x) v(t,x) + \frac{\mu^2(t,x)}{\bar{m}(x)} \right) \di x.$$
		Then, there exist $\Theta \le 1$ depending on $\bar{m}, \eta$ and a constant $\sigma > 0$ depending on $\eta,C_P,C_f,\bar{m}$ (where $C_P$ is the Poincar\'e constant of $\bar{m}$) such that, if
		$$\sup_{t \in [0,T]}\|\mu(t,\cdot)\|_{L^\infty(\Omega)} + \sup_{t \in [0,T]} \| Dv(t,\cdot) \|_{L^\infty(\Omega)} \le \Theta,$$
		then
		\begin{equation}
			\Phi(T) e^{-(\sigma+\delta)(T-t)} \le \Phi(t) \le \Phi(0)e^{-(\sigma-\delta) t},
			\label{phidelta}
		\end{equation}
		for all $t \in [0,T]$, where
		\begin{equation*}
			\Phi(t):=e^{\delta t} \tilde{\Phi}(t) = \int_{\Omega} \mu(t,x) v(t,x) + \frac{\mu^2(t,x)}{\bar{m}(x)} \, \di x.
		\end{equation*}
	\end{prop}
	\begin{proof}
		We proceed directly with the computation of the derivative of $\tilde{\Phi}$. We first perform some integration by parts, recalling that $D\mu + \mu D\bar{u}= \bar{m}D\left( \frac{\mu}{\bar{m}} \right)$, since $D \bar m = - \bar m D \bar u$. Hypothesis~\eqref{integrab} ensures the operation of integration by parts (in particular when $\Omega = \R^n$) and that derivatives can be moved into integrals.
		\begin{equation*}
			\dot{\tilde{\Phi}}(t)=\int_{\Omega} e^{-\delta t} \left[ (\partial_t \mu) v + \mu (\partial_t v) - \delta \mu v + 2\frac{\mu (\partial_t \mu)}{\bar{m}} - \delta \frac{\mu^2}{\bar{m}} \right] \di x=
		\end{equation*}
		\begin{multline*}
			=\int_{\Omega} e^{-\delta t} \left[ v(\Delta \mu + \text{div}(\mu D\bar{u} + \bar{m}Dv + \mu Dv ) \right] \di x + \\
			+\int_{\Omega} e^{-\delta t} \left[ \mu \left( -\Delta v + \langle Dv,D\bar{u} \rangle - f_m(\bar{m})\mu - \xi \mu^2 + \frac{|Dv|^2}{2} + \delta v \right) \right] - \delta \mu v e^{-\delta t} \di x +  \\
			+ \int_{\Omega} e^{-\delta t} \left[ 2\frac{\mu}{\bar{m}} \left( \text{div}(D\mu + \mu D\bar{u} + \bar{m} Dv + \mu Dv \right) \right] - \delta \frac{\mu^2}{\bar{m}} e^{-\delta t} \di x=
		\end{multline*}
		\begin{multline*}
			=\int_{\Omega} e^{-\delta t} \left[ - \bar{m} |Dv|^2 - \mu |Dv|^2 - f_m(\bar{m})\mu^2 - \xi \mu^3 + \frac{\mu|Dv|^2}{2} - \delta \frac{\mu^2}{\bar{m}} - 2\bar{m} \left| D\left(\frac{\mu}{\bar{m}} \right) \right|^2 \right] \di x +\\
			+\int_{\Omega} e^{-\delta t} \left[- 2 \bar{m} \langle Dv,D\left( \frac{\mu}{\bar{m}} \right) \rangle - 2 \mu \langle Dv,D\left( \frac{\mu}{\bar{m}} \right) \rangle \right] \di x.
		\end{multline*}
		From now on, for the sake of simplicity, we will denote by
		$$z:=\frac{\mu}{\bar{m}}.$$
		Changing all the signs, we get the following equality:
		\begin{multline}
			-\dot{\tilde{\Phi}}(t) = \int_{\Omega} e^{-\delta t} \Big[ \frac{\mu |Dv|^2}{2} + \bar{m}|Dv|^2 +  f_m(\bar{m})\mu^2 + \xi \mu^3 + \delta \frac{\mu^2}{\bar{m}} +\\ + 2\bar{m}|Dz|^2 + 2\bar{m} \langle Dv,Dz \rangle + 2\mu \langle Dv,Dz \rangle \Big] \di x.
			\label{part1}
		\end{multline}
	We now use the hypothesis \Sp with $\mu$. 
%This can be done because
%		\begin{equation}
%			\Delta \mu(t,x) + \text{div}(-\bar{m} Dz(t,x)) + \text{div}(\mu(t,x)D\bar{u}(x)) = 0,
%		\end{equation}
	%	for all $t \in [0,T]$.
		Hence, for all $t \in [0,T]$,
		\begin{equation}
			\int_{\Omega} f_m (\bar{m})\mu^2 + \bar{m}|Dz|^2 + \delta \frac{\mu^2}{\bar{m}} \di x \geq \eta \int_{\Omega} \bar{m} |Dz|^2 \di x,
			\label{coerc}
		\end{equation}
		for some $\eta \in (0,1)$.
		Then, if one plugs~\eqref{coerc} into~\eqref{part1}, this yields
		\begin{equation*}
			\begin{split}
				-\dot{\tilde{\Phi}}(t) &\geq \int_{\Omega} e^{-\delta t} \left[ \frac{\mu |Dv|^2}{2} + \bar{m}|Dv|^2 + \xi \mu^3+ (1+\eta)\bar{m}|Dz|^2  + 2\bar{m} \langle Dv,Dz \rangle + 2\mu \langle Dv,Dz \rangle \right] \di x = \\
				&=\int_{\Omega} e^{-\delta t} \left[ \frac{\mu|Dv|^2}{2}+ \bar{m}|Dv+Dz|^2 + \eta \bar{m}|Dz|^2 +\xi  \mu^3 + 2\mu \langle Dv,Dz \rangle \right] \di x.
			\end{split}
		\end{equation*}
		Since it can be easily shown that (as $\eta \in (0,1)$),
		\begin{equation*}
			\bar{m}|Dv+Dz|^2 + \eta \bar{m} |Dz|^2 \geq \frac{\eta}{4} \bar{m} |Dv|^2 + \frac{\eta}{4} \bar{m} |Dz|^2,
		\end{equation*}
		we can rewrite the whole inequality as follows:
		\begin{equation}
			-\dot{\tilde{\Phi}}(t) \geq \int_{\Omega} e^{-\delta t} \left[ \frac{\eta}{4} \bar{m}|Dv|^2 + \frac{\eta}{4} \bar{m}|Dz|^2 +\xi \mu^3 + 2\mu \langle Dv,Dz \rangle + \frac{\mu|Dv|^2}{2} \right] \di x.
			\label{part2}
		\end{equation}
		Now, we deal with the last three terms of~\eqref{part2} using the smallness hypothesis on $\mu$ and $Dv$ :
		\begin{itemize}
			\item $2 \mu \langle Dv,Dz \rangle \geq - 2 |\mu| |Dv| |Dz| \geq - 2 \Theta |\mu| |Dz| \geq - \Theta \frac{\mu^2}{\bar{m}} - \Theta \bar{m}|Dz|^2;$
			\item $\xi \mu^3 \geq - C_f \left \| \mu(t,\cdot) \right \|_{L^{\infty}(\Omega)} \left \| \bar{m} \right \|_{L^{\infty}(\Omega)} \frac{\mu^2}{\bar{m}} \geq - 
			C_f \|\bar{m}\|_{L^{\infty}(\Omega)} \Theta \frac{\mu^2}{\bar{m}};$
			\item $\frac{1}{2} \mu |Dv|^2 \geq -\frac{1}{2} |\mu| |Dv|^2 \geq - \frac{1}{2} \Theta |\mu| |Dv| \geq - \frac{1}{4} \Theta \frac{\mu^2}{\bar{m}} - \frac{1}{4} \Theta \bar{m} |Dv|^2.$
		\end{itemize}
		We plug all the previous estimates into~\eqref{part2} and we get:
		\begin{multline*}
			-\dot{\tilde{\Phi}}(t) \geq \\
			\int_{\Omega} e^{-\delta t} \bar{m} |Dv|^2 \left[ \frac{\eta}{4} - \frac{\Theta}{4} \right] + e^{-\delta t} \bar{m} |Dz|^2 \left[ \frac{\eta}{4} - \Theta \right]+ e^{-\delta t} \frac{\mu^2}{\bar{m}}  \left[ - \Theta C_f \|\bar{m}\|_{L^{\infty}(\Omega)} - \frac{5 \Theta}{4} \right] \di x \geq
		\end{multline*}
		Now, since $\int_{\R^n} \bar{m}(x) z(t,x) \di x = \int_{\Omega} \mu(t,x) \di x = 0$ for all $t$, by the Poincar\'e inequality we have that
		\[
		 \int_{\Omega} \frac{\mu^2}{\bar{m}} \di x = \int_{\Omega} z^2 \bar{m} \di x \le C_P \int_{\Omega} |Dz|^2 \bar{m} \di x,
		\]
		hence
		\begin{multline*}
			-\dot{\tilde{\Phi}}(t) \geq \\
			\int_{\Omega} e^{-\delta t} \bar{m} |Dv|^2 \left[ \frac{\eta}{4} - \frac{\Theta}{4} \right] + e^{-\delta t} \bar{m} |Dz|^2 \left[ \frac{\eta}{8} - \Theta \right] + e^{-\delta t} \left[\frac{\eta}{8C_P}- \Theta C_f \|\bar{m}\|_{L^{\infty}(\Omega)} -\frac{5\Theta}{4} \right] \frac{\mu^2}{\bar{m}} \di x.
		\end{multline*}
		
		Now, if $$\Theta \le \frac\eta{16} \quad \text{and} \quad \Theta \le \frac{\eta}{16C_P(C_f \|\bar{m}\|_{L^{\infty}(\Omega)} + 5/4)}$$ the following inequality follows:
		\begin{equation}\label{part3}
			\begin{split}
				-\dot{\tilde{\Phi}}(t) & \geq \int_{\Omega} e^{-\delta t}  \left[ \frac{3\eta}{16} \bar{m} |Dv|^2+  \frac\eta{16}  \bar{m} |Dz|^2 +\left(\frac{\eta}{8C_P}- \Theta C_f \|\bar{m}\|_{L^{\infty}(\Omega)} -\frac{5\Theta}{4} \right) \frac{\mu^2}{\bar{m}} \right] \di x \\
				& \geq \int_{\Omega} e^{-\delta t}  \left[ \frac{3\eta}{16} \bar{m} |Dv|^2+  \frac\eta{16}  \bar{m} |Dz|^2+ \frac{\eta}{16C_P} \frac{\mu^2}{\bar{m}} \right] \di x \\
				& \geq \frac{\eta}{32 C_P} \int_{\Omega} e^{-\delta t} \left( \bar{m}\hat{v}^2 + \frac{\mu^2}{\bar{m}} \right) + e^{-\delta t} \frac{\mu^2}{\bar{m}} \di x,
			\end{split}
		\end{equation}
		where $\hat{v}=v-\int_{\Omega} \bar{m}(y) v(t,y) \, \di y$. \\
		Hence, by Young's inequality and the triangle inequality, we get that, for $\sigma = \frac{\eta}{32 C_P}$,
		\begin{equation*}
			-\dot{\tilde{\Phi}}(t) \geq \sigma \int_{\Omega} e^{-\delta t} |\mu \hat{v}| + e^{-\delta t} \left| \frac{\mu^2}{\bar{m}} \right| \di x \geq \sigma e^{-\delta t} \left| \int_{\Omega} \mu \hat{v} + \frac{\mu^2}{\bar{m}} \di x \right| \geq \sigma e^{-\delta t} \left| \int_{\Omega} \mu \hat{v} + \frac{\mu^2}{\bar{m}} \di x \right| = \sigma |\tilde{\Phi}(t)|,
			%\label{inphi}
		\end{equation*}
		since
		\begin{equation*}
			\int_{\Omega} \mu(t,x)\hat{v}(t,x) \di x = \int_{\Omega} \mu v \, \di x - \int_{\Omega} \mu(t,x) \underbrace{\left( \int_{\Omega} \bar{m}(y)v(t,y) \di y \right)}_{=\text{constant w.r.t. } x} \di x = \int_{\Omega} \mu(t,x) v(t,x) \di x.
		\end{equation*}
		Then, we finally obtain that
		\begin{equation*}
			\begin{cases}
				\dot{\tilde{\Phi}}(t) \le \sigma \tilde{\Phi}(t) \\
				\dot{\tilde{\Phi}}(t) \le -\sigma \tilde{\Phi}(t),
			\end{cases}
		\end{equation*}
		from which, after integration on $[t,T]$ and $[0,T]$, one gets
		\begin{equation}\label{tildephiineq}
			\tilde{\Phi}(T) e^{-\sigma(T-t)} \le \tilde{\Phi}(t) \le \tilde{\Phi}(0) e^{-\sigma t},
		\end{equation}
		which implies~\eqref{phidelta}.
		
		Finally, note that integrating \eqref{part3} on $[t_1,t_2]$ gives
		\begin{equation}\label{nuovaint}
		\int_{t_1}^{t_2} \int_{\Omega} e^{-\delta t} \left[\bar{m} |Dv|^2+  \bar{m} |Dz|^2+  \frac{\mu^2}{\bar{m}}\right] \di x \di t \le c [\tilde{\Phi}(t_1)-  \tilde{\Phi}(t_2)],
		\end{equation}
		where $c$ depends on $\eta, C_P$.
	\end{proof}
	
	The next two Lemmata show how to convert the exponential decay estimate \eqref{phidelta} into a similar one on $\mu$ and $Dv$.
	
	\begin{prop}\label{phinl2}
		Suppose that \eqref{integrab} and \Sp holds, together with the following condition on $\mu$ and $Dv$:
		\begin{equation*}
			\sup_{t \in [0,T]} \left \| \mu(t,\cdot) \right \|_{L^\infty(\Omega)} + \sup_{t \in [0,T]} \left \| Dv(t,\cdot) \right\|_{L^\infty(\Omega)} \le \Theta,
			%\label{smallness}
		\end{equation*}
		where $\Theta$ is as in Proposition~\ref{phinl1}.
		Then, we the following inequality holds:
		\begin{equation*}
			\left \| \frac{\mu(t,\cdot)}{\sqrt{\bar{m}(\cdot)}}\right \|^2_{L^2(\Omega)} = \left \| \frac{m(t,\cdot)-\bar{m}(\cdot)}{\sqrt{\bar{m}(\cdot)}} \right \|^2_{L^2(\Omega)} \le \lambda_1 e^{-(\sigma-\delta) t} + \lambda_2 e^{-(\sigma+\delta) (T-t)},
		%	\label{nlexpprop2}
		\end{equation*}
		for all $t \in [0,T]$, with
		$$\lambda_1 = C\left(|\Phi(0)|+ \left\| \frac{\mu_0}{\sqrt{\bar{m}}} \right\|_{L^2(\Omega)}^2 \right), \quad \lambda_2 = C|\Phi(T)|$$
		\label{L2mu}
		where $\sigma$ is defined in Proposition~\ref{phinl1} and $C$ depends on $\eta, C_P$.
	\end{prop}
	\begin{proof}
		We start from the equation for $\mu$ and we multiply it by $\frac{\mu}{\bar{m}} e^{-\delta t}$. Then, we integrate by parts with respect to $x$ to get,
%		\begin{equation*}
%			\int_{\mathbb{R}^n} \partial_t \mu \left( \frac{\mu}{\bar{m}} \right) e^{-\delta t} - \Delta \mu \left( \frac{\mu}{\bar{m}} \right) e^{-\delta t} - \text{div}(\bar{m} Dv) \left( \frac{\mu}{\bar{m}} \right) e^{-\delta t} - \text{div}(\mu D\bar{u}) \left( \frac{\mu}{\bar{m}} \right) e^{-\delta t} - \text{div}(\mu Dv) \left( \frac{\mu}{\bar{m}} \right) e^{-\delta t} \di x = 0,
%		\end{equation*}
		with $z=\frac{\mu}{\bar{m}}$ and $D\bar u = -D \bar m / \bar m$ as before,
		\begin{equation*}
			\int_{\Omega} \partial_t \left( \frac{\mu^2}{2\bar{m}} e^{-\delta t}  \right) + \delta \frac{\mu^2}{2\bar{m}} e^{-\delta t} + e^{-\delta t} \bar{m} |Dz|^2 + e^{-\delta t} \mu \langle Dv,Dz \rangle + e^{-\delta t} \bar{m} \langle Dv,Dz \rangle \di x = 0.
		\end{equation*}
		We then use a Young inequality on the last term to get
		\begin{equation*}
			\frac{\di}{\di t} \int_{\Omega} \frac{\mu^2}{2\bar{m}} e^{-\delta t} \di x + \int_{\Omega} \delta e^{-\delta t}  \frac{\mu^2}{2\bar{m}} + \frac{1}{2} \bar{m} |Dz|^2 e^{-\delta t} \di x \le \int_{\Omega} \frac{1}{2} e^{-\delta t} \bar{m} |Dv|^2 - e^{-\delta t} \mu \langle Dv,Dz \rangle \di x.
		\end{equation*}
		Integration in time between $t_1$ and $t_2$ yields
		\begin{multline}
				\int_{\Omega} \frac{\mu^2(t_2,x)}{\bar{m}(x)} e^{-\delta t_2} \di x  \le \\ \int_{\Omega} \frac{\mu^2(t_1,x)}{\bar{m}(x)} e^{-\delta t_1} \di x + \int_{t_1}^{t_2} \int_{\Omega} e^{-\delta t} \bar{m}|Dv|^2 \di x \di t - 2 \int_{t_1}^{t_2} \int_{\Omega} e^{-\delta t} \mu \langle Dv,Dz \rangle \di x \di t \le \\
				 \le \int_{\Omega} \frac{\mu^2(t_1,x)}{\bar{m}(x)} e^{-\delta t_1} \di x + \int_{t_1}^{t_2} \int_{\Omega} e^{-\delta t} \left( \bar{m}|Dv|^2 + \Theta \frac{\mu^2}{\bar{m}} + \Theta \bar{m}|Dz|^2 \right) \di x \di t.
			\label{muintpart}
		\end{multline}
%		From Proposition~\ref{phinl1} and in particular from~\eqref{part3}, one has
%		\begin{equation}
%			-\dot{\tilde{\Phi}}(t) \geq \frac1c \int_{\mathbb{R}^n} e^{-\delta t} \left( \bar{m} |Dv|^2 + \frac{\mu^2}{\bar{m}} + \bar{m}|Dz|^2 \right) \di x,
%		\end{equation}
		Plugging \eqref{nuovaint} into~\eqref{muintpart} yields
		\begin{equation}
				\int_{\Omega} \frac{\mu^2(t_2,x)}{\bar{m}(x)} e^{-\delta t_2} \di x \le  \int_{\Omega} \frac{\mu^2(t_1,x)}{\bar{m}(x)} e^{-\delta t_1} \di x + c \left( \tilde{\Phi}(t_1)  - \tilde{\Phi}(t_2)  \right),
			\label{muintpart2}
		\end{equation}
		where $c$ depends on $\eta, C_P$. 
		
		Fix now $t \in [1,T]$. There exists $\tau \in (t-1,t)$ such that
		\begin{equation*}
			\int_{\Omega} \frac{\mu^2(\tau,x)}{\bar{m}(x)} e^{-\delta \tau} \di x = \int_{t-1}^{t} \int_{\Omega} \frac{\mu^2(s,x)}{\bar{m}(x)} e^{-\delta s} \di x \di s \le c \left( \tilde{\Phi}(t-1)  - \tilde{\Phi}(t)  \right),
		\end{equation*}
		again by \eqref{nuovaint}. Then, using~\eqref{muintpart2} with $t_1 = \tau$ and $t_2 = t$,
		\begin{equation*}
				\int_{\Omega} \frac{\mu^2(t,x)}{\bar{m}(x)} e^{-\delta t} \di x \le \int_{\Omega} \frac{\mu^2(\tau,x)}{\bar{m}(x)} e^{-\delta \tau} \di x +c \left( \tilde{\Phi}(\tau)  - \tilde{\Phi}(t)  \right) \le c \left(\tilde{\Phi}(\tau) + \tilde{\Phi}(t-1) -2 \tilde{\Phi}(t)  \right),
		\end{equation*}
		whose right-hand side can be bounded, using \eqref{phidelta}, by
		\[
		2c\left( e^{\sigma} e^{-\sigma t} \left| \Phi(0) \right| + e^{-\sigma(T-t)} e^{-\delta T}  \left| \Phi(T) \right| \right).
		\]
		We have a similar result for $t \in [0,1]$, since by \eqref{muintpart2}
		\begin{equation*}
				\int_{\Omega} \frac{\mu^2(t,x)}{\bar{m}(x)} e^{-\delta t} \di x \le  \int_{\Omega} \frac{\mu_0^2(x)}{\bar{m}(x)} \di x + c \left( \tilde{\Phi}(0)  - \tilde{\Phi}(t)  \right),
		%	\label{muintpart3}
		\end{equation*}
		from which we obtain the thesis.
	\end{proof}
	
	\begin{prop}
		Suppose that \eqref{integrab} and \Sp holds, together with the following condition on $\mu$ and $Dv$:
		\begin{equation*}
			\sup_{t \in [0,T]} \left \| \mu(t,\cdot) \right \|_{L^\infty(\Omega)} + \sup_{t \in [0,T]} \left \| Dv(t,\cdot) \right \|_{L^\infty(\Omega)} \le \Theta,
			%\label{smallness}
		\end{equation*}
		where $\Theta$ is as in Proposition~\ref{phinl1}.
		Then, the following inequality holds:
		\begin{equation*}
			\left \| \sqrt{\bar{m}(\cdot)} |Dv|(t,\cdot) \right \|^2_{L^2(\Omega)} = \left \|\sqrt{\bar{m}(\cdot)} |Du(t,\cdot)-D\bar{u}(\cdot)| \right \|^2_{L^2(\Omega)} \le \lambda_3 e^{-(\sigma-\delta) t} + \lambda_4 e^{-(\sigma+\delta) (T-t)},
			%\label{nlexpprop2Dv}
		\end{equation*}
		for all $t \in [0,T]$, with
		$$\lambda_3=C|\Phi(0)|, \qquad \lambda_4=C |\Phi(T)| + \| \sqrt{\bar{m}}|Dv|^2(T,\cdot)\|^2_{L^2(\Omega)},$$
		where $\sigma$ is defined in Proposition~\ref{phinl1} and $C$ depends on $\eta, C_P, \Theta,\bar{u},C_f,\bar{m}$.
		\label{L2Dv}
	\end{prop}
	\begin{proof}
		Fix $j=1,...,n$. Differentiate the equation for $v$ with respect to $x_j$, multiply it by $\bar{m} v_j e^{-\delta t}$, where $v_j:=\frac{\partial v}{\partial x_j}$, to get
		\begin{multline*}
			-\bar{m} v_j \left( \partial_t v_j \right) e^{-\delta t} - \Delta v_j \bar{m} v_j e^{-\delta t} + \left( \langle Dv,D\bar{u} \rangle \right)_j \bar{m} v_j e^{-\delta t} - (f_m(\bar{m}) \mu)_j \bar{m} v_j e^{-\delta t} -\\
			- \left( \xi \mu^2 \right)_j \bar{m} v_j e^{-\delta t} + \bar{m} v_j \langle Dv,Dv_j \rangle e^{-\delta t} + \delta \bar{m} v_j^2 e^{-\delta t}= 0.
		\end{multline*}
		Now, we integrate with respect to $x$; this yields
		\begin{multline*}
			-\frac{d}{dt} \int_{\Omega} \frac{\bar{m}v_j^2 e^{-\delta t}}{2} \di x + \underbrace{\int_{\Omega} \frac{\delta}{2} \bar{m} v_j^2 e^{-\delta t} \di x}_{\geq 0} = \int_{\Omega} \Delta v_j \bar{m} v_j e^{-\delta t} + (f_m(\bar{m}) \mu)_j \bar{m} v_j e^{-\delta t} \di x + \\
			+ \int_{\Omega} \left( \xi \mu^2 \right)_j \bar{m} v_j e^{-\delta t} -  \bar{m} v_j \langle Dv,Dv_j \rangle  e^{-\delta t} - \left( \langle Dv,D\bar{u} \rangle \right)_j \bar{m} v_j e^{-\delta t} \di x,
		\end{multline*}
		therefore, integrating by parts and using the identity $D \bar m = - \bar m D \bar u $,
		\begin{multline*}
			-\frac{d}{dt} \int_{\Omega} \frac{\bar{m}v_j^2 e^{-\delta t}}{2} \di x  \le \int_{\Omega} e^{-\delta t} \left( - \bar{m} |Dv_j|^2 + \bar{m} \langle D\bar{u},Dv_j \rangle v_j - f_m(\bar{m}) \mu \bar{m} v_{jj} + \bar{m} \mu \bar{u}_j f_{m}(\bar{m}) v_j \right) \di x +  \\
			+ \int_{\Omega} e^{-\delta t} \left( - \xi \mu^2 \bar{m} v_{jj} + \xi \mu^2 \bar{m} \bar{u}_j v_j - \bar{m} \bar{u}_j v_j \langle Dv,D\bar{u} \rangle + \bar{m} \langle Dv,D\bar{u} \rangle v_{jj} - \bar{m} v_j \langle Dv,Dv_j \rangle \right) \di x.
		\end{multline*}
		Moving the first term on the right hand side to the left, and employing Young's inequality we get
		\begin{multline*}
			-\frac{d}{dt} \int_{\Omega} \frac{\bar{m}v_j^2 e^{-\delta t}}{2} \di x + \int_{\Omega} \bar{m}|Dv_j|^2 e^{-\delta t} \di x \le 
			\int_{\Omega} e^{-\delta t} \Big( \frac{\bar{m}|Dv_j|^2}{5} + \frac{5}{4} \bar{m} |D\bar{u}|^2 v_j^2 + \frac{5}{4} (f_m(\bar{m}))^2 \bar{m}{\mu^2} \\
			  \frac{\bar{m}}{5} \left| v_{jj} \right|^2 + (f_m(\bar{m}))^2 \bar{m} \frac{\mu^2}{2} + \frac{\bar{m} \bar{u}_j^2  v_j^2}{2} + \frac{5}{4} \xi^2 \Theta^2 \bar{m}\mu^2 + \frac{\bar{m}}{5} \left| v_{jj} \right|^2 + \frac{\xi^2}{2} \Theta^2 \bar{m} \mu^2 +  \frac{1}{2}\bar{u}_j^2\bar{m}v_j^2\\
			 + \frac{\bar{m}|Dv|^2}2 + \frac{\bar{m}v_j^2 |D\bar{u}|^2\bar{u}_j^2}{2} + \frac{\bar{m}}{5} \left| v_{jj} \right|^2  + \frac{5}{4} \bar{m} |D\bar{u}|^2|Dv|^2 + \frac{\bar{m}|Dv_j|^2}{5} + \frac{5}{4} \Theta \bar{m} |Dv|^2  \Big) \di x \\
			\le \int_{\Omega} e^{-\delta t} \bar{m}|Dv_j|^2 + C_1 \bar{m}|Dv|^2 e^{-\delta t} + C_2  \frac{\mu^2}{\bar{m}} e^{-\delta t} \di x,
		\end{multline*}
		where 
		\begin{equation*}
			C_1:=\frac{7}{2} \|D\bar{u}\|_{L^{\infty}(\Omega)}^2 + \frac {\|D\bar{u}\|_{L^{\infty}}^4 }{2} + \frac{1}{2} + \frac{5}{4} \Theta =C_1(\Theta,\bar{u})
		\end{equation*}
		and
		\begin{equation*}
			C_2:=\left(\frac{5}{4}+\frac12\right) C_f^2 \| \bar{m}\|_{L^{\infty}}^2  + \left(\frac{5}{4}+\frac12\right) C_f^2 \Theta^2 \|\bar{m}\|^2_{L^{\infty}}  =C_2(C_f,\bar{m},\Theta,\bar{u}).
		\end{equation*}
		This implies that
		\begin{equation*}
			-\frac{d}{dt} \int_{\Omega} \frac{\bar{m}v_j^2}{2} e^{-\delta t} \di x \le C \int_{\Omega} \left( \bar{m}|Dv|^2 + \frac{\mu^2}{\bar{m}} \right) e^{-\delta t} \di x,
		\end{equation*}
		where $C = C(\Theta,\bar{u},C_f,\bar{m}) :=\max\{ C_1,C_2 \}$. Then, after summing over $j=1,...,n$, the previous becomes (after replacing $C$ with $nC$)
		\begin{equation}
			-\frac{d}{dt} \int_{\Omega} \frac{\bar{m}|Dv|^2}{2} e^{-\delta t} \di x \le C \int_{\Omega} \left( \bar{m}|Dv|^2 + \frac{\mu^2}{\bar{m}} \right) e^{-\delta t} \di x.
			\label{mdv}
		\end{equation}
		
		We follow now the same arguments of Proposition~\ref{L2mu} to get the desired result. First of all, we integrate~\eqref{mdv} on both sides between two generic times $t_1$ and $t_2$ ($t_1 < t_2$) to get
		\begin{equation*}
			\int_{\Omega} \frac{\bar{m}}{2} |Dv|^2(t_1,x) e^{-\delta t_1} \di x \le \int_{\Omega} \frac{\bar{m}}{2} |Dv|^2(t_2,x) e^{-\delta t_2} \di x + C \int_{t_1}^{t_2} \int_{\Omega} \left( \bar{m}|Dv|^2 + \frac{\mu^2}{\bar{m}} \right) e^{-\delta s} \di x \di s,
		\end{equation*}
		which, according to~\eqref{nuovaint} and \eqref{tildephiineq}, can be bounded from above as follows
		\begin{multline*}
			\le \int_{\Omega} \frac{\bar{m}}{2} |Dv|^2(t_2,x) e^{-\delta t_2} \di x + c\left( \tilde{\Phi}(t_1) - \tilde{\Phi}(t_2) \right) \le \\
			\le \int_{\Omega} \frac{\bar{m}}{2} |Dv|^2(t_2,x) e^{-\delta t_2} \di x + c \left( \tilde{\Phi}(0) e^{-\sigma t_1} - \tilde{\Phi}(T) e^{-\sigma(T-t_2)} \right).
			%\label{genmdv}
		\end{multline*}
		Fix now $t \in [0,T-1]$. The Mean Value Theorem guarantees the existence of some $\tau \in [t,t+1]$ such that
		\begin{equation*}
			\int_{\Omega} \frac{\bar{m}}{2} |Dv|^2(\tau,x) e^{-\delta \tau} \di x = \int_{t}^{t+1} \int_{\Omega} \frac{\bar{m}}{2} |Dv|^2(s,x) e^{-\delta s} \di x \di s.
		\end{equation*}
		Then, choosing $t_1=t, t_2=\tau$, we get, again in view of \eqref{nuovaint} and \eqref{tildephiineq}, 
		\begin{multline*}
			\int_{\Omega} \frac{\bar{m}}{2} |Dv|^2(t,x) e^{-\delta t} \di x \le \int_{\Omega} \frac{\bar{m}}{2} |Dv|^2(\tau,x) e^{-\delta \tau} \di x + c \left( \tilde{\Phi}(0) e^{-\sigma t} - \tilde{\Phi}(T) e^{-\sigma(T-\tau)} \right) \le \\
			\le \int_t^{t+1} \int_{\Omega} \frac{\bar{m}}{2} |Dv|^2(s,x) e^{-\delta s} \di x + c \left( \tilde{\Phi}(0) e^{-\sigma t} - \tilde{\Phi}(T) e^{-\sigma(T-\tau)} \right) \le \\
			\le c \left(\tilde{\Phi}(t) - \tilde{\Phi}(t+1)+ \tilde{\Phi}(0) e^{-\sigma t} - \tilde{\Phi}(T) e^{-\sigma(T-\tau)} \right) \\ \le c \left( 2\tilde{\Phi}(0) e^{-\sigma t} + |\tilde{\Phi}(T)| e^{-\sigma(T-t-1)} + |\tilde{\Phi}(T)|e^{-\sigma(T-\tau)} \right) \le \\ c\left( 2\tilde{\Phi}(0) e^{-\sigma t} + 2e^\sigma |\tilde{\Phi}(T)| e^{-\sigma(T-t)} \right).
		\end{multline*}
		Finally, if $t \in [T-1,T]$, we choose $t_1=t, t_2=T$ to get (note that $1 \le e^\sigma e^{-\sigma(T-t)}$)
		\begin{multline*}
			\int_{\Omega} \frac{\bar{m}}{2} |Dv|^2(t,x) e^{-\delta t} \di x \le \int_{\Omega} \frac{\bar{m}}{2} |Dv|^2(T,x) e^{-\delta T} \di x + c \left( \tilde{\Phi}(0) e^{-\sigma t} - \tilde{\Phi}(T)\right) \\
			\le e^\sigma e^{-\delta T-\sigma(T-t)} \int_{\Omega} \frac{\bar{m}}{2} |Dv|^2(T,x)  \di x + c \tilde{\Phi}(0) e^{-\sigma t} + c e^\sigma |\tilde{\Phi}(T)| e^{-\sigma(T-t)}.
		\end{multline*}
		Putting all together, we deduce that
		\[
			\left \| \sqrt{\bar{m}}|Dv|(t,\cdot) \right\|^2_{L^2(\Omega)} e^{-\delta t} \le 2c|\tilde{\Phi}(0)| e^{-\sigma t} +2ce^\sigma |\tilde{\Phi}(T)| e^{-\sigma(T-t)}+e^\sigma e^{-\delta T-\sigma(T-t)}\left\| \sqrt{\bar{m}}|Dv|(T,\cdot) \right\|^2_{L^2(\Omega)}
		\]
		from which we derive the conclusion, after recalling that $\tilde{\Phi}(t) =e^{-\delta t} \Phi(t)$.
	\end{proof}
	
%	We have eventually established the desired weighted $L^2$ turnpike property also in the non-linear setting, with some hypothesis on $\mu$, i.e. the difference between (the) optimal solution $m$ to the dynamic minimization problem and the solution $\bar{m}$ to the stationary one.

\section{$L^\infty$ turnpike a priori estimates} \label{nlturnlinf}

	The goal of this section is to prove Theorem \ref{apriori}. To this aim, we need to improve the $L^2$ decay estimates of the previous Section into $L^\infty$ ones; this will require some technical work, and some further conditions on the initial / terminal data.
	
	Throughout the section, $v, \mu, \Phi$ are as in the previous one; the constants $\sigma, \Theta$, $\lambda_1, \lambda_2, \lambda_3, \lambda_4$ are those appearing in Propositions \ref{phinl1}, \ref{phinl2}, \ref{L2Dv}. Finally, we set
	$$\sigma_1=\frac{\sigma-\delta}{n+1}, \qquad \sigma_2=\frac{\sigma+\delta}{n+1}.$$
	
	\begin{rmk} \label{nurmk3}
		As in Section~\ref{nlturnl2}, we specify that all the estimates in this section hold also for solutions to~\eqref{nupde} for all $\nu \in [0,1]$. 
	\end{rmk}

	\begin{prop} \label{propfix1}
		Suppose that the assumptions of Theorem \ref{apriori} are in force. Then, there exist $\bar{\eps} \in (0,\frac{1}{2}) $ and $\bar{\lambda} > 0$, depending on $\eta,\bar{m},C_f$, but independent of $T$, such that the following is true: for all $\gamma \le \bar{\gamma}$, $\eps \le \bar \eps$, if
		\begin{equation}
			\| \mu(t,\cdot)\|_{L^{\infty}(\Omega)} \le \eps \left( e^{-\sigma_1 t} + e^{-\sigma_2(T-t)} \right) \qquad \forall t \in [0,T],
			\label{epsmu}
		\end{equation}
		\begin{equation}
			\|\mu_0(\cdot)\|_{L^{\infty}(\Omega)} + \| v_T(\cdot) \|_{W^{1,\infty}(\Omega)} < \gamma
			\label{cond1}
		\end{equation}
		%and
		%\begin{equation}
			%\left\| \frac{\mu_0}{\sqrt{\bar{m}}} \right\|_{L^{2}(\Omega)} + \left \|\sqrt{\bar{m}} |Dv|(T,\cdot)\right \|_{L^2(\Omega)} < \lambda
			%\label{cond2}
		%\end{equation}
		then any solution $(v,\mu)$ to the non-linear system~\eqref{nlsys} satisfies the following estimates:
		\begin{equation}
			\|\mu(t,\cdot)\|_{L^{\infty}(\Omega)} < \Theta \qquad \|v(t,\cdot)\|_{W^{1,\infty}(\Omega)} < \Theta,
			\label{muDvestimates}
		\end{equation}
		for all $t \in [0,T]$, where $\Theta$ is as in Proposition~\ref{phinl1}.
	\end{prop}

		We divide the proof of Proposition~\ref{propfix1} into two Lemmas.
	\begin{lem} \label{lemv}
		Suppose that $(v,\mu)$ is a solution to~\eqref{nlsys} and that the assumptions of Proposition~\ref{propfix1} are in force. Then,	the following estimate holds:
		\begin{equation*}
			\|v(t,\cdot)\|_{L^{\infty}(\Omega)} \le \| v_T(\cdot) \|_{L^{\infty}(\Omega)} + 2 \eps \left( \frac{1}{\sigma_1} + \frac{1}{\sigma_2} \right) C_f
			%\label{estv}
		\end{equation*}
		for all $t \in [0,T]$.
	\end{lem}
	\begin{proof}
		Let $\tilde{v}(t,x):=v(t,x)e^{-\delta t}$. First, we have to establish the equation satisfied by $\tilde{v}$:
		\begin{equation}
			- \partial_t \tilde{v} - \Delta \tilde{v} + \langle D\tilde{v},D\bar{u} \rangle - f_m(\bar{m}) \mu e^{-\delta t} -\xi \mu^2 e^{-\delta t} + \frac{|Dv|^2}{2} e^{-\delta t} = 0.
			\label{eqtildev}
		\end{equation}
		Then, consider $\tilde{\Gamma}$ solution to
		\begin{equation*}
			\begin{cases}
				\partial_s \tilde{\Gamma} - \Delta \tilde{\Gamma} - \text{div}(\tilde{\Gamma} D\bar{u}) = 0,& \text{ on } (t,T) \times \Omega \\
				\tilde{\Gamma}(t,\cdot) = \delta_x,& \text{ on } \{ s=t \} \times \Omega \\
			\end{cases}
		\end{equation*}
		multiply it by $\tilde v$ and then multiply~\eqref{eqtildev} by $\tilde{\Gamma}$
		After integration with respect to time and space, one gets
		\begin{multline*}
			\int_t^T \int_{\Omega} - \partial_s \tilde{v} \tilde{\Gamma} - \left( \Delta \tilde{v} \right) \tilde{\Gamma} + \langle D\tilde{v},D\bar{u} \rangle \tilde{\Gamma} - f_m(\bar{m}) \mu e^{-\delta s} \tilde{\Gamma} - \xi \mu^2 e^{-\delta s} \tilde{\Gamma} \, \di y \di s + \\
			+ \int_t^T \int_{\Omega} \frac{1}{2} |Dv|^2 e^{-\delta s} \tilde{\Gamma} - \partial_s \tilde{\Gamma} \tilde{v} + \Delta \tilde{\Gamma} \tilde{v} + \text{div}(\tilde{\Gamma} D\bar{u}) \tilde{v}\, \di y \di s = 0.
		\end{multline*}
		Integration by parts in space yields
		\begin{equation*}
			\int_{\Omega} - \tilde v \tilde{\Gamma} \bigg |_{t}^{T} \di y = \int_t^T \int_{\Omega} f_m(\bar{m})\mu e^{-\delta s} \tilde{\Gamma} + \xi \mu^2 e^{-\delta s} \tilde{\Gamma} - \frac{1}{2} |Dv|^2 \tilde{\Gamma} e^{-\delta s} \di y \di s.
		\end{equation*}
		Using the initial condition on $\tilde{\Gamma}$, the fact that $\tilde{\Gamma}$ is a flow of probability measures and the square norm of the gradient is positive, we derive that
		$$\tilde{v}(t,x) \le \int_{\Omega} \tilde v(T,y) \tilde{\Gamma}(T,y) \di y + \int_t^T \int_{\Omega} f_m(\bar{m})\mu e^{-\delta s} \tilde{\Gamma} + \xi \mu^2 e^{-\delta s} \tilde{\Gamma} \, \di y \di s \le$$
		$$\le \|\tilde{v}_T(\cdot)\|_{L^{\infty}(\Omega)} + 2C_f \int_t^T \|\mu(s,\cdot)\|_{L^{\infty}(\Omega)} e^{-\delta s} \di s,$$
		where we used the fact that $ \| \mu(s,\cdot) \|_{L^{\infty}(\Omega)} \le 1$ and the Hölder's inequality. Now, we use assumption~\eqref{epsmu} to get
		\begin{equation}
			\tilde{v}(t,x) \le \|\tilde{v}_T(\cdot)\|_{L^{\infty}(\Omega)} +2 C_f \int_t^T \eps \left( e^{-\sigma_1 s} + e^{-\sigma_2(T-s)} \right) e^{-\delta s} \di s.
			\label{side1}
		\end{equation}
		To get the other side of the inequality, we define a second flow of probability, called $\bar{\Gamma}$, such that
		\begin{equation*}
			\begin{cases}
				\partial_s \bar{\Gamma} - \Delta \bar{\Gamma} - \text{div}(\bar{\Gamma}(D\bar{u} + Dv)) = 0,& \text{ on } (t,T) \times \Omega \\
				\bar{\Gamma}(t,\cdot) = \delta_x,& \text{ on } \{ s=t \} \times \Omega \\
			\end{cases}
		\end{equation*}
		and we do the same computations as before. We get
		\begin{equation*}
			\tilde{v}(t,x)= \int_{\Omega} \tilde{v}_T(y) \bar{\Gamma}(T,y) \di y + \int_t^T \int_{\Omega} f_m(\bar m) \mu e^{-\delta s}  \bar{\Gamma} + \xi \mu^2 e^{-\delta s}  \bar{\Gamma} - \frac{1}{2} |Dv|^2  e^{-\delta s}  \bar{\Gamma} + \langle Dv,D\tilde{v} \rangle \bar{\Gamma} \, \di y \di s.
		\end{equation*}
		Recall that $\tilde{v}(t,x)=v(t,x) e^{-\delta t}$, hence from the previous equation we get
		\begin{equation*}
			\tilde{v}(t,x)= \int_{\Omega} \tilde{v}_T(y) \bar{\Gamma}(T,y) \di y + \int_t^T \int_{\Omega} f_m(\bar m) \mu e^{-\delta s}  \bar{\Gamma} + \xi \mu^2 e^{-\delta s}  \bar{\Gamma} + \frac{1}{2} |Dv|^2 e^{-\delta s} \bar{\Gamma} \, \di y \di s \geq
		\end{equation*}
		\begin{equation*}
			\geq \int_{\Omega} \tilde{v}_T(y) \bar{\Gamma}(T,y) \di y + \int_t^T \int_{\Omega} f_m(\bar m) \mu e^{-\delta s}  \bar{\Gamma} + \xi \mu^2 e^{-\delta s}  \bar{\Gamma} \di y \di s.
		\end{equation*}
		Changing all signs yields
		\begin{equation*}
			-\tilde{v}(t,x) \le - \int_{\Omega} \tilde{v}_T(y) \bar{\Gamma}(T,y) \di y + \int_t^T \int_{\Omega} - f_m(\bar m) \mu e^{-\delta s}  \bar{\Gamma} - \xi \mu^2 e^{-\delta s}  \bar{\Gamma} \, \di y \di s \le
		\end{equation*}
		\begin{equation}
			\le \|\tilde{v}_T(\cdot)\|_{L^{\infty}(\Omega)} + 2C_f \int_t^T \eps \left( e^{-\sigma_1 s} + e^{-\sigma_2(T-s)} \right) e^{-\delta s} \di s.
			\label{side2}
		\end{equation}
		Putting~\eqref{side1} and~\eqref{side2} together, we finally get
		\begin{equation} \label{bothsides}
			\|\tilde{v}(t,\cdot)\|_{L^{\infty}(\Omega)} \le \|\tilde{v}_T(\cdot)\|_{L^{\infty}(\Omega)} + 2C_f \int_t^T \eps \left( e^{-\sigma_1 s} + e^{-\sigma_2(T-s)} \right) e^{-\delta s} \di s,
		\end{equation}
		from which, by definition of $\tilde{v}$, we have
		\begin{equation*}
			\|v(t,\cdot)\|_{L^{\infty}(\Omega)} \le \|v_T(\cdot)\|_{L^{\infty}(\Omega)} e^{-\delta(T-t)}
			+2C_f \int_t^T \eps \left( e^{-\sigma_1 s} + e^{-\sigma_2(T-s)} \right) e^{-\delta (s-t)} \di s.
		\end{equation*}
		Since $s \ge t$, we have that $e^{-\delta(s-t)} \le 1$, hence 
		\begin{equation*}
			\|v(t,\cdot)\|_{L^{\infty}(\Omega)} \le \|v_T(\cdot)\|_{L^{\infty}(\Omega)} + 2 \eps \left( \frac{1}{\sigma_1} + \frac{1}{\sigma_2} \right) C_f,
		\end{equation*}
		which is the thesis.
	\end{proof}
	Now, let us define 
	\begin{equation}C_v := \| v_T(\cdot) \|_{L^{\infty}(\Omega)} +2 \eps \left( \frac{1}{\sigma_1} + \frac{1}{\sigma_2} \right) C_f.
		\label{Cv}
	\end{equation}
	It's clear that it tends to $0$ as both $\| v_T(\cdot) \|_{L^{\infty}(\Omega)}$ and $\eps$ approach $0$.
	Then, we have to get a "smallness" estimate also for $Dv$. Since the equation for $v$ is quadratic in the gradient variable, we can linearize it by using the so-called Cole-Hopf transform, i.e. by defining
	\begin{equation*}
		w(t,x):=e^{-\frac{v(t,x)}{2}}.
		%\label{colehopf}
	\end{equation*}
	Now, we prove the following
	\begin{lem}	\label{lemDw}
		Suppose that $(v,\mu)$ is a solution to~\eqref{nlsys} and that the assumptions of Proposition~\ref{propfix1} are in force. Then,	the following estimate holds:
		\begin{equation*}
		\|Dv(t,\cdot)\|_{L^{\infty}(\Omega)} \le 4 e^{\frac{C_v}{2}} \| Dv_T(\cdot) \|_{L^{\infty}(\Omega)} + 2 e^{\frac{C_v}{2}} C' \left( \eps + \eps^2 + \delta C_v + C_v \right),
			%\label{estDw}
		\end{equation*}
		for all $t \in [0,T]$, $C_v$ is as in~\eqref{Cv} and $C'$ depends on $\bar u, C_v$.
	\end{lem}
	\begin{proof}
		From the definition of $w$ we derive that $v=-2 \text{ log}(w)$, hence by the chain rule,
		\begin{equation*}
		%\label{derivatives}
		\begin{aligned}
		& \partial_t v = -\frac{2}{w} \partial_t w, \\
		& \partial_{x_i} v = - \frac{2}{w} \partial_{x_i} w, \\
		& \partial^2_{x_i x_j} v= -\frac{2}{w}\partial^2_{x_i x_j} w + \frac{2}{w^2} \partial_{x_i} w \cdot \partial_{x_j} w, \\
		&\Delta v = - \frac{2}{w} \Delta w + \frac{2}{w^2} |Dw|^2 \\
		&Dv= - \frac{2}{w} Dw.
		\end{aligned}
		\end{equation*}
		After substituting the previous expressions into the equation satisfied by $v$, we get that $w$ satisfies the following Cauchy problem
		\begin{equation*}
			\begin{cases}
				-\partial_t w - \Delta w + \langle Dw,D\bar{u} \rangle + \frac{w}{2} \left( f_m(\bar m) \mu + \xi \mu^2 - \delta v \right)= 0,& \text{ on } (0,T) \times \Omega \\
				w(T,\cdot)=e^{-\frac{v_T(\cdot)}{2}},& \text{ on } \Omega \\
			\end{cases}
			\label{eqforw}
		\end{equation*}
		We have shown in Lemma~\ref{lemv} that $\|v(t,\cdot)\|_{L^{\infty}(\Omega)} \le C_{v}$ for all $t$, with $C_v$ independent of $T$. 
		%Then, the estimate on $v$ implies also that $w$ is bounded, i.e. $\|w(t,\cdot)\|_{L^{\infty}(\Omega)} \le e^{\frac{C_v}{2}}$. \\
		We now turn our attention to the function $\tilde w := w - 1$, that we expect to be close to zero; indeed, since by Lemma~\ref{lemv} we can choose $C_v$ as small as we need (in particular smaller than $1$), we can derive the following estimates on $\tilde w$.
		$$e^{-\frac{C_v}{2}} \le w(t,x) \le e^{\frac{C_v}{2}}, \qquad \forall (t,x) \in [0,T] \times \Omega,$$
		hence
		$$e^{-\frac{C_v}{2}} - 1 \le w(t,x) - 1 \le e^{\frac{C_v}{2}} - 1, \qquad \forall (t,x) \in [0,T] \times \Omega.$$
		It's true that
		$$e^{\frac{C_v}{2}} - 1 \le C_v \quad \text{ and } \quad e^{-\frac{C_v}{2}} - 1 \ge -\frac{C_v}{2} \ge - C_v,$$
		provided that $C_v \le 1$. We conclude that
		$$-C_v \le w(t,x) - 1 \le C_v \Rightarrow \| \tilde w (t,\cdot) \|_{L^{\infty}(\Omega)} = \|w(t,\cdot) - 1\|_{L^{\infty}(\Omega)} \le C_v \qquad \forall t \in [0,T].$$
		%The idea is now to exploit this fact and to study the function $w(t,x)-1$, which can be morally made as close to $0$ as we need, and, since it differs from $w$ by a constant, it satisfies a similar equation as $w$.
		%Indeed, if we do the computations
		Such function $\tilde w$ satisfies the following equation:
		\begin{equation*}
			\begin{cases}
				-\partial_t \tilde{w} - \Delta \tilde{w} + \langle D\tilde{w},D\bar{u} \rangle + \frac{w}{2} \left( f_m(\bar m) \mu + \xi \mu^2 - \delta v \right) = 0,& \text{ on } (0,T) \times \Omega \\
				\tilde{w}(T,\cdot)=e^{-\frac{v_T(\cdot)}{2}} - 1,& \text{ on } \Omega. \\
			\end{cases}
			\label{eqfortildew}
		\end{equation*}
		We can apply Lemma~\ref{app2} with $H=-D\bar{u}$ and $g=-\frac{w}{2} \left( f_m(\bar m) \mu + \xi \mu^2 - \delta v \right)$ and we get the following estimate:
		\begin{equation}
			\begin{split}
			\underset{t \in [0,T]}{\sup} \|D\tilde{w}(t,\cdot)\|_{L^{\infty}(\Omega)} &\le 2 \| D \tilde w_T(\cdot) \|_{L^{\infty}(\Omega)} \\
			&+ C \left[ \| w \|_{L^{\infty}([0,T] \times \Omega)} \left( \| f_m \|_{L^{\infty}(\Omega)} \| \mu \|_{L^{\infty}([0,T] \times \Omega)} + \| \xi \|_{L^{\infty}(\Omega)} \| \mu \|^2_{L^{\infty}([0,T] \times \Omega)} \right)\right] \\
			&+C \left( \delta \| w \|_{L^{\infty}([0,T] \times \Omega)} \| v \|_{L^{\infty}([0,T] \times \Omega)}+ \| \tilde w \|_{L^{\infty}([0,T] \times \Omega)} \right) \le \\
			&\le 2 \| D \tilde w_T(\cdot) \|_{L^{\infty}(\Omega)} + C' \left( \eps + \eps^2 + \delta C_v + C_v \right),
			\end{split}
			\label{estapp2}
		\end{equation}
		for a possibly different constant $C'$, which depends also on $C_f$.\\
		Recall that $D\tilde{w}=Dw$, and that $Dv=-2 e^{\frac{v}{2}} Dw$. Estimate~\eqref{estapp2} can be rewritten as follows:
		\begin{equation*}
			\underset{t \in [0,T]}{\sup} \|Dv(t,\cdot)\|_{L^{\infty}(\Omega)} \le 4 e^{\frac{C_v}{2}} \| Dv_T(\cdot) \|_{L^{\infty}(\Omega)} + 2 e^{\frac{C_v}{2}} C'_H \left( \eps + \eps^2 + \delta C_v + C_v \right),
		\end{equation*}
		which is the thesis.
	\end{proof}
	We are ready to show Proposition~\ref{propfix1}.
	\begin{proof}[Proof of Proposition~\ref{propfix1}]
		Gathering all the estimates from Lemmas~\ref{lemv} and~\ref{lemDw} and assumption~\eqref{epsmu} and using hypothesis~\eqref{cond1}, we have that
		\begin{equation*}
			\| \mu(t,\cdot)\|_{L^{\infty}(\Omega)} \le \eps \left( e^{-\sigma_1 t} + e^{-\sigma_2(T-t)} \right) \le 2 \eps,
		\end{equation*}
		\begin{equation*}
			\begin{split}
			\|v(t,\cdot)\|_{L^{\infty}(\Omega)} &\le C_v = \| v_T(\cdot) \|_{L^{\infty}(\Omega)} + 2 \eps \left( \frac{1}{\sigma_1} + \frac{1}{\sigma_2} \right)C_f \le \\
			&\le \gamma +  2 \eps \left( \frac{1}{\sigma_1} + \frac{1}{\sigma_2} \right) C_f,
			\end{split}
		\end{equation*}
		and
		\begin{equation*}
			\begin{split}
			\|Dv(t,\cdot)\|_{L^{\infty}(\Omega)} &\le 4 e^{\frac{C_v}{2}} \| Dv_T(\cdot) \|_{L^{\infty}(\Omega)} + 2 e^{\frac{C_v}{2}} C'_H \left( \eps + \eps^2 + \delta C_v + C_v \right) \le \\ &\le 4 e^{\frac{C_v}{2}} \gamma + 2 e^{\frac{C_v}{2}} C'_H \left( \eps + \eps^2 + \delta C_v + C_v \right),
		\end{split}
		\end{equation*}
	for all $t \in [0,T]$.
	Since we can choose $\eps$, $\gamma$ as small as possible, we can establish that there exist $\bar \eps, \bar \gamma$, depending on $\Theta$, which depends only on $\eta,f,\bar m$, small in such a way that \eqref{muDvestimates} holds.
	\end{proof}
	We now state the core of the fixed point argument. First of all, we have to derive some useful estimates for $\Phi(0)$ and $\Phi(T)$, which are valid under the result of Proposition~\ref{propfix1}.
	\begin{lem}
		Under the assumptions of Proposition~\ref{propfix1}, one has the following estimate
		\begin{equation}
			|\Phi(0)| + |\Phi(T)| \le \left\| \frac{\mu_0}{\sqrt{\bar{m}}} \right\|_{L^{2}(\Omega)} + \left\| \frac{\mu_0}{\sqrt{\bar{m}}} \right\|^2_{L^{2}(\Omega)} + \frac{C_P}{2} \left\| \sqrt{\bar{m}} |Dv|(T,\cdot)\right\|^2_{L^2(\Omega)},
			\label{phi0T}
		\end{equation}
		where $C_P$ is the Poincar\'e-Wirtinger constant of $\bar m$.
		\label{lemphi}
	\end{lem}
	\begin{proof}
		Before starting, we notice that, by Proposition~\ref{propfix1}, we have that in particular
		\begin{equation}
			\|v(t,\cdot)\|_{L^{\infty}(\Omega)} < \Theta \le 1.
			\label{small}
		\end{equation}
		We start from $\Phi(0)$, by applying Holder's inequality:
		$$\Phi(0) = \int_{\Omega} \mu_0 \cdot v(0) + \frac{\mu_0^2}{\bar{m}} \di x \le \left( \int_{\Omega} \frac{\mu_0^2}{\bar{m}} \, \di x \right)^{\frac{1}{2}} \cdot \left( \int_{\Omega} \bar{m}\underbrace{v^2(0)}_{\le 1} \di x \right)^{\frac{1}{2}} + \int_{\Omega} \frac{\mu_0^2}{\bar{m}}\,  \di x \le \left( \int_{\Omega} \frac{\mu_0^2}{\bar{m}} \, \di x \right)^{\frac{1}{2}} + \int_{\Omega} \frac{\mu_0^2}{\bar{m}} \, \di x,$$
		where we used~\eqref{small}.
		Then, we observe that, if we define $\hat v:=v - \int_{\Omega} v$, since $\int_{\Omega} \mu = 0$, we have the following
		\begin{multline*}
			\Phi(T)=\int_{\Omega} \mu(T,x) v(T,x) + \frac{\mu^2(T,x)}{\bar{m}} \di x = \\ = \int_{\Omega} \mu(T,x) \hat v(T,x) \di x + \int_{\Omega} \mu(T,x)\left(\int_{\Omega} v(T,y) \di y \right)\di x + \int_{\Omega} \frac{\mu^2(T,x)}{\bar{m}} \di x  = \\ = \int_{\Omega} \mu(T,x) \hat v(T,x) +  \frac{\mu^2(T,x)}{\bar{m}} \di x \geq \int_{\Omega} - \frac{\mu^2(T,x)}{2\bar{m}} - \frac{\bar{m}\hat v^2(T,x)}{2} + \frac{\mu^2(T,x)}{\bar{m}} \di x \geq \\ \geq \int_{\Omega} \frac{\mu^2(T,x)}{2\bar{m}} - \frac{C_P}{2} \bar{m} |D \hat v|^2(T,x) \di x \geq -\int_{\Omega} \frac{C_P}{2} \bar{m} |Dv|^2(T,x) \di x,
	\end{multline*}
		where we used Young's inequality and we applied the weighted Poincaré inequality.
		To get the reverse inequalities, we use~\eqref{phidelta} and we obtain
		\begin{equation*}
			\Phi(0) \ge \Phi(T) e^{-(\sigma+\delta)(T-t)} \ge - e^{-(\sigma+\delta)(T-t)} \int_{\Omega} \frac{C_P}{2} \bar m |Dv|^2(T,x) \di x \ge -\int_{\Omega} \frac{C_P}{2} \bar m |Dv|^2(T,x) \di x,
		\end{equation*}
		and
		\begin{equation*}
			\Phi(T) \le \Phi(0) e^{-(\sigma-\delta)t} \le \left( \int_{\Omega} \frac{\mu_0^2}{\bar{m}} \, \di x \right)^{\frac{1}{2}} + \int_{\Omega} \frac{\mu_0^2}{\bar{m}} \, \di x.
		\end{equation*}
		Putting all the previous estimates together, we get~\eqref{phi0T}.
	\end{proof}
	\begin{rmk}
		Observe that the second estimate is true also without the assumptions of Proposition~\ref{propfix1}.
	\end{rmk}
	From Lemma~\ref{lemphi}, together with the results in Propositions~\ref{L2mu} and~\ref{L2Dv}, we deduce the following estimates for $\lambda_i$, with $i=1,...,4$.
	\begin{lem}	\label{lambdasmall}
		Under the assumptions of Proposition~\ref{propfix1}, the following (weighted) $L^2$ exponential-type turnpike properties hold:
		\begin{equation*}
			\left \| \frac{\mu(t,\cdot)}{\sqrt{\bar{m}(\cdot)}}\right \|^2_{L^2(\Omega)} \le \lambda_1 e^{-(\sigma-\delta) t} + \lambda_2 e^{-(\sigma+\delta) (T-t)}, \qquad \forall t \in [0,T],
			%\label{l2turnmusmall}
		\end{equation*}
		\begin{equation*}
			\left \| \sqrt{\bar{m}} |Dv|(t,\cdot) \right \|^2_{L^2(\Omega)} \le \lambda_3 e^{-(\sigma-\delta) t} + \lambda_4 e^{-(\sigma+\delta) (T-t)}, \qquad \forall t \in [0,T],
			%\label{l2turnDvsmall}
		\end{equation*}
		where $\sigma$ is as in Proposition~\ref{phinl1},
		\begin{align*}
			&\lambda_1 \le C \left( A + 2 A^2 + \frac{C_P}{2} B^2 \right),  &  &\lambda_3 \le C \left( A + A^2 + \frac{C_P}{2} B^2 \right),         \\
			&\lambda_2 \le C \left( A + A^2 + \frac{C_P}{2} B^2 \right), &  &\lambda_4 \le C \left( A + A^2 + \frac{C_P}{2} B^2 + B^2 \right),
		\end{align*}
		with $A:= \left\| \frac{\mu_0}{\sqrt{\bar{m}}} \right\|_{L^2(\Omega)}$ and $B:=\left\| \sqrt{\bar{m}}|Dv|(T,\cdot) \right\|_{L^2(\Omega)}$ and $C$ depends on $\bar{m},\bar{u},f$.
	\end{lem}
	\begin{proof}
		The estimates hold true since both $\|\mu(t,\cdot)\|_{L^{\infty}(\Omega)}$ and $\|Dv(t,\cdot)\|_{L^{\infty}(\Omega)}$ are smaller than $\Theta$ as a consequence of Proposition~\ref{propfix1}, hence we can apply Propositions~\ref{L2mu} and~\ref{L2Dv} together with Lemma~\ref{lemphi} to get the thesis.
	\end{proof}
	\begin{rmk}
		Observe that if $A$ and $B$ are small, then also $\lambda_i$ are small: when $A+B \rightarrow 0$, then $\lambda_i \rightarrow 0$, for all $i=1,...,4$. This is a key point to prove the following
		\label{rmk2}
	\end{rmk}
	\begin{prop} \label{propfix2}
		Suppose $(\nu,\mu)$ is a solution to~\eqref{nlsys}. Assume that
		\begin{equation}
			\|\mu(t,\cdot)\|_{L^{\infty}(\Omega)} \le \bar{\eps} \left( e^{-\sigma_1 t} + e^{-\sigma_2(T-t)} \right),
			\label{cond3}
		\end{equation}
		for all $t \in [0,T]$, where $\bar{\eps}$ comes from Proposition~\ref{propfix1}. Then, there exists $\tilde{\gamma} \le \bar{\gamma}$ such that if
			\begin{align}
			&\|\mu_0(\cdot)\|_{L^{\infty}(\Omega)} + \|v_T(\cdot)\|_{W^{1,\infty}(\Omega)} < \tilde{\gamma}, 
			\label{cond4} \\
			&\left\| \frac{\mu_0}{\sqrt{\bar{m}}} \right\|_{L^{2}(\Omega)} + \|\sqrt{\bar{m}} |Dv_T|(\cdot)\|_{L^2(\Omega)} < \tilde{\gamma},
			\label{cond5}
			\end{align}
		then
		\begin{equation}
			\|\mu(t,\cdot)\|_{L^{\infty}(\Omega)} \le \frac{\bar{\eps}}{2} \left( e^{-\sigma_1 t} + e^{-\sigma_2(T-t)} \right),
			\label{estmu}
		\end{equation}
		for all $t \in [0,T]$.
	\end{prop}
		\begin{proof}
			Assume first that $\tilde{\gamma}=\bar{\gamma}$. During the proof, this constant might be further reduced, but this will not change the applicability of the result in Proposition~\ref{propfix1}, since the latter holds for all $\gamma \le \bar{\gamma}$. In particular, from this observation and assumptions~\eqref{cond3} and~\eqref{cond4}, we have that
			$$\|\mu(t,\cdot)\|_{L^{\infty}(\Omega)} < \Theta \qquad \|Dv(t,\cdot)\|_{L^{\infty}(\Omega)} < \Theta,$$
			hence all the $L^2$ weighted estimates in Lemma~\ref{lambdasmall} hold and can be used. Moreover, observe that since $\Theta < 1$, the same inequality holds for $\|\mu(t,\cdot)\|_{L^{\infty}(\Omega)}$ and $\|Dv(t,\cdot)\|_{L^{\infty}(\Omega)}$. \\
			Now, recall that $\mu$ satisfies the second equation in~\eqref{nlsys}. We can apply Lemma~\ref{app1} with $F=D\bar{u}+Dv$ and $G=\bar{m}Dv$. \\
			Therefore,~\eqref{estapp1} implies that
			\begin{equation}
				\underset{t \in [0,T]}{\sup} \|\mu(t,\cdot)\|_{L^{\infty}(\Omega)} \le 2 \| \mu_0(\cdot) \|_{L^{\infty}(\Omega)} + C \left( \| \bar m Dv \|_{L^{\infty}((t-\tau,t); L^{n+1}(\Omega))} + \| \mu \|_{L^{\infty}((t-\tau,t);L^n(\Omega))} \right),
				\label{estmufin1}
			\end{equation}
			where $\tau=\frac{1}{16 (\|D\bar{u}\|_{L^{\infty}(\Omega)}+1)^2 C_2(n)^2}$, the constant $C$ depends only on $\| D \bar u \|_{L^{\infty}(\Omega)}, \tau, n$ and $C_2(n):=C_2(n,1)$ comes from Lemma~\ref{lemfix1} below. Let us estimate the terms in brackets. \\
			First of all, we have that
			\begin{equation} \label{Dvinterpol}
				\begin{split}
				\| \bar m Dv \|_{L^{\infty}((t-\tau,t);L^{n+1}(\Omega))} &= \underset{s \in (t-\tau,t)}{\sup} \left( \int_{\Omega} |\bar m|^{n+1} |Dv|^{n+1}(s,y) \di y \right)^{\frac{1}{n+1}} \\ 
				&\le \| \bar m \|_{L^{\infty}(\Omega)}^{\frac{n}{n+1}} \| Dv \|_{L^{\infty}([0,T] \times \Omega)}^{\frac{n-1}{n+1}} \underset{s \in (t-\tau,t)}{\sup} \left( \int_{\Omega} \bar m |Dv|^2(s,y) \di y \right)^{\frac{1}{n+1}} \\
				&\le \| \bar m \|_{L^{\infty}(\Omega)}^{\frac{n}{n+1}} \| Dv \|_{L^{\infty}([0,T] \times \Omega)}^{\frac{n-1}{n+1}} \underset{s \in (t-\tau,t)}{\sup} \left( \lambda_3 e^{-(\sigma-\delta)s} + \lambda_4 e^{-(\sigma+\delta)(T-s)} \right)^{\frac{1}{n+1}} \\
				&\le \| \bar m \|_{L^{\infty}(\Omega)}^{\frac{n}{n+1}} \| Dv \|_{L^{\infty}([0,T] \times \Omega)}^{\frac{n-1}{n+1}} \underset{s \in (t-\tau,t)}{\sup} \left( \lambda_3^{\frac{1}{n+1}} e^{-\sigma_1 s} + \lambda_4^{\frac{1}{n+1}} e^{-\sigma_2(T-s)} \right) \\
				&\le \| \bar m \|_{L^{\infty}(\Omega)}^{\frac{n}{n+1}} \| Dv \|_{L^{\infty}([0,T] \times \Omega)}^{\frac{n-1}{n+1}} \left( \lambda_3^{\frac{1}{n+1}} e^{-\sigma_1(t-\tau)} + \lambda_4^{\frac{1}{n+1}} e^{-\sigma_2(T-t)} \right) \\
				&= \| \bar m \|_{L^{\infty}(\Omega)}^{\frac{n}{n+1}} \| Dv \|_{L^{\infty}([0,T] \times \Omega)}^{\frac{n-1}{n+1}} \left( \lambda_3^{\frac{1}{n+1}} e^{\sigma_1 \tau} e^{-\sigma_1 t} + \lambda_4^{\frac{1}{n+1}} e^{-\sigma_2(T-t)} \right).
				\end{split}
			\end{equation}
			Now we proceed with the other term in brackets.
			\begin{equation} \label{muinterpol}
				\begin{split}
					\| \mu \|_{L^{\infty}((t-\tau,t);L^n(\Omega))} &= \underset{s \in (t-\tau,t)}{\sup} \left( \int_{\Omega} |\mu|^n(s,y) \di y \right)^{\frac{1}{n}} \\
					&\le \| \mu \|_{L^{\infty}([0,T] \times \Omega)}^{\frac{n-2}{n}} \| \bar m \|^{\frac{1}{n}}_{L^{\infty}(\Omega)} \underset{s \in (t-\tau,t)}{\sup} \left( \int_{\Omega} \frac{\mu^2(s,y)}{\bar m(y)} \di y \right)^{\frac{1}{n}} \\
					&\le \| \mu \|_{L^{\infty}([0,T] \times \Omega)}^{\frac{n-2}{n}} \| \bar m \|^{\frac{1}{n}}_{L^{\infty}(\Omega)} \underset{s \in (t-\tau,t)}{\sup} \left( \lambda_1 e^{-(\sigma-\delta)s} + \lambda_2 e^{-(\sigma+\delta)(T-s)} \right)^{\frac{1}{n}} \\
					&\le \| \mu \|_{L^{\infty}([0,T] \times \Omega)}^{\frac{n-2}{n}} \| \bar m \|^{\frac{1}{n}}_{L^{\infty}(\Omega)} \underset{s \in (t-\tau,t)}{\sup} \left( \lambda_1^{\frac{1}{n}} e^{-\frac{\sigma-\delta}{n}s} + \lambda_2^{\frac{1}{n}} e^{-\frac{\sigma+\delta}{n}(T-s)} \right) \\
					&\le \| \mu \|_{L^{\infty}([0,T] \times \Omega)}^{\frac{n-2}{n}} \| \bar m \|^{\frac{1}{n}}_{L^{\infty}(\Omega)} \left( \lambda_1^{\frac{1}{n}} e^{-\sigma_1(t-\tau)} + \lambda_2^{\frac{1}{n}} e^{-\sigma_2(T-t)} \right) \\
					&=\| \mu \|_{L^{\infty}([0,T] \times \Omega)}^{\frac{n-2}{n}} \| \bar m \|^{\frac{1}{n}}_{L^{\infty}(\Omega)} \left( \lambda_1^{\frac{1}{n}} e^{\sigma_1 \tau} e^{-\sigma_1 t} + \lambda_2^{\frac{1}{n}} e^{-\sigma_2(T-t)} \right) \\
				\end{split}
			\end{equation}

			Plugging estimates~\eqref{Dvinterpol} and~\eqref{muinterpol} into~\eqref{estmufin1} we get
			\begin{equation*}
				\begin{split}
				\underset{t \in [0,T]}{\sup} \|\mu(t,\cdot)\|_{L^{\infty}(\Omega)} &\le 2 \| \mu_0 \|_{L^{\infty}(\Omega)} + C \| \bar m \|_{L^{\infty}(\Omega)}^{\frac{n}{n+1}} \| Dv \|_{L^{\infty}([0,T] \times \Omega)}^{\frac{n-1}{n+1}} \left( \lambda_3^{\frac{1}{n+1}} e^{\sigma_1 \tau} e^{-\sigma_1 t} + \lambda_4^{\frac{1}{n+1}} e^{-\sigma_2(T-t)} \right)  \\
				&+ C \| \mu \|_{L^{\infty}([0,T] \times \Omega)}^{\frac{n-2}{n}} \| \bar m \|^{\frac{1}{n}}_{L^{\infty}(\Omega)} \left( \lambda_1^{\frac{1}{n}} e^{\sigma_1 \tau} e^{-\sigma_1 t} + \lambda_2^{\frac{1}{n}} e^{-\sigma_2(T-t)} \right) \\ 
				&\le 2 \| \mu_0 \|_{L^{\infty}(\Omega)} + C \| \bar m \|_{L^{\infty}(\Omega)}^{\frac{n}{n+1}} \left( \lambda_3^{\frac{1}{n+1}} e^{\sigma_1 \tau} e^{-\sigma_1 t} + \lambda_4^{\frac{1}{n+1}} e^{-\sigma_2(T-t)} \right) \\
				&+C \| \bar m \|^{\frac{1}{n}}_{L^{\infty}(\Omega)} \left( \lambda_1^{\frac{1}{n}} e^{\sigma_1 \tau} e^{-\sigma_1 t} + \lambda_2^{\frac{1}{n}} e^{-\sigma_2(T-t)} \right)
				\end{split}
			\end{equation*}
			%\TODO{Since $\lambda_i$ are arbitrarily small if $\bar \lambda$ is chosen small enough in assumptions~\eqref{cond3} and~\eqref{cond4}, for all $i=1,...,4$, we find that there exists a $\tilde \lambda > 0$ (possibly smaller than $\bar \lambda$) such that}
			
			Since $\lambda_i$, for all $i=1,...,4$, go to $0$ as $\gamma$ goes to $0$, and the same happens for $\|\mu_0\|_{L^{\infty}(\Omega)}$, we find that there exists $\tilde{\gamma} > 0$ (possibly smaller than $\bar{\gamma}$) such that 
			$$\|\mu(t,\cdot)\|_{L^{\infty}(\Omega)} \le \frac{\bar{\eps}}{2} \left( e^{-\sigma_1 t} + e^{-\sigma_2(T-t)} \right) \qquad \forall t \in [0,T],$$
			i.e.~\eqref{estmu}.
		\end{proof}
		
		\section{Fixed point theorem and existence of solutions} \label{fix2}
		
Throughout this section, we restrict ourselves to the periodic case $\Omega=\T^n$. We are going to build a solution $(u^T,m^T)$ to~\eqref{dynsys} which satisfies the turnpike property with respect to a stationary solution $(\bar u, \bar m)$ to~\eqref{statsys}, for every $T > 0$. Despite the estimates shown in Sections~\ref{nlturnl2} and~\ref{nlturnlinf} hold also for $\Omega=\R^n$, we restrict ourselves here to the torus to avoid further technicalities that come from the non compact nature of $\R^n$. We point out that solutions  $(u^T,m^T)$  are not \textit{a priori} unique, but our claim is that we are able to find at least one solution which is exponentially close to $(\bar u, \bar m)$. 

Then, in the second part of the present section, we will prove also an existence result for the infinite horizon problem
	\begin{equation} \label{infhor}
		\begin{cases}
			-\partial_t u - \Delta u + \frac{|Du|^2}{2} - f(x,m) + \delta u = 0,& \text{ on } (0,+\infty) \times \T^n \\
			\partial_t m - \Delta m - \text{div}(m Du)=0,& \text{ on } (0,+\infty) \times \T^n \\
			m(0,x)=m_0(x), \qquad u \text{ bounded }, & \text{ on } \T^n,
		\end{cases}
		\end{equation}
such that $m$ is exponentially close to $\bar m$ for large $t$. %The essential feature to get this result is that the turnpike constant $\bar\eps$ in~\eqref{estmu} does not depend on $T$, thereby enabling to pass to the limit as $T \to +\infty$.

\medskip

Our first theorem (that is Theorem \ref{existenceintro} in the introduction) guarantees the existence of a solution to~\eqref{dynsys} which satisfies the exponential turnpike decay, for each $T > 0$, under some smallness assumptions on the initial / terminal data and \Sp, that allow to apply the estimates in Section~\ref{nlturnlinf}. 
%Now, we can finally state the existence result about the (discounted) mean field system	
%		\begin{equation} \label{MFGsys}
%			\begin{cases}
%				\partial_t m - \Delta m - \text{div}(mDu)=0,& \text{ in } [0,T] \times \mathbb{R}^n \\
%				-\partial_t u - \Delta u + \frac{|Du|^2}{2} + \delta u - F'(m) - V(x) = 0,& \text{ in } [0,T] \times \mathbb{R}^n \\
%				m(0,\cdot)=m_0(\cdot), \qquad u(T,\cdot)=u_T(\cdot),& \text{ in } \mathbb{R}^n.
%			\end{cases}
%		\end{equation}
		\begin{thm} \label{existence}
		Let $\Omega=\T^n$ and $T>0$ be fixed. Suppose that $(\bar u, \bar m)$ is a (classical) solution to~\eqref{statsys} and that \Sp holds. Then, there exist constants $\bar \eps > 0, \gamma > 0, \sigma > 0$ (which depend only on $\bar m, \bar u, \eta, C_f$ but not on the time horizon $T$) such that, if
		\begin{equation*}
					\delta < \sigma, \quad \| m_0(\cdot)-\bar m(\cdot) \|_{L^{\infty}(\T^n)} + \|u_T(\cdot) - \bar u(\cdot)\|_{W^{1,\infty}(\T^n)} < \gamma,
					%\label{ininf}
		\end{equation*}
		then there exists a (classical) solution $(u^T,m^T)$ to~\eqref{dynsys} satisfying the following properties:
				\begin{align} \label{turnpike}
				&\|m^T(t,\cdot)-\bar m(\cdot)\|_{L^{\infty}(\T^n)} \le \frac{\bar{\eps}}{2} \left( e^{-\sigma_1 t} + e^{-\sigma_2(T-t)} \right), \\
				&\left \| \sqrt{\bar m(\cdot)} \left| Du^T(t,\cdot) - D\bar u(\cdot) \right| \right \|_{L^2(\T^n)} \le \lambda_3 e^{-\sigma_1 t} + \lambda_4 e^{-\sigma_2(T-t)}, \label{turnpikeDuT}
			\end{align}
			for all $t \in [0,T]$, where $\sigma_1=\frac{\sigma-\delta}{n+1}$, $\sigma_2=\frac{\sigma+\delta}{n+1}$, $\lambda_3$ and $\lambda_4$ come from Lemma~\ref{lambdasmall}.
		\end{thm}
		
		The proof uses the standard \textit{Leray-Schauder fixed point theorem}:
		\begin{thm}
			Let $M>0$ and $X$ be a Banach space, and let $\mathcal{T}$ be a compact mapping of $\overline{B_M(0)} \times [0,1]$ into $X$ such that $\mathcal{T}(x,0)=0$ for all $x \in X$. Suppose that 
			\begin{equation*}\label{xM}
			\| x \|_X < M
			\end{equation*}
			for all $(x,\nu) \in \overline{B_M(0)} \times [0,1]$ satisfying $x = \mathcal{T}(x,\nu)$. Then, the mapping $\mathcal{T}(x,1)$ has a fixed point in $B_M(0)$.
			\label{fixedGT}
		\end{thm}
		
		The proof can be found for instance in \cite[Theorem 11.6]{GT77}. Note that in this reference the theorem is stated for $\mathcal{T}$ defined on the whole $X \times [0,1]$ (though the proof does not involve the evaluation of the operator outside $\overline{B_M(0)}$); it is possible anyway to extend $\mathcal{T}$ from $\overline{B_M(0)}$ to $X$ via a retraction (as in \cite[Lemma 11.7]{GT77}).

		\begin{proof}[Proof of Theorem~\ref{existence}]
			Let $T > 0$ be fixed. To simplify the notation, throughout the proof, we drop the superscript $T$; thus, the triple $(\rho,u,m)$ should be intended as $(\rho^T,u^T,m^T)$. The ambient Banach space is defined as follows:
			$$X:=\mathcal{C}([0,T] \times \T^n).$$
			We look for a solution in this space. We equip it with the norm
			\begin{equation}
				\triplenorm{\rho}_X := \sup_{t \in [0,T]} \left( \frac{\|\rho(t,\cdot)\|_{{\infty},\T^n}}{e^{-\sigma_1 t} + e^{-\sigma_2(T-t)}} \right),
				\label{norm3}
			\end{equation}
			where $\sigma_1=\frac{\sigma-\delta}{n+1}$ and $\sigma_2=\frac{\sigma+\delta}{n+1}$ and $\sigma$ comes from Proposition~\ref{phinl1}. \\
			Notice that $\triplenorm{\rho}_X$ and
			\begin{equation}
				\|\rho\|_X:=\sup_{t \in [0,T]} \|\rho(t,\cdot)\|_{L^{\infty}(\T^n)}
				\label{norm2}
			\end{equation} 
			are equivalent ($T$ is fixed). In particular, $(X,\triplenorm{\cdot}_X)$ is a Banach space, since $(X,\| \cdot \|_X)$ is.
			Therefore, to prove that a sequence is converging in $X$, we can use equivalently either the $\triplenorm{\cdot}_X$ norm or the $\| \cdot \|_X$ one.
			Now, let
			$$U:=\overline{B_{\bar{\eps}}(0)} = \left\{ \rho \in X : \triplenorm{\rho}_X\le  \bar{\eps} \right\},$$
			where $\bar{\eps}$ comes from Proposition~\ref{propfix1}. % We are going to focus our attention to the fixed point operator restricted to $U$.
			We define $\mathcal{F} :  U \times [0,1] \to X$ such that $\mathcal{F}(\rho,\nu)=\mu$, where $\mu$ is the solution to the following PDE system:
			\begin{equation}
				\begin{cases}
					-\partial_t v - \Delta v + \langle Dv,D\bar{u} \rangle - f(x,\bar{m}+ \nu \rho) + f(x,\bar{m}) + \delta v + \frac{|Dv|^2}{2} = 0,& \text{ on } (0,T) \times \T^n \\
					\partial_t \mu - \Delta \mu - \text{div}(\bar{m}Dv) - \text{div}(\mu D\bar{u}) - \text{div}(\mu Dv)=0,& \text{ on } (0,T) \times \T^n \\
					\mu(0,\cdot)=\nu \mu_0(\cdot) \qquad v(T,\cdot)=\nu v_T(\cdot).& \text{ on } \T^n,
				\end{cases}
				\label{pdefixed}
			\end{equation}
		where $\mu_0=m_0-\bar m$ and $v_T=u_T-\bar u$. \\
		Let us briefly describe how the operator acts. It takes $\rho \in U$, sends it to the solution $v$ to the first equation in~\eqref{pdefixed}, with the corresponding terminal condition, and then maps it to $\mu$, solution to the second equation of the same PDE system, with the related initial datum. Notice that the fixed points $\mu=\mathcal{F}(\mu,1)$ are precisely the solutions to~\eqref{nlsys}. \\
%		It is useful to take advantage of the quadratic term in $Dv$ by applying the standard \textit{Hopf-Cole transform}
%			\begin{equation} \label{hopfcole}
%				w(t,x):=e^{-\frac{v(t,x)}{2}}.
%			\end{equation}
%			In fact, the PDE system~\eqref{pdefixed} reads as a simpler system of coupled equations in the variables $(w,\mu)$, as follows:
%			\begin{equation}
%				\begin{cases}
%					-\partial_t w - \Delta w + \langle Dw,D\bar{u} \rangle + \frac{w}{2} \left( f(x,\bar{m}+\rho)- f(x,\bar{m}) \right) - 2 \delta w \text{ log}(w)=0,& \text{ on } (0,T) \times \T^n \\
%					\partial_t \mu - \Delta \mu - \text{div}\left (\bar{m}\frac{Dw}{w} \right) - \text{div} \left( \mu D\bar{u} \right) - \text{div} \left(\mu \frac{Dw}{w}\right)=0,& \text{ on } (0,T) \times \T^n \\
%					\mu(0,\cdot)=\mu_0(\cdot) \qquad w(T,\cdot)=e^{-\frac{v_T(\cdot)}{2}},& \text{ on } \T^n,
%				\end{cases}
%				\label{pdefixed2}
%			\end{equation}
			We need to prove that $\mathcal{F}$ satisfies the assumptions of Theorem~\ref{fixedGT}.
			\medskip
			
			\textbf{Step 1: $\mathcal{F}$ is continuous.} \\
			Let $\{ (\rho_k,\nu_k) \}_k$ be a sequence in $U \times [0,1]$ such that $(\rho_k,\nu_k) \to (\rho,\nu)$ as $k \to +\infty$, for some $\rho \in X$ and some $\nu \in [0,1]$, i.e.
			\begin{equation*}
				\| \rho_k - \rho \|_X \to 0 \qquad | \nu_k - \nu | \to 0,
			\end{equation*}
			as $k \to +\infty$. We have to prove that $\mu_k \to \mu$ in $X$, where $\mu_k:=\mathcal{F}(\rho_k,\nu_k)$ and $\mu:=\mathcal{F}(\rho,\nu)$. We can do it by using any of the equivalent norms~\eqref{norm3} and~\eqref{norm2}; in particular, we wish to prove that
			\begin{equation*} 
				%\label{mukmutoprove}
				\| \mu_k - \mu \|_X = \underset{t \in [0,T]}{\sup} \| \mu_k(t,\cdot) - \mu(t,\cdot) \|_{L^{\infty}(\T^n)} \to 0, \qquad \text{ for } k \to +\infty.
			\end{equation*}
			Recall that, by assumption, $\triplenorm{\rho_k} \le \bar \eps$, hence the same procedure described in Proposition~\ref{propfix1} allows us to say that
			$$\| Dv_k(t,\cdot) \|_{L^{\infty}(\T^n)} \le \Theta < 1,$$
			for all $t \in [0,T]$ and for all $k$.
			The same reasoning holds also for $\rho$ and $Dv$, since $\triplenorm{\rho} \leq  \bar \eps$. \\
			First of all, we define $\hat v_k:=v_k - v$. While computing the difference between the equations satisfied by $v_k$ and $v$, the following term involving $f$ appears:
			\begin{equation*}
				-f(x,\bar m + \nu_k \rho_k) + f(x,\bar m + \nu \rho) = - \xi_\nu(t,x) (\nu_k \rho_k- \nu \rho),
			\end{equation*}
			where $\xi_\nu(t,x)=\int_0^1 f_m(x,\bar m + \nu \rho + z(\nu_k \rho_k - \nu \rho)) \, \di z$. Hence, the equation satisfied by $\hat v_k$ is
			\begin{equation*}
				%\label{eqhatvk}
				\begin{cases}
					-\partial_t \hat v_k - \Delta \hat v_k + \langle D\hat v_k , D\bar u \rangle - \xi_{\nu} (\nu_k \rho_k - \nu \rho) + \delta \hat v_k + \frac{1}{2} \langle D \hat v_k , Dv_k + Dv \rangle = 0,& \text{ on } (0,T) \times \T^n, \\
					\hat v_k(T,\cdot) = (\nu_k - \nu) v_T(\cdot),& \text{ on } \T^n.
				\end{cases}
			\end{equation*}
			In the same way as~\eqref{bothsides} is obtained in Lemma~\ref{lemv}, we get the following estimate:
			\begin{equation*}
				\begin{aligned}
					\| \hat v_k (t,\cdot) \|_{L^{\infty}(\T^n)} & \le | \nu_k - \nu | \| v_T(\cdot) \|_{L^{\infty}(\T^n)} + \int_t^T \| \xi_\nu \|_X \cdot \| (\nu_k \rho_k - \nu \rho) (s,\cdot) \|_{L^{\infty}(\T^n)} \di s \\
					&\le | \nu_k - \nu |( \| v_T(\cdot) \|_{L^{\infty}(\T^n)} + T C_f \| \rho_k \|_X ) + T C_f \| \rho_k - \rho \|_X,
				\end{aligned}
			\end{equation*}
			for all $t \in [0,T]$. Since the sequence $\rho_k$ converges to $\rho \in X$, $\nu_k \to \nu$ in $[0,1]$ by hypothesis and $v_T$ are bounded, we have that
		 	$$\| v_k - v \|_X \to 0, \qquad \text{ for } k \to +\infty.$$
		 	Let us define $F_k:=Dv_k + Dv$. We apply Lemma~\ref{app2} to the equation in~\eqref{eqtildev}, with \mbox{$H=- D \bar u - \frac{1}{2}F_k$} and \mbox{$g = \xi_{\nu}(\nu_k \rho_k - \nu \rho) - \delta (v_k - v)$}. Note that $H$ is bounded uniformly in $k$ because both $Dv_k$ and $Dv$ are. Therefore, we get:
		 	\begin{equation*}
		 		\begin{aligned}
		 		\| D \hat v_k \|_X &\le 2 \| D \hat v_k(T,\cdot) \|_{L^{\infty}(\T^n)} + C \left( \| \xi_{\nu} \|_X \cdot \| \nu_k \rho_k - \nu \rho \|_X + (\delta+1) \|v_k - v \|_X \right) \\
		 		& \le 2 \left( | \nu_k - \nu | \| Dv_T(\cdot) \|_{L^{\infty}(\T^n)} + C \cdot C_f \|\rho_k\|_X \right) + C \left(C_f \| \rho_k - \rho \|_X + (\delta + 1) \| v_k - v \|_X \right),
		 		\end{aligned}
		 	\end{equation*}
		 	where $C$ depends on $\| \bar u \|_{L^{\infty}(\T^n)}$ and $\Theta$, with the right hand side going to $0$ as $k \to +\infty$.
		 	
		 	\medskip
		 
		 	Finally, let $\hat \mu_k:=\mu_k - \mu$. This function satisfies the following PDE:
		 	\begin{equation*}
		 		\begin{cases}
		 			\partial_t \hat \mu_k - \Delta \hat \mu_k - \text{div}(\hat \mu_k D \bar u) - \text{div}(\bar m D\hat v_k) - \text{div}(\hat \mu_k Dv) - \text{div}(\mu_k D \hat v_k) = 0,& \text{ on } (0,T) \times \T^n \\
		 			\hat \mu_k (0,\cdot) = (\nu_k - \nu) \mu_0(\cdot),& \text{ on } \T^n.
		 		\end{cases}
		 	\end{equation*}
		 	A direct application of Lemma~\ref{app1} requires to control the spatial $L^n$ norm of $\hat \mu_k$, at least on $[\tau,T]$. To avoid it, we take advantage of the fact that $T$ is fixed, of estimate~\eqref{roughbound2} in Remark~\ref{roughrmk} and the boundedness of $\T^n$ to get:
		 	\begin{equation} \label{hatmuk}
		 		\begin{split}
		 		\| \hat \mu_k (t,\cdot) \|_{L^{\infty}(\T^n)} &\le C \left( | \nu_k - \nu | \| \mu_0(\cdot) \|_{L^{\infty}(\T^n)} + \| \bar m D \hat v_k \|_{L^{\infty}([0,T];L^{n+1}(\T^n))} + \| \mu_k  D \hat v_k \|_{L^{\infty}([0,T];L^{n+1}(\T^n))} \right) \\
		 		&\le C \left( | \nu_k - \nu | \| \mu_0(\cdot) \|_{L^{\infty}(\T^n)} + \| \bar m D \hat v_k \|_X + \| \mu_k \|_X \cdot \| D \hat v_k \|_X \right),
		 		\end{split}
		 	\end{equation}
		 	where $C$ depends on $\tau,T,\| D \bar u \|_{L^{\infty}(\T^n)}$. Moreover, since $\| Dv_k \|_{X}$ are uniformly bounded in $k$, a further application of estimate~\eqref{roughbound2} in Remark~\ref{roughrmk}, we observe that also $\| \mu_k \|_X$ are uniformly bounded in $k$. This enables us to improve~\eqref{hatmuk} as follows:
		 	\begin{equation*}
		 		%\label{hatmuk2}
		 		\| \hat \mu_k (t,\cdot) \|_{L^{\infty}(\T^n)} \le C \left( | \nu_k - \nu | \| \mu_0(\cdot) \|_{L^{\infty}(\T^n)} + \| \bar m \|_{L^{\infty}(\T^n)} \cdot \| D \hat v_k \|_X + \tilde{C} \| D \hat v_k \|_X \right), \qquad \forall t \in [0,T].
		 	\end{equation*}
		 	The right hand side goes to $0$, hence we conclude that
		 	$$\| \mu_k - \mu \|_X \to 0,$$
		 	for $k \to +\infty$.
		 	
		 	\medskip

			\textbf{Step 2: $\mathcal{F}$ is compact.} 

At this stage of the proof, we take advantage of the compactness of $\T^n$ to deduce compactness for $\mathcal{F}$. To reach this aim, we wish to prove that, given a sequence $\mu_k:=\mathcal{F}(\rho_k,\nu_k)$, for some $\{ (\rho_k,\nu_k) \} \subseteq U \times [0,1]$, it admits a limit point $\mu \in X$ with respect to the convergence given by either the $\triplenorm{\cdot}_X$ or the $\|\cdot\|_X$ norm. 
By the previous step, we already know that both $\mu_k$ and $v_k$ are bounded with respect to the norm on $X$ uniformly in $k$. 
Then, since $D \bar u$, $\bar m$, $ D \bar m$, $v_k$, $\rho_k$ are uniformly (in $k$) bounded in $L^{p}(\T^n)$ for any $p \geq 1$ and the initial / terminal data are regular, from standard results in parabolic regularity theory (see for instance~\cite{LSU68}), we deduce that $\mu_k$ is bounded in $\mathcal{C}^{\theta,\theta / 2} ([0,T] \times \T^n)$ for some $\theta \in (0,1)$ independent of $k$.
As a consequence, the whole sequence $\{ \mu_k \}$ is equi-continuous in $X$, thus one can apply Ascoli-Arzelà theorem. Therefore, there exists $\mu \in X$ and a subsequence $\{ \mu_{k_j} \}$ such that 
$$\|\mu_{k_j}-\mu\|_X \to 0 \qquad \text{ for } j \to +\infty,$$
hence the compactness of $\mathcal{F}$.

\medskip

\textbf{Step 3: Fixed points satisfy $\mu=\mathcal{F}(\mu,\nu) \not \in \partial U$, i.e. $\triplenorm{\mu}_X < \bar \eps$.} \\
In this final step we will use the \textit{a priori} estimates of Sections~\ref{nlturnl2} and~\ref{nlturnlinf}. \\
Note that Proposition~\ref{propfix2} requires the fundamental hypothesis~\Sp and the integrability assumptions~\eqref{integrab} (condition \eqref{cond5} is a consequence of \eqref{cond4} on the torus). To get the latter, we show first that $(v,\mu)$ possess further regularity properties, by means of a bootstrap argument. First of all, starting from \mbox{$\mu \in \mathcal{C}([0,T] \times \T^n)$}, we immediately observe that both $v$ and $Dv$ are bounded with respect to the norm on $X$, hence they belong to any $L^p$ space, for $p \ge 1$; thus, standard parabolic regularity theory, for instance~\cite{LSU68,Lieb}, yields that \mbox{$\mu \in \mathcal{C}^{\theta,\theta / 2}([0,T] \times \T^n)$} for some $\theta \in (0,1)$, as in the previous compactness step. Again, making use of the extra regularity just derived, we see that \mbox{$v \in \mathcal{C}^{2+\theta,1+\theta / 2}([0,T] \times \T^n)$} by parabolic Schauder theory. Computing the divergence terms in the second equation yields that also \mbox{$\mu \in \mathcal{C}^{2+\theta,1+\theta / 2}([0,T] \times \T^n)$}. Actually, $v$ has at least one additional continuous spatial derivative, hence the regularity assumptions in ~\eqref{integrab} are satisfied because $\bar m$ has a strictly positive minimum on $\T^n$. Note also that, $\Omega=\T^n$ and the structure of the equation for $\mu$ has only terms in divergence form, we notice that, by integration by parts,
$$\frac{d}{dt} \int_{\T^n} \mu(t,x) \, \di x = 0,$$
hence
$$\int_{\T^n} \mu(t,x)\, \di x = \int_{\T^n} \mu_0(x) \, \di x = 0.$$

%In particular, the functions are classical solutions with the same regularity as the initial and final data. Thus, we can apply the \textit{a priori} estimates of Section~\ref{nlturnlinf} since~\Sp can be used and the integrability assumptions~\eqref{integrab} are satisfied. Therefore,
We finally apply Proposition~\ref{propfix2} (note that the estimate holds for all $\nu$ in view of Remarks~\ref{nurmk2} and~\ref{nurmk3}) to get that
$$\triplenorm{\mu}_X \le \bar{\eps} \Longrightarrow \triplenorm{\mu}_X \le \frac{\bar{\eps}}{2} < \bar \eps,$$
i.e. $\mu \not \in \partial U$, hence the fixed points of $\mathcal{F}(\rho,\nu)$ cannot lie on the boundary of $U$ for all $\nu \in [0,1]$.

\medskip

Lastly, we need to verify that $\mathcal{F}(\cdot,0)=0$ on $U$. Note that $\mu:=\mathcal{F}(\rho,0)$ satisfies the following PDE system
\begin{equation*}
	%\label{f0}
	\begin{cases}
		-\partial_t v - \Delta v + \langle Dv,D\bar{u} \rangle + \delta v + \frac{|Dv|^2}{2} = 0,& \text{ on } (0,T) \times \T^n \\
		\partial_t \mu - \Delta \mu - \text{div}(\bar{m}Dv) - \text{div}(\mu D\bar{u}) - \text{div}(\mu Dv)=0,& \text{ on } (0,T) \times \T^n \\
		\mu(0,\cdot)=0 \qquad v(T,\cdot)=0.& \text{ on } \T^n
	\end{cases}
\end{equation*}
Since $v \equiv 0$ is the unique solution of the first equation, we deduce that $\mu\equiv0$, by uniqueness for the second equation.

\medskip

Therefore, the operator $\mathcal{F}(\rho,1)$ admits a fixed point $\mu \in \bar{U}$, which means a couple $(v,\mu)$ that solves system~\eqref{nlsys} and satisfies the turnpike property. \\
Finally, if we take $m^T:=\mu + \bar{m}$ and $u^T:=v+\bar{u}$, we get a solution to~\eqref{dynsys} on $\T^n$ which satisfies the turnpike property~\eqref{turnpike}. As regards property~\eqref{turnpikeDuT}, it follows from Lemma~\ref{lambdasmall}.
\end{proof}

Estimate~\eqref{turnpike} plays a fundamental role in the proof of the existence of solutions to the infinite horizon problem. More precisely, the fact that the constant $\bar \eps$ is independent of $T$ allows us to pass to the limit as $T \to +\infty$ and to preserve the same estimate also for the limit function. We are able to state the existence theorem for~\eqref{infhor}.
\begin{thm} \label{exinfhor}
	Let $\Omega=\T^n$. Suppose that $(\bar u,\bar m)$ is a (classical) solution to~\eqref{statsys} and that~\Sp holds. Then, there exist constants $\gamma > 0, \bar \eps > 0, \sigma > 0$ (which depend only on $\bar m, \bar u, \eta, C_f$) such that if
		\begin{equation*}
		\delta < \sigma, \quad \| m_0(\cdot)-\bar m(\cdot) \|_{L^{\infty}(\T^n)} < \gamma,
		%\label{ininfinfty}%\\
%		& \left\| \frac{m_0(\cdot)-\bar m(\cdot)}{\sqrt{\bar{m}(\cdot)}}\right\|_{L^{2}(\T^n)} < \lambda,
%		\label{inl2infty}
	\end{equation*}
	then there exists a (classical) solution $(u^\infty,m^\infty)$ to~\eqref{infhor} satisfying the following properties:
		\begin{align*}
			%\label{epsturn2}
			&\|m^{\infty}(t,\cdot)-\bar m(\cdot)\|_{L^{\infty}(\T^n)} \le \frac{\bar{\eps}}{2} e^{-\sigma_1 t}, \\
			&\left \| \sqrt{\bar m(\cdot)} \left| Du^{\infty}(t,\cdot) - D\bar u(\cdot) \right| \right \|_{L^2(\T^n)} \le \lambda_3 e^{-\sigma_1 t},
		\end{align*}
	for all $t > 0$, where $\sigma_1=\frac{\sigma-\delta}{n+1}$ and $\lambda_3$ comes from Lemma~\ref{lambdasmall}.
	%Moreover, $m^{\infty}$ and $u^{\infty}$ can be obtained as limits of locally (with respect to $t$) converging subsequences of solutions to MFG system~\eqref{dynsys} on finite-measure sets $(0,k) \times \T^n$.
\end{thm}

\begin{proof} The proof relies on a diagonal argument.
	From Theorem~\ref{existence}, we already know that, for each $T > 0$, there exists a solution $(v^T,\mu^T)$ to~\eqref{nlsys} which satisfies the turnpike property, if the initial / terminal data are small enough. 
	
	Let $k \in \mathbb{N}$ and consider a solution $(v_k,\mu_k)$ to the following MFG system
	\begin{equation} \label{finitek}
		\begin{cases}
			-\partial_t v_k - \Delta v_k + \langle Dv_k,D\bar{u} \rangle - f(\bar{m}+\mu_k) + f(\bar{m}) + \frac{|Dv_k|^2}{2} + \delta v_k=0,&  \text{ on } (0,k) \times \T^n \\
			\partial_t \mu_k - \Delta \mu_k - \text{div}(\bar{m}Dv_k) - \text{div}(\mu_k D\bar{u}) - \text{div}(\mu_k Dv_k)= 0,&  \text{ on } (0,k) \times \T^n  \\
			\mu_k(0,\cdot)=\mu_0(\cdot) \quad v_k(k,\cdot)=0,& \text{ on } \T^n, \\
		\end{cases}
	\end{equation}
	where $\mu_0=m_0-\bar m$.
	
	Fix now $j=1$. Each couple of functions $(v_k,\mu_k)$ admits a restriction to $[0,1] \times \T^n$ (if $k \geq 1$) which is again a solution to the equations in~\eqref{finitek}. Moreover, since by Theorem~\ref{existence}
	$$\| \mu_k(t,\cdot) \|_{L^{\infty}(\T^n)} \le \frac{\bar \eps}{2} \left( e^{-\sigma_1 t} + e^{-\sigma_2(k-t)} \right) \le \bar \eps \qquad \forall t \in [0,k],$$
	it is true that for all $k > 1$ (i.e. if $k$ is large enough)
	$$\sup_{t \in [0,1]} \| \mu_k(t,\cdot) \|_{L^{\infty}(\T^n)} \le \bar \eps,$$
	which means that $\{ \mu_k \}$ is equibounded on $[0,1] \times \T^n$. The same regularity argument used in Theorem~\ref{existence} (Step 2) enables us to establish that the sequence $\{ \mu_k \}$ is also equi-continuous on $[0,1] \times \T^n$. Therefore, Ascoli-Arzelà theorem yields that there exists a subsequence $\{ \tilde \mu_k^{(1)} \} \subseteq \{ \mu_k \}$ such that
	$$\tilde \mu_k^{(1)} \to \mu^{(1)},$$
	for some $\mu^{(1)}$ defined on $[0,1] \times \T^n$.
	
	We now do the same for $v_k$. We consider the sequence $\{ \tilde v_k^{(1)} \}$, corresponding to $\{ \tilde \mu_k^{(1)} \}$. As a first step, we know that, by the previous Theorem and Lemmata in Section~\ref{nlturnlinf}, 
	$$\sup_{t \in [0,1]} \| \tilde v_k^{(1)}(t,\cdot) \|_{L^{\infty}(\T^n)} \le \sup_{t \in [0,k]} \| \tilde v_k^{(1)}(t,\cdot) \|_{L^{\infty}(\T^n)} \le C,$$
	for some $C$ independent of $T$ and $k$. Moreover, parabolic Schauder regularity theory (as in the proof of Theorem~\ref{existence}, Step 3) yields that $\tilde v_k^{(1)}$ and $\tilde \mu_k^{(1)}$ are bounded in $C^{2+\theta, 1+\theta/2}([0,1] \times \T^n)$. We can apply again Ascoli-Arzelà theorem to extract a uniformly converging subsequence $v_k^{(1)} \to v^{(1)}$ for some $v^{(1)}$ defined on $[0,1] \times \T^n$.  Moreover, because of the global estimates for $v_k^{(1)}$ on $[0,k] \times \T^n$, also $\| v^{(1)}(1,\cdot) \|_{L^{\infty}(\T^n)} \le C$, where $C$ is the same as above.
	Finally, we take the subsequence $\{ \mu_k^{(1)} \}$, where each $\mu_k^{(1)}$ corresponds to $v_k^{(1)}$ in the couple $(v_k^{(1)},\mu_k^{(1)})$ which solves both equations in~\eqref{finitek}, by the $C^{2,1}$ convergence. Again, by uniqueness of limit, $\mu_k^{(1)} \to \mu^{(1)}$ on $[0,1] \times \T^n$ as $k \to +\infty$. \\
	Summing up, so far we have obtained subsequences $\{ \mu_k^{(1)} \}$ and $\{ v_k^{(1)} \}$ converging to $\mu^{(1)}$ and $v^{(1)}$ respectively, on $[0,1] \times \T^n$. \\
	To go on, fix now $j=2$. By repeating the same argument as above, one gets that there exists a subsequence $\{ \tilde \mu_k^{(2)} \} \subseteq \{ \mu_k^{(1)} \}$ such that
	$$\tilde \mu_k^{(2)} \to \mu^{(2)},$$
	for some $\mu^{(2)}$ defined on $[0,2] \times \T^n$. Notice that
	$$\mu^{(2)} \big|_{[0,1] \times \T^n} = \mu^{(1)},$$
	by uniqueness of limit, since $\{ \tilde \mu_k^{(2)} \}$ is a subsequence of $\{ \mu_k^{(1)} \}$. \\
	Analogously, we get that from the corresponding $\{ \tilde v_k^{(2)} \}$, one can extract a subsequence $\{ v_k^{(2)} \}$ which converges to some $v^{(2)}$ on $[0,2] \times \T^n$. \\
	We iterate the procedure; in the end, we obtain that there exist subsequences $\{ \mu_k^{(j)} \} \subseteq \{ \mu_k \}$, $\{ v_k^{(j)} \} \subseteq \{ v_k \}$ and sequences $\{ \mu^{(j)} \}$,$\{ v^{(j)} \}$ such that $\mu_k^{(j)} \to \mu^{(j)}$, $v_k^{(j)} \to v^{(j)}$ (in $\mathcal{C}^{2,1}$, locally in time and globally in space), as $k \to +\infty$, $\mu^{(j)} \big|_{[0,l] \times \T^n} = \mu^{(l)}$ and $v^{(j)} \big|_{[0,l] \times \T^n} = v^{(l)}$ for all $l < j$. As a consequence of these facts, one takes diagonal subsequences $\{ \mu^{(k)}_k \}$ and $\{ v^{(k)}_k \}$, which converge in $\mathcal{C}^{2,1}$, respectively, to the functions
	\begin{align*}
		&\mu(t,x):=\mu^{(j)}(t,x), \qquad \text{ if } t \in [0,j] \text{ for some }j, \\
		&v(t,x):=v^{(j)}(t,x), \qquad \text{ if } t \in [0,j] \text{ for some }j,
	\end{align*}
	where the convergence is local in time. \\
	Since the convergence in $C^{2,1}$ allows to pass to the limit also the nonlinear terms in the equations of~\eqref{nlsys}, the couple $(v,\mu)$ defined above gives a solution to the infinite horizon problem with initial condition given by $\mu_0=m_0 - \bar m$ and with $v$ bounded for all $t > 0$. In particular, this is true since $v$ is built by means of the functions $v^{(j)}$, which are all bounded by a constant independent of $T$. \\
	Finally, since each $\mu_k^{(k)}$ satisfies the turnpike property
	\begin{equation*}
		\| \mu_k^{(k)}(t,\cdot) \|_{L^{\infty}(\T^n)} \le \frac{\bar \eps}{2} \left( e^{-\sigma_1 t} + e^{-\sigma_2(k-t)} \right) \qquad \forall t \in [0,k],
	\end{equation*}
	if we fix $t \in [0,k]$, then
	$$\| \mu(t,\cdot) \|_{L^{\infty}(\T^n)} = \left \| \lim_{k \to +\infty} \mu_k^{(k)}(t,\cdot) \right \|_{L^{\infty}(\T^n)} = \lim_{k \to +\infty} \| \mu_k^{(k)}(t,\cdot) \|_{L^{\infty}(\T^n)} \le \frac{\bar \eps}{2} e^{-\sigma_1 t}.$$
	The previous inequality holds for all $t > 0$, hence the thesis, after having defined $m^{\infty}:=\mu + \bar m$ and $u^{\infty}:=v+\bar u$. In fact, also $u^{\infty}$ is bounded since both $v$ and $\bar u$ are.  As regards property~\eqref{turnpikeDuT}, it follows from Lemma~\ref{lambdasmall}, passing to the limit as $T \to +\infty$.
\end{proof}

\section{Appendix} \label{appendix}
	This final section is devoted to present the proof of some useful inequalities used in Section~\ref{nlturnlinf} to get the \textit{a priori} estimates. We point out that similar estimates can be found also in~\cite{CG22}; however, we need more precise ones, with constants independent of the time horizon $T$, and to do this we have to perform a slightly more refined work. \\
	Moreover, we specify that the estimates in this Section hold for both $\Omega = \T^n$ and $\Omega = \R^n$.
	
	First of all, we recall that the heat kernel on $\R^n$ is defined as the function
	\begin{equation*}
		\Gamma(t,x)=
		\begin{cases}
			\frac{1}{(4\pi t)^{\frac{n}{2}}} e^{-\frac{|x|^2}{4t}} ,& \text{ if } t > 0 \\
			0,& \text{ if } t \le 0,
		\end{cases}
		%\label{heatkernel}
	\end{equation*}
	for $x \in \mathbb{R}^n$. Recall that its $L^{1}$ norm over $\R^n$ is equal to 1, for $t > 0$. Moreover, one can consider the heat kernel defined on $\T^n$ by taking the periodization of the function $\Gamma$, as follows:
	\begin{equation*}
		x \mapsto \sum_{k \in \mathbb{Z}^n} \Gamma(t,x+2k\pi);
	\end{equation*}
	such function is defined for all $x \in \R^n$ and it is periodic, hence it can be seen as a function on $\T^n$. The properties recalled in the following Lemma are standard, and valid both on $\Omega=\R^n$ and on $\Omega=\T^n$ (for $\Gamma$ and its periodization respectively).
%	We briefly recall here some very well-known results about the $L^p$ norms of $\Gamma$ and its gradient both on $\R^n$ and on $\T^n$. The heat kernel on $\T^n$ has a similar structure, and in particular it has the same behavior as on $\R^n$.
\begin{lem}	\label{lemfix1}
		There exist two constants $C_1(n,p)$ and $C_2(n,p)$ such that, for all $p \in [1,+\infty]$,
		\begin{equation*}
			\| \Gamma(t,x) \|_{L^{p}(\Omega)} \le C_1(n,p) t^{-\frac{n}{2p'}}
			%\label{normgamma}
		\end{equation*}
		and
		\begin{equation*}
			\left\| |D\Gamma|(t,x) \right\|_{L^p(\Omega)} \le C_2(n,p) t^{-\frac{n}{2p'}-\frac{1}{2}},
			%\label{normDgamma}
		\end{equation*}
		where $p'$ is such that $\frac{1}{p}+\frac{1}{p'} = 1$ (with $p'=+\infty$ if $p=1$).
	\end{lem}
%\begin{rmk}
%	When $\Omega=\R^n$, the inequalities above become equalities, since the computations can be made precisely by using the explicit formula for the heat kernel.
%\end{rmk}

%	\begin{proof}[Proof (Lemma~\ref{lemfix1})]
%		We compute
%		$$\| \Gamma(t,x) \|^p_{L^p(\mathbb{R}^n)} = \int_{\mathbb{R}^n} \frac{1}{(4 \pi t)^{\frac{np}{2}}} e^{-\frac{p|x|^2}{4t}} \di x = (\square)$$
%		We make the following change of variables
%		$$x \frac{\sqrt{p}}{\sqrt{4t}} = z,$$
%		hence $dx = dz \left( \frac{4t}{p} \right)^{\frac{n}{2}}$.
%		We continue our computation in the following way
%		$$(\square)=\int_{\mathbb{R}^n}  \frac{1}{(4 \pi t)^{\frac{np}{2}}} e^{-|z|^2} \frac{(4t)^\frac{n}{2}}{p^{\frac{n}{2}}} \di z = \frac{1}{p^{\frac{n}{2}}} \frac{1}{(4 \pi t)^{\frac{pn}{2}-\frac{n}{2}}}.$$
%		After computing the $p$-th root, we derive~\eqref{normgamma}, recalling the definition of Hölder conjugate of $p$. \\
%		Then, it's easy to see that, where it isn't zero,
%		$$D\Gamma(t,x)=-\frac{1}{\pi^{\frac{n}{2}}} \frac{1}{(4t)^{\frac{n+2}{2}}} x \cdot e^{-\frac{|x|^2}{4t}},$$
%		hence
%		$$\| D\Gamma(t,x) \|_{L^p(\mathbb{R}^n)}^p = \int_{\mathbb{R}^n} \frac{1}{(4\pi t)^{\frac{np}{2}} (4t)^{p}} |x|^p e^{-\frac{p|x|^2}{4t}} \di x=(\ast)$$
%		We make the same change of variables as before and we get
%		$$(\ast)=\frac{1}{(\pi)^{\frac{np}{2}}} \frac{1}{p^{\frac{p+n}{2}}} (4t)^{-\frac{p}{2} +\frac{n}{2} - \frac{np}{2}} \int_{\mathbb{R}^n} |z|^p e^{-|z|^2} \di z = C_1(n,p) t^{-\frac{np}{2} - \frac{p}{2} + \frac{n}{2}},$$
%		hence, after extracting the $p$-th root, one gets~\eqref{normDgamma}.
%	\end{proof}

\medskip

We now prove two Lemmas used in Section~\ref{nlturnlinf}.
	\begin{lem}	\label{app1}
		Let $z$ be the unique (mild) solution to the following linear forward PDE:
		\begin{equation*}
			\begin{cases}
				\partial_t z - \Delta z = {\rm div}(z F) + {\rm div}(G),& \text{ on } (0,T) \times \Omega \\
				z(0,\cdot)=z_0(\cdot),& \text{ on } \Omega. \\
			\end{cases}
			%\label{app1pde}
		\end{equation*}
		Suppose that $F \in L^{\infty}([0,T] \times \Omega)$, $F \neq 0$, $G \in L^{\infty}([0,T];L^{n+1}(\Omega))$.
		Assume moreover that $z \in L^{\infty}([0,T];L^{n}(\Omega))$. Then, the following estimate holds:
		\begin{equation}
			\| z(t,\cdot) \|_{L^{\infty}(\Omega)} \le 2 \| z_0(\cdot) \|_{L^{\infty}(\Omega)} + C_F \left( \| G \|_{L^{\infty}((t-\tau,t); L^{n+1}(\Omega))} + \mathbbm{1}_{[\tau,T]}(t) \| z \|_{L^{\infty}((t-\tau,t);L^n(\Omega))} \right),
			\label{estapp1}
		\end{equation}
		for all $t \in [0,T]$, where $\tau=\frac{1}{16 \| F \|^2_{L^{\infty}([0,T] \times \Omega)} C_2(n)^2}$ and $C_F$ depends on $F,n,\tau$.
	\end{lem}
	\begin{proof}
		First of all, let $C_1(n):=C_1(n,\frac{n}{n-1})$, $C_2(n):=C_2(n,1)$ and $\tilde C_2(n):=C_2(n,\frac{n+1}{n})$, where $C_1(n,p)$ and $C_2(n,p)$ are as in Lemma~\ref{lemfix1}. We consider the PDE as a heat equation with an external source; therefore, we use Duhamel representation formula to get the following expression
		\begin{equation*}
			\begin{split}
			z(t,x)&=\int_{\Omega} \Gamma(t,x-y)z_0(y) \di y + \int_0^t \int_{\Omega} \Gamma(t-u,x-y) \left[ \text{div}(zF) + \text{div}(G) \right](u,y) \di y \di u \\
			&=\int_{\Omega} \Gamma(t,x-y)z_0(y) \di y + \int_0^t \int_{\Omega} \langle D\Gamma(t-u,x-y), -z(u,y)F(u,y) - G(u,y) \rangle \di y \di u.
			\end{split}
			%\label{duhmu1}
		\end{equation*}
		Thus, by means of Hölder's inequality for convolutions,
		\begin{multline*}
			\| z(t,\cdot) \|_{L^{\infty}(\Omega)} \le \| \Gamma(t-u,x-\cdot) \|_{L^1(\Omega)} \cdot \| z_0 \|_{L^{\infty}(\Omega)} +\\ + \int_0^t \| D\Gamma(t-u,x-\cdot) \|_{L^1(\Omega)} \| z(u) \|_{L^{\infty}(\Omega)} \| F \|_{L^{\infty}([0,T] \times \Omega)} \di u
			+\\ + \int_0^t \| D\Gamma(t-u,x-\cdot) \|_{L^{\frac{n+1}{n}}(\Omega)} \| G(u) \|_{L^{n+1}(\Omega)} \di u \\
			\le \| z_0 \|_{L^{\infty}(\Omega)} +  C_2(n) \| F \|_{L^{\infty}([0,T] \times \Omega)}  \int_0^t (t-u)^{-\frac{1}{2}} \underset{u \in [0,\tau]}{\sup} \| z(u) \|_{L^{\infty}(\Omega)} \di u
			+ \\ +\tilde C_2(n) \int_0^t (t-u)^{-1+\frac{1}{2n+2}} \| G(u) \|_{L^{n+1}(\Omega)} \di u,
		\end{multline*}
		where $\tau$ is to be fixed and we consider only $t \in [0,\tau]$.
		After performing the change of variables $r:=t-u$, we get
		\begin{equation*}
			\begin{split}
			&\le \| z_0 \|_{L^{\infty}(\Omega)} + C_2(n)  \| F \|_{L^{\infty}([0,T] \times \Omega)}  \underset{u \in [0,\tau]}{\sup} \| z(u) \|_{L^{\infty}(\Omega)} \int_0^t r^{-\frac{1}{2}} \di r \\
			&+\tilde C_2(n) \int_0^t r^{-1+\frac{1}{2n+2}} \| G(t-r) \|_{L^{n+1}(\Omega)} \di r \\
			&\le \| z_0 \|_{L^{\infty}(\Omega)} + C_2(n) \| F \|_{L^{\infty}([0,T] \times \Omega)}  \underset{u \in [0,\tau]}{\sup} \| z(u) \|_{L^{\infty}(\Omega)} 2\tau^{\frac{1}{2}} + \\
			&+ \tilde C_2(n) \int_0^{\tau} r^{-1+\frac{1}{2n+2}} \| G(t-r) \|_{L^{n+1}(\Omega)} \di r \\
			&\le \| z_0 \|_{L^{\infty}(\Omega)} + C_2(n) \| F \|_{L^{\infty}([0,T] \times \Omega)} \underset{u \in [0,\tau]}{\sup} \| z(u) \|_{L^{\infty}(\Omega)} 2\tau^{\frac{1}{2}} + \\
			&+(2n+2) \tilde C_2 (n)\tau^{\frac{1}{2n+2}} \| G \|_{L^{\infty}((t-\tau,t);L^{n+1}(\Omega))}.
			\end{split}
		\end{equation*}
		Choose $\tau=\frac{1}{16 C_2(n)^2 \| F \|^2_{L^{\infty}([0,T] \times \Omega)}}$ and take the supremum over $[0,\tau]$ in the left hand side. Therefore,
		\begin{equation*}
			\underset{t \in [0,\tau]}{\sup} \| z(t,\cdot) \|_{L^{\infty}(\Omega)} \le \| z_0 \|_{L^{\infty}(\Omega)} + \frac{1}{2} \underset{u \in [0,\tau]}{\sup} \| z(u) \|_{L^{\infty}(\Omega)} +(2n+2) \tilde C_2(n) \tau^{\frac{1}{2n+2}} \| G \|_{L^{\infty}((t-\tau,t);L^{n+1}(\Omega))},
		\end{equation*}
		from which
		\begin{equation*}
			\underset{t \in [0,\tau]}{\sup} \| z(t,\cdot) \|_{L^{\infty}(\Omega)} \le 2 \| z_0 \|_{L^{\infty}(\Omega)} + 2 (2n+2) \tilde C_2(n) \tau^{\frac{1}{2n+2}} \| G \|_{L^{\infty}((t-\tau,t);L^{n+1}(\Omega))}.
			%\label{estmu1}
		\end{equation*}
		which gives the first part of the estimate~\eqref{estapp1}. \\
		For larger times, we have to use a cut-off method. Fix $s \in (0,T-1)$ and let
		\begin{equation*}
			\eta(t):=
			\begin{cases}
				0,& \text{ if } t \in [0,s] \\
				t-s,& \text{ if } t \in [s,s+1] \\
				1,& \text{ if } t \in [s+1,T] \\ 
			\end{cases}
			%\label{eta1}
		\end{equation*}
		and let $\tilde{z}(t,x)=\eta(t) z(t,x)$. We have that
		\begin{equation*}
			\begin{split}
				\partial_t \tilde{z}= \eta \partial_t z + \eta' z &= \eta (\Delta z + \text{div}(zF) + \text{div}(G)) + \eta' z = \\
				&= \Delta \tilde{z} + \text{div} (\tilde{z}F) + \eta \text{ div}(G) +  \eta' z,
			\end{split}
		\end{equation*}
		with initial condition $\tilde{z}(0,\cdot)=\eta(0) z(0,\cdot)=0$.
		Hence, $\tilde{z}$ satisfies (a.e.) the following Cauchy problem:
		\begin{equation*}
			\begin{cases}
				\partial_t \tilde{z} - \Delta \tilde{z} = \text{div}(\tilde{z}F) + \eta \text{ div}(G) + \eta' z,& \text{ on } (0,T) \times \Omega \\
				\tilde{z}(0,\cdot)=0,& \text{ on } \Omega
			\end{cases}
			%\label{cutoffz1}
		\end{equation*}
		Using Duhamel representation formula, we get
		\begin{equation*}
			\begin{split}
				\tilde{z}(t,x)&=\int_0^t \int_{\Omega} \Gamma(t-u,x-y) \left( \text{div}(\tilde{z}F)+ \eta \text{ div}(G) + \eta' z \right) (u,y) \di y \di u \\
				&=\int_0^t \int_{\Omega} \langle D\Gamma(t-u,x-y),-\tilde{z}F(u,y) - \eta(u) G(u,y) \rangle + \Gamma(t-u,x-y) \eta' (u,y) z(u,y) \di y \di u \\
				&=\int_s^t \int_{\Omega} \langle D\Gamma(t-u,x-y),-\tilde{z}F(u,y) - \eta(u) G(u,y) \rangle + \Gamma(t-u,x-y) \eta' (u,y) z(u,y) \di y \di u,
			\end{split}
		\end{equation*}
		where the last equality follows from the fact that $\eta$ (and consequently $\tilde{z}$) has support, with respect to $t$, on $[s,T]$. Consider $s \le t \le s+\tau$, where $\tau$ is defined as before. Then, again by Hölder's inequality for convolutions,
		\begin{equation*}
			\begin{split}
			\| \tilde{z}(t,\cdot) \|_{L^{\infty}(\Omega)} &\le \int_s^t \| D\Gamma(t-u,x-\cdot) \|_{L^1(\Omega)} \| F \|_{L^{\infty}([0,T] \times \Omega)} \| \tilde{z}(u) \|_{L^{\infty}(\Omega)} \di u \\
			&+\int_s^t \| D\Gamma(t-u,x-\cdot) \|_{L^{\frac{n+1}{n}}(\Omega)} \| G(u) \|_{L^{n+1}(\Omega)} + \| \Gamma(t-u,x-\cdot) \|_{L^{\frac{n}{n-1}}(\Omega)} \| z(u) \|_{L^{n}(\Omega)} \di u \\
			&\le C_2(n) \| F \|_{L^{\infty}([0,T] \times \Omega)} \underset{u \in [s,s+\tau]}{\sup} \| \tilde{z}(u) \|_{L^{\infty}(\Omega)} \int_s^t (t-u)^{-\frac{1}{2}} \di u \\
			&+ \tilde C_2(n) \int_s^t (t-u)^{-1+\frac{1}{2n+2}} \| G(u) \|_{L^{n+1}(\Omega)} \di u + C_1(n) \int_s^t (t-u)^{-\frac{1}{2}} \| z(u) \|_{L^{n}(\Omega)} \di u \\
			&\overset{r:=t-u}{\le} 2 C_2(n)  \| F \|_{L^{\infty}([0,T] \times \Omega)} (t-s)^{\frac{1}{2}}  \underset{u \in [s,s+\tau]}{\sup} \| \tilde{z}(u) \|_{L^{\infty}(\Omega)}  \\ 
			&+ \tilde C_2(n) \int_0^{t-s} r^{-1+\frac{1}{2n+2}} \| G(t-r) \|_{L^{n+1}(\Omega)} \di r + C_1(n) \int_0^{t-s} r^{-\frac{1}{2}} \| z(t-r) \|_{L^{n}(\Omega)} \di r \\
			&\le 2 C_2(n) \| F \|_{L^{\infty}([0,T] \times \Omega)} \tau^{\frac{1}{2}}  \underset{u \in [s,s+\tau]}{\sup} \| \tilde{z}(u) \|_{L^{\infty}(\Omega)} \\ 
			&+\tilde C_2(n) \int_0^{\tau} r^{-1+\frac{1}{2n+2}} \| G(t-r) \|_{L^{n+1}(\Omega)} \di r + C_1(n) \int_0^{\tau} r^{-\frac{1}{2}} \| z(t-r) \|_{L^{n}(\Omega)} \di r \\
			&\le 2 C_2(n) \| F \|_{L^{\infty}([0,T] \times \Omega)} \tau^{\frac{1}{2}}  \underset{u \in [s,s+\tau]}{\sup} \| \tilde{z}(u) \|_{L^{\infty}(\Omega)} \\ 
			&+ (2n+2)\tilde C_2(n) \tau^{\frac{1}{2n+2}} \| G \|_{L^{\infty}((t-\tau,t);L^{n+1}(\Omega))} + 2 \tau^{\frac{1}{2}} C_1(n)  \| z \|_{L^{\infty}((t-\tau,t);L^n(\Omega))} \\
			&\le \frac{1}{2} \underset{u \in [s,s+\tau]}{\sup} \| \tilde{z}(u) \|_{L^{\infty}(\Omega)} + (2n+2) \tilde C_2(n) \tau^{\frac{1}{2n+2}} \| G \|_{L^{\infty}((t-\tau,t);L^{n+1}(\Omega))} \\
			&+2 \tau^{\frac{1}{2}} C_1(n) \| z \|_{L^{\infty}((t-\tau,t);L^n(\Omega))}.
			\end{split}
		\end{equation*}
		Taking the supremum over $[s,s+\tau]$, we get
		\begin{equation*}
			\begin{split}
				\underset{u \in [s,s+\tau]}{\sup} \| \tilde{z}(u,\cdot) \|_{L^{\infty}(\Omega)} &\le 2(2n+2)\tilde C_2(n)  \tau^{\frac{1}{2n+2}} \| G \|_{L^{\infty}((t-\tau,t);L^{n+1}(\Omega))} \\
				&+4 \tau^{\frac{1}{2}} C_1(n) \| z \|_{L^{\infty}((t-\tau,t);L^n(\Omega))}.
			\end{split}
		\end{equation*}
		Recalling the definition of $\tilde{z}$, we get that
		\begin{multline*}
			\| z(s+\tau,\cdot) \|_{L^{\infty}(\Omega)} = \\ = \frac{1}{\tau} \| \tilde{z}(s+\tau,\cdot) \|_{L^{\infty}(\Omega)} \le \frac{2(2n+2) \tilde C_2(n) \tau^{\frac{1}{2n+2}} \| G \|_{L^{\infty}((t-\tau,t);L^{n+1}(\Omega))}}{\tau} \\ + \frac{4 \tau^{\frac{1}{2}} C_1(n)  \| z \|_{L^{\infty}((t-\tau,t);L^n(\Omega))}}{\tau},
		\end{multline*}
		for all $0 < s < T-1$.
		This can be extended also for $s > T-1$ (but $s < T-\tau$), by getting a different cut-off function
		\begin{equation*}
			\eta_2(t):=
			\begin{cases}
				0,& \text{ if } t \in [0,s] \\
				t-s,& \text{ if } t \in [s,T] \\
			\end{cases}
		%	\label{eta2}
		\end{equation*}
		and again $\tilde{z}(t,x)=\eta_2(t)z(t,x)$. The computations are the same as before. Gathering all the estimates, we finally get~\eqref{estapp1}.
	\end{proof}
	\begin{rmk} \label{roughrmk}
		We notice that in the proof of~\eqref{estapp1} the last term arises from the cut-off argument. This means that the estimate on $[0,\tau]$ reads as follows:
		\begin{equation} \label{roughbound1}
			\| z(t,\cdot) \|_{L^{\infty}(\Omega)} \le 2 \| z_0(\cdot) \|_{L^{\infty}(\Omega)} + C_F \| G \|_{L^{\infty}((0,\tau);L^{n+1}(\Omega))}.
		\end{equation}
	A rougher bound than~\eqref{estapp1} can be obtained by iterating~\eqref{roughbound1} $m$ times, where $m$ is the integer part of $\frac{T}{\tau}$. In this way,
	\begin{equation} \label{roughbound2}
		\begin{split}
		\| z(t,\cdot) \|_{L^{\infty}(\Omega)} &\le 2^m \| z_0(\cdot) \|_{L^{\infty}(\Omega)} + C_F \left( \sum_{k=0}^m 2^k \right) \| G \|_{L^{\infty}((0,\tau);L^{n+1}(\Omega))} \\
		& = 2^m \| z_0(\cdot) \|_{L^{\infty}(\Omega)} + C_F (2^{m+1}-1) \| G \|_{L^{\infty}((0,\tau);L^{n+1}(\Omega))},
		\end{split}
	\end{equation}
	for all $t \in [0,T]$.
	\end{rmk}
	The second estimate is about backward equations with terminal condition.
	\begin{lem} \label{app2}
		Let $z$ be the unique (mild) solution to the following linear backward PDE:
		\begin{equation}
			\begin{cases}
				-\partial_t z - \Delta z = \langle Dz,H \rangle + g, & \text{ on } (0,T) \times \Omega \\
				z(T,\cdot)=z_T(\cdot),& \text{ on } \Omega. \\
			\end{cases}
			\label{app2pde}
		\end{equation}
		Suppose that $H,g \in L^{\infty}([0,T] \times \Omega)$. Assume moreover that $z(t,\cdot) \in L^{\infty}(\Omega)$ for all $t \in [0,T]$. Then,
		\begin{equation}
			\| Dz(t,\cdot) \|_{L^{\infty}(\Omega)} \le 2 \| Dz_T(\cdot) \|_{L^{\infty}(\Omega)} + C_H \left( \| g \|_{L^{\infty}([0,T] \times \Omega)} + \mathbbm{1}_{[\tau,T]}(t) \| z \|_{L^{\infty}([0,T] \times \Omega)} \right),
			\label{estlemDz}
		\end{equation}
		for all $t \in [0,T]$.
		%where $\tau=\frac{1}{16 \| H \|^2_{L^{\infty}_t L^{\infty}_x} C_2(n)^2}$, 
		%\begin{equation}
			%D_t = \underset{t' \in [T-t,T-t+\tau]}{\sup} \int_0^{\tau} r^{-\frac{1}{2}} \| g(t'+r,\cdot) \|_{L^{\infty}} \di r
			%\label{D_t}
		%\end{equation}
		%and
		%\begin{equation}
			%E_t=\underset{t' \in [T-t,T-t+\tau]}{\sup} \int_0^{\tau} r^{-\frac{1}{2}} \| z(t'+r,\cdot) \|_{L^{\infty}} \di r.
			%\label{E_t}
		%\end{equation}
	\end{lem}
	\begin{proof}
		Let $C_2(n):=C_2(n,1)$, where $C_2(n,p)$ is as in Lemma~\ref{lemfix1}. We consider again, as in Lemma~\ref{app1}, the PDE as a heat equation with an external source. We recast the equation in~\eqref{app2pde} as a forward PDE by introducing the following change of (time) variable:
		$$\tilde z(t,x):=z(T-t,x).$$
		Then, we determine the equation satisfied by $\bar{z}$. 
		\begin{equation*}
			\begin{aligned}
				\partial_t \tilde z(t,x) &= -\partial_s z(T-t,x) \\
				&= \Delta z(T-t,x) + \langle Dz(T-t,x),H(T-t,x) \rangle + g(T-t,x) \\
				&=\Delta \tilde z (t,x) + \langle D\tilde z(t,x),H(T-t,x) \rangle + g(T-t,x),
			\end{aligned}
		\end{equation*}
		with initial condition $\tilde z(0,\cdot)=z_T(\cdot)$. Therefore, by using Duhamel representation formula, we get
		\begin{equation*}
			\tilde z(t,x)=\int_{\Omega} \Gamma(t,x-y) z_T(y) \di y + \int_0^t \int_{\Omega} \Gamma(t-u,x-y) \left[ \langle D\tilde z(u,y),H(T-u,y) \rangle + g(T-u,y) \right] \di y \di u,
		\end{equation*}
		hence, recalling the definition of $\tilde z$, 
		\begin{multline*}
			z(t,x)=\int_{\Omega} \Gamma(T-t,x-y)z_T(y) \di y \\ + \int_0^{T-t} \int_{\Omega} \Gamma(T-t-u,x-y) \left[ \langle Dz(T-u,y),H(T-u,y) \rangle + g(T-u,y) \right] \di y \di u.
		\end{multline*}
		Now, let us compute the $i$-th spatial derivative of $z$, exploiting the properties of the convolution with respect to differentiation. We get
		\begin{multline*}
			\partial_{x_i} z(t,x) = \int_{\Omega} \Gamma(T-t,x-y) \partial_{y_i} z_T(y) \di y + \\ + \int_0^{T-t} \int_{\Omega} \partial_{x_i} \Gamma(T-t-u,x-y) \langle Dz(T-u,y),H(T-u,y) \rangle \di y \di u +\\
			+\int_0^{T-t} \int_{\Omega} \partial_{x_i} \Gamma(T-t-u,x-y) g(T-u,y) \di y \di u,
		\end{multline*}
		from which, by Young's inequality for convolutions, we get
		\begin{equation*}
			\begin{split}
			\| \partial_{x_i} z(t,\cdot) \|_{L^{\infty}(\Omega)} &\le \| \Gamma(T-t,x-\cdot)\|_{L^1(\Omega)} \| Dz_T \|_{L^{\infty}(\Omega)} \\ 
			&+ \int_0^{T-t} \| D\Gamma(T-t-u,x-\cdot) \|_{L^1(\Omega)} \| Dz(T-u) \|_{L^{\infty}(\Omega)} \| H(T-u) \|_{L^{\infty}(\Omega)} \di u \\
			&+ \int_0^{T-t} \| D\Gamma(T-t-u,x-\cdot) \|_{L^1(\Omega)} \| g(T-u) \|_{L^{\infty}(\Omega)} \di u \\
			&= \| Dz_T \|_{L^{\infty}(\Omega)} + C_2(n) \int_{0}^{T-t}  (T-t-u)^{-\frac{1}{2}} \| Dz(T-u) \|_{L^{\infty}(\Omega)} \| H(T-u) \|_{L^{\infty}(\Omega)} \di u \\
			&+ C_2(n) \int_{0}^{T-t}  (T-t-u)^{-\frac{1}{2}} \| g(T-u) \|_{L^{\infty}} \di u \\
			&\overset{u':=T-u}{=} \| Dz_T \|_{L^{\infty}(\Omega)} + \\
			&+C_2(n) \int_t^T (u'-t)^{-\frac{1}{2}} \left( \| Dz(u') \|_{L^{\infty}(\Omega)} \| H(u') \|_{L^{\infty}(\Omega)} + \| g(u') \|_{L^{\infty}(\Omega)} \right) \di u'.
			\end{split}
		\end{equation*}
		Let $T-\tau \le t \le T$, with $\tau$ to be fixed later. From the previous inequality, with the change of variables $r:=u'-t$, we deduce
		\begin{equation*}
			\begin{split}
			&\le \| Dz_T \|_{L^{\infty}(\Omega)} + C_2(n) \int_t^T (u'-t)^{-\frac{1}{2}} \underset{u' \in [T-\tau,T]}{\sup} \| Dz(u') \|_{L^{\infty}(\Omega)} \| H(u') \|_{L^{\infty}(\Omega)} \di u' \\
			&+C_2(n) \int_0^{T-t} r^{-\frac{1}{2}} \| g(t+r) \|_{L^{\infty}(\Omega)} \di r \\
			&\le \| Dz_T \|_{L^{\infty}(\Omega)} + C_2(n) \| H \|_{L^{\infty}([0,T] \times \Omega)} \underset{u' \in [T-\tau,T]}{\sup} \| Dz(u') \|_{L^{\infty}(\Omega)} \int_0^{\tau} r^{-\frac{1}{2}} \di u' \\
			&+C_2(n) \| g \|_{L^{\infty}([0,T] \times \Omega)} \int_0^{\tau} r^{-\frac{1}{2}} \di r  \\
			&\le \| Dz_T \|_{L^{\infty}(\Omega)} + 2 \tau^{\frac{1}{2}} C_2(n) \| H \|_{L^{\infty}([0,T] \times \Omega)} \underset{u' \in [T-\tau,T]}{\sup} \| Dz(u') \|_{L^{\infty}} + 2 \tau^{\frac{1}{2}} C_2(n) \| g \|_{L^{\infty}([0,T] \times \Omega)}.
			\end{split}
		\end{equation*}
		Choose $\tau=\frac{1}{16 \| H \|^2_{L^{\infty}([0,T] \times \Omega)} C_2^2(n)}$. Then, taking the supremum on both sides on $[T-\tau,T]$ and over all spatial derivatives of $z$, we get
		$$\underset{t \in [T-\tau,T]}{\sup} \| Dz(t) \|_{L^{\infty}(\Omega)} \le 2 \| Dz_T(\cdot) \|_{L^{\infty}(\Omega)} + C_H \| g \|_{L^{\infty}([0,T] \times \Omega)},$$
		for some $C_H$ depending on $\| H \|_{L^{\infty}([0,T] \times \Omega)},n$. \\
		To extend the estimate for further times, we have to apply a cut-off method. Let $s > 0$ and $s < T-1$ and define
		\begin{equation*}
			\omega(t)=
			\begin{cases}
				1,& \text{ if } t \in [0,T-s-1] \\
				T-t-s,& \text{ if } t \in [T-s-1,T-s] \\
				0,& \text{ if } t \in [T-s,T] \\
			\end{cases}
			%\label{omega1}
		\end{equation*}
		The support of $\omega(T-t)$ is $[s,T]$. Define $\bar{z}(t,x)=\omega(t)z(t,x)$ and we compute the equation satisfied (a.e.) by $\bar{z}$, which is:
		\begin{equation*}
			\begin{aligned}
				-\partial_t \bar{z} &= -\omega \partial_t z - \omega' z \\
				&= \omega \Delta z + \omega \langle Dz,H \rangle + \omega g- \omega' z \\
				&= \Delta \bar{z} + \langle D\bar{z},H \rangle + \omega g - \omega' z,
			\end{aligned}
		\end{equation*}
		with boundary condition $\bar{z}(T,\cdot)=0$. Duhamel representation formula, together with the previous change of time variable and the property of the support of $\omega$, gives:
		\begin{equation*}
			\begin{split}
			\bar{z}(t,x) &= \int_s^{T-t} \int_{\Omega} \Gamma(T-t-u,x-y) \langle D\bar{z}(T-u,y),H(T-u,y) \rangle \di y \di u \\
			&+ \int_s^{T-t} \int_{\Omega} \Gamma(T-t-u,x-y) \left[ \omega(T-u)g(T-u,y) - \omega'(T-u)z(T-u,y) \right] \di y \di u = \\
			&\overset{u':=T-u}{=}\int_t^{T-s} \int_{\Omega} \Gamma(u'-t,x-y) \left[ \langle D\bar{z}(u',y),H(u',y) \rangle + \omega(u')g(u',y) - \omega'(u')z(u',y) \right] \di y \di u'.
			\end{split}
		\end{equation*}
		Computing the $i$-th spatial derivative and exploiting the properties of the convolution with respect to differentiation yields:
		$$\partial_{x_i} \bar{z}(t,x) = \int_t^{T-s} \int_{\Omega} \partial_{x_i} \Gamma(u-t,x-y) \left[ \langle D\bar{z}(u,y),H(u,y) \rangle + \omega(u)g(u,y) - \omega'(u)z(u,y) \right] \di y \di u.$$
		Take now $t \in [T-s-\tau,T-s]$, where $\tau$ is the same as before. Taking the sup norm and using the Young's inequality for convolutions, we obtain
		\begin{equation*}
			\begin{split}
			\| \partial_{x_i} \bar{z}(t,\cdot) \|_{L^{\infty}(\Omega)} &\le \int_t^{T-s} \| D\Gamma(u-t,x-\cdot) \|_{L^1(\Omega)} \left[ \| D\bar{z}(u) \|_{L^{\infty}(\Omega)} \| H \|_{L^{\infty}([0,T] \times \Omega)} + \| g(u) \|_{L^{\infty}(\Omega)} \right] \di u \\
			&+\int_t^{T-s} \| D\Gamma(u-t,x-\cdot) \|_{L^1(\Omega)} \| z(u) \|_{L^{\infty}(\Omega)} \di u \\
			&\overset{r:=u-t}{=} C_2(n) \underset{u \in [T-s-\tau,T-s]}{\sup} \| D\bar{z}(u) \|_{L^{\infty}(\Omega)} \| H \|_{L^{\infty}([0,T] \times \Omega)} \int_0^{T-s-t} r^{-\frac{1}{2}} \di r \\
			&+C_2(n) \int_0^{T-s-t} r^{-\frac{1}{2}} \| z(t+r) \|_{L^{\infty}(\Omega)} \di r + C_2(n) \int_0^{T-s-t} r^{-\frac{1}{2}} \| g(t+r) \|_{L^{\infty}(\Omega)} \di r \\
			&\le \frac{1}{2} \underset{u \in [T-s-\tau,T-s]}{\sup} \| D\bar{z}(u) \|_{L^{\infty}(\Omega)} + C_2(n) \| z \|_{L^{\infty}([0,T] \times \Omega)} \int_0^{\tau} r^{-\frac{1}{2}} \di r \\
			&+ C_2(n) \| g \|_{L^{\infty}([0,T] \times \Omega)} \int_0^{\tau} r^{-\frac{1}{2}} \di r \\
			&= \frac{1}{2} \underset{u \in [T-s-\tau,T-s]}{\sup} \| D\bar{z}(u) \|_{L^{\infty}(\Omega)} +2 \tau^{\frac{1}{2}} C_2(n) \left( \| z \|_{L^{\infty}([0,T] \times \Omega)} + \| g \|_{L^{\infty}([0,T] \times \Omega)} \right).
			\end{split}
		\end{equation*}
		Taking the supremum over all derivatives and over $[T-s-\tau,T-s]$, we get
		$$\underset{u \in [T-s-\tau,T-s]}{\sup} \| D\bar{z}(u,\cdot) \|_{L^{\infty}(\Omega)} \le C_H \left( \| g \|_{L^{\infty}([0,T] \times \Omega)} + \| z \|_{L^{\infty}([0,T] \times \Omega)} \right).$$
		Recalling the definition of $\bar{z}$, we have that
		$$\| Dz(T-s-\tau,\cdot) \|_{L^{\infty}(\Omega)} = \frac{1}{\tau} \| D\bar{z}(T-s-\tau,\cdot) \|_{L^{\infty}(\Omega)} \le \frac{C_H}{\tau} \left( \| g \|_{L^{\infty}([0,T] \times \Omega)} + \| z \|_{L^{\infty}([0,T] \times \Omega)} \right),$$
		for all $0 < s < T-1$. The same reasoning holds for $s \in [T-1,T-\tau]$, if one chooses the cut-off function
		\begin{equation*}
			\omega_2(t)=
			\begin{cases}
				T-t-s,& \text{ if } t \in [0,T-s] \\
				0,& \text{ if } t \in [T-s,T]. \\
			\end{cases}
			%\label{omega2}
		\end{equation*}
		Gathering all the estimates we obtained, we finally get~\eqref{estlemDz}.
	\end{proof}
	\bibliography{turnpike}
	\bibliographystyle{abbrv}

\bigskip

\begin{flushright}
\noindent \verb"cirant@math.unipd.it"\\
Dipartimento di Matematica ``Tullio Levi-Civita''\\ Universit\`a di Padova\\
via Trieste 63, 35121 Padova (Italy)

\medskip
\noindent \verb"debernar@math.unipd.it"\\
Dipartimento di Matematica ``Tullio Levi-Civita''\\ Universit\`a di Padova\\
via Trieste 63, 35121 Padova (Italy)
\end{flushright}

\end{document}